\documentclass[12pt,a4paper,twoside,reqno]{amsart}
\allowdisplaybreaks
\usepackage[all,cmtip]{xy}
\usepackage[T1]{fontenc}
\usepackage{amsmath,amscd}
\usepackage{amssymb,amsfonts}%,mathrsfs}
\usepackage[driver=pdftex,margin=3cm,heightrounded=true,centering]{geometry}
\usepackage{mathtools}
\usepackage{tensor}
\usepackage{url}
\usepackage[colorlinks=true,linkcolor=blue,citecolor=blue]{hyperref}
\usepackage{slashed}
\usepackage{graphicx}

\usepackage{amsthm}

\theoremstyle{plain}

\newtheorem{thm}{Theorem}[section]
\newtheorem*{thm*}{Theorem}
\newtheorem{prop}[thm]{Proposition}
\newtheorem{lemma}[thm]{Lemma}
\newtheorem{cor}[thm]{Corollary}

\newtheorem{assum}[thm]{Assumption}

\theoremstyle{definition}
\newtheorem{dfn}[thm]{Definition}

\theoremstyle{remark} % i.e. same as plain, but theorembodyfont is rm

\numberwithin{equation}{section}

\newcommand{\alpheqn}[1][\relax]{
     \refstepcounter{equation}
     \if#1\relax \relax
       \else \label{#1}
     \fi  
     \setcounter{saveeqn}{\value{equation}}%
    \setcounter{equation}{0}%
    \renewcommand{\theequation}{\thealphequation}}
\newcommand{\reseteqn}{\setcounter{equation}{\value{saveeqn}}%
     \renewcommand{\theequation}{\thearabicequation}}

\providecommand{\mathscr}{\mathcal} % a priori mathscr is mathcal
\IfFileExists{mathrsfs.sty}{
\usepackage{mathrsfs}}         %  \mathscr produces rsfs fonts
   {%  else try euler fonts
     \IfFileExists{eucal.sty}{
        \usepackage[mathscr]{eucal}} % \mathscr produces Euscript fonts,
   {% else both unavailable, do nothing, mathscr remains mathcal
   }
}

\newcommand{\Lip}{\operatorname{Lip}}
\newcommand{\Ana}{\operatorname{Ana}}

\newcommand{\vertiii}[1]{{\left\vert\kern-0.25ex\left\vert\kern-0.25ex\left\vert #1 
    \right\vert\kern-0.25ex\right\vert\kern-0.25ex\right\vert}}
\newcommand{\Bvert}[1]{{\Big\vert\kern-0.25ex\Big\vert\kern-0.25ex\Big\vert #1 
    \Big\vert\kern-0.25ex\Big\vert\kern-0.25ex\Big\vert}}
\newcommand{\bvert}[1]{{\big\vert\kern-0.25ex\big\vert\kern-0.25ex\big\vert #1 
    \big\vert\kern-0.25ex\big\vert\kern-0.25ex\big\vert}}
\newcommand{\nvert}[1]{{\vert\kern-0.25ex\vert\kern-0.25ex\vert #1 
    \vert\kern-0.25ex\vert\kern-0.25ex\vert}}

\renewcommand{\leq}{\leqslant}
\renewcommand{\geq}{\geqslant}

\newcommand{\cd}{\cdot}
\newcommand{\clc}{\cdot\ldots\cdot}
\newcommand{\ot}{\otimes}
\newcommand{\hot}{\widehat \otimes}

\newcommand{\op}{\oplus}
\newcommand{\bop}{\bigoplus}

\newcommand{\ci}{\circ}

\newcommand{\ti}{\times}

\newcommand{\cc}{\mathbb{C}}

\newcommand{\al}{\alpha}
\newcommand{\be}{\beta}
\newcommand{\ga}{\gamma}
\newcommand{\Ga}{\Gamma}
\newcommand{\de}{\delta}
\newcommand{\De}{\Delta}
\newcommand{\ep}{\varepsilon}

\newcommand{\la}{\lambda}
\newcommand{\La}{\Lambda}
\newcommand{\Na}{\nabla}
\newcommand{\om}{\omega}
\newcommand{\Om}{\Omega}
\newcommand{\si}{\sigma}

\newcommand{\te}{\theta}

\newcommand{\ze}{\zeta}
\newcommand{\pa}{\partial}
\newcommand{\da}{\dagger}

\newcommand{\ov}{\overline}
\newcommand{\C}[1]{\mathcal{#1}}
\newcommand{\G}[1]{\mathfrak{#1}}
\newcommand{\T}[1]{\textup{#1}}

\newcommand{\B}[1]{\mathbb{#1}}

\newcommand{\fork}[2]{\left\{ \begin{array}{#1} #2 \end{array} \right.}

\newcommand{\ma}[2]{\left(\begin{array}{#1} #2 \end{array} \right)}

\newcommand{\su}{\subseteq}

\newcommand{\q}{\qquad}

\newcommand{\wit}{\widetilde}

\newcommand{\inn}[1]{\langle #1 \rangle}

\newcommand{\binn}[1]{\big\langle #1 \big\rangle}

\newcommand{\sem}{\setminus}

\makeatletter
\@namedef{subjclassname@2020}{%
  \textup{2020} Mathematics Subject Classification}
\makeatother

\begin{document}
\title[Noncommutative metric geometry of quantum circle bundles]{Noncommutative metric geometry of quantum circle bundles}

%\dedicatory{Dedicated to Fritz Gesztesy on the occasion of his 70th birthday.}

\author{Jens Kaad}

\address{Department of Mathematics and Computer Science,
The University of Southern Denmark,
Campusvej 55, DK-5230 Odense M,
Denmark}

\email{kaad@imada.sdu.dk}

%\subjclass[2020]{58B34; 19K35, 46L89, 46L30} %Primary: 46L05; Secondary: 46L89, 47L65,  46L05 }
%58B32=Geometry of quantum groups
%58B34= NCG ala Connes
 %46L89= Other ÒnoncommutativeÓ mathematics based on C?-algebra theory
%46L30 =States of selfadjoint operator algebras
%81R15 taken out
%81R60= Noncommutative geometry in quantum theory


%\keywords{Quantum metric spaces, External products of spectral triples, Operator systems, Minimal tensor products.}

%MSC: primary 19K35; secondary 46H25, 58J30, 46L80

%\thanks{2010 \emph{Mathematical Subject Classification}: Primary: 58B32, %Secondary: 58B34, 33D80, 19D55, 81R50.}
%\thanks{The author was partially supported by the DFF-Research Project 2 ``Automorphisms and Invariants of Operator Algebras'', no. 7014-00145B.}
%\thanks{The second author is partially supported by the Danish National Research Foundation (DNRF) through the Centre for Symmetry and Deformation.}
%\subjclass[2010]{19K35; 58B34}
%\keywords{$KK$-theory, Unbounded $KK$-theory, Equivalence relations, Bounded transform}

\begin{abstract}%{{{
  In this paper we investigate quantum circle bundles from the point of view of compact quantum metric spaces. The raw input data is a circle action on a unital $C^*$-algebra together with a quantum metric of spectral geometric origin on the fixed point algebra. Under a few extra conditions on the spectral subspaces we show that the spectral geometric data on the base algebra can be lifted to the total algebra. Notably, the lifted spectral geometry is independent of the choice of frames and is permitted to interact with the total algebra via a twisted derivation. Under these conditions, it is explained how to assemble our data into a quantum metric on the total algebra in a way which unifies and generalizes a couple of results in the literature relating to crossed products by the integers and to quantum $SU(2)$. We apply our ideas to the higher Vaksman-Soibelman quantum spheres and endow them with quantum metrics arising from $q$-geometric data. In this context, the twist is enforced by the structure of the Drinfeld-Jimbo deformation arising from the Lie algebra of the special unitary group.  
\end{abstract}%}}}

\maketitle
\tableofcontents
%\listoffigures

\section{Introduction}
This paper is part of a larger program aiming to reconcile the theory of compact quantum groups, \cite{Dri:QG,Jim:QYB,Wor:CMP,Wor:CQG}, with noncommutative geometry in the sense of Alain Connes, \cite{Con:NCG}. This problem has been singled out in the seminal paper \cite{CoMo:TST} as one of the most challenging future research directions for advancing the field of noncommutative geometry. 

The primary focus of the present paper is the Vaksman-Soibelman quantum spheres, \cite{VaSo:AFQ}, which are homogeneous spaces for the quantum special unitary group, \cite{Wor:TGN,Wor:TKD}. In the special case where the underlying real dimension is equal to three, we emphasize that the Vaksman-Soibelman quantum sphere $S_q^3$ coincides with quantum $SU(2)$. We are however treating all the odd-dimensional quantum spheres simultaneously. In fact, our main results are formulated in the much more general context of quantum principal bundles where the fiber agrees with the unit circle and our work is therefore underpinned by the algebraic considerations in \cite{BrMa:QGG,Haj:SCQ,BeMa:QRG}.

From the point of view of noncommutative geometry, the Vaksman-Soibelman quantum spheres are difficult to access because of the presence of twists. To understand the reason for this, it suffices to look at the coproduct structure of the quantized enveloping algebra as introduced in \cite{Dri:QG,Jim:QYB}. In the case of the special unitary group the generators $E_1,\ldots,E_r$ (where $r$ is the rank) are subject to the coalgebra rule $\De(E_i) = E_i \ot K_i + K_i^{-1} \ot E_i$ where the elements $K_1,\ldots,K_r$ are group-like elements. This means that the generators $E_i$ act as twisted derivations on the coordinate algebra for quantum $SU(r+1)$ where the twist is described by the automorphism associated to $K_i$. Consequently, any reasonable attempt at describing the noncommutative geometry of the quantum spheres, using the quantized enveloping algebra as the starting point, must in one way or another incorporate twisted commutators instead of straight commutators (because $K_r$ yields a non-trivial automorphism of $\C O(S_q^{2r+1})$). This is therefore strongly suggesting the use of twisted spectral triples to obtain a sensible spectral geometric description of the quantum spheres, \cite{CoMo:TST}.

However, even for quantum $SU(2)$ there are, at least to our knowledge, no examples of (twisted) spectral triples where the Dirac operator can be described in terms of the quantized enveloping algebra, see \cite{Dab:GQS} for an overview. Notably, the approach of the papers \cite{DLSSV:DOS,NeTu:DOC} is detached from the quantized enveloping algebra in so far that the spectra of the corresponding Dirac operators are the same as the spectrum of their classical counterpart and hence remain independent of the deformation parameter. A variant of this comment also applies to the constructions in \cite{ChPa:EST}. We would moreover like to point out that the resolvent fails to be compact (in the usual sense) for the ``twisted spectral triples'' constructed in \cite{KaSe:TST}. 

In the present paper we construct what could be called a ``twisted spectral quadruple'' for each of the higher quantum spheres. The key point is here that most of the quantized enveloping algebra is kept in its deformed state except for the element $K_r$ which is, in some sense, replaced by its classical counterpart. The element $K_r$ plays a special role since it is related to the circle action which describes the quantum circle bundle structure of the quantum sphere, singling out quantum projective space as invariant elements. Replacing $K_r$ with its classical counterpart ensures that the compact resolvent condition is satisfied for our spectral data but the price is relatively high since we are sometimes forced to consider a pair of abstract Dirac operators instead of a single one. This is precisely because of the commutator condition: One of our abstract Dirac operators interacts with the coordinate algebra via straight bounded commutators whereas the other one interacts via twisted bounded commutators. Nonetheless, both of these Dirac operators are needed in order to describe the spectral geometry of the quantum sphere since only their sum has compact resolvent, not the individual factors. Our approach is therefore related to the papers \cite{KRS:RFH,KaKy:SUq2,Cac:GFC} and our work do indeed generalize parts of the results in \cite{KaKy:SUq2}. 

As already mentioned, we are not only focusing on the quantum spheres since our main results are formulated in the generality of quantum circle bundles. In particular, we provide general conditions which allow us to write down ``twisted spectral quadruples'' for quantum circle bundles. A part from the quantum spheres our methods might therefore also apply to the ``quantum Grassmann spheres'' which are quantum circle bundles over the quantum Grassmannians, \cite{TaTo:QDG,Fio:QDG,CMO:BWT}. Treating this case in detail would however lead us too far astray for the time being and we therefore leave it open for future investigations. It should also be emphasized that our results to some extend cover the results of \cite{BMR:DSS,HSWZ:STC,KaKy:DCQ} regarding crossed products by the integers as one may indeed view the corresponding algebraic crossed product as a quantum circle bundle over the algebra which carries the action of the integers.

It should be clarified that our techniques are related to and inspired by the two seminal papers, \cite{AmBa:DON,DaSi:NCB}, on Dirac operators on (noncommutative) circle bundles. There are nonetheless a couple of important differences which should be singled out. Contrary to the situation in \cite{DaSi:NCB} we do not start out with a unital spectral triple on the \emph{total algebra} of a quantum circle bundle (subject to extra conditions). Our starting point is a unital spectral triple on the \emph{base algebra} of a quantum circle bundle (subject to extra conditions) and the constructions carried out in this paper do not yield a unital spectral triple on the total algebra except in the special situation where a certain parameter is equal to one. In the general setting, as alluded to above, we obtain twisted spectral quadruples where the relevant twist is given in terms of an analytic extension of the circle action (at the $C^*$-algebraic level). For the quantum spheres the parameter just mentioned is equal to the inverse of the deformation parameter $q$ and the corresponding analytic extension of the circle action also depends on the deformation parameter $q$.

As an upshot, we apply our constructions to describe the noncommutative metric geometry of quantum circle bundles while paying particular attention to the higher quantum spheres. Let us explain in more details what we mean by this statement.

In a series of papers, \cite{Rie:MSA,Rie:MSS,Rie:GHQ}, Marc Rieffel pioneered the theory of compact quantum metric spaces with a lot of inspiration coming from the metric aspects of noncommutative geometry as described by Alain Connes in \cite{Con:CFH}. Since then, a lot of effort has been devoted to the development of the theory of compact quantum metric spaces. Notably, the work of Fr\'ed\'eric Latr\'emoli\`ere has resulted in increasingly refined methods for measuring the distance between two compact quantum metric spaces, \cite{Lat:QGP,Lat:PMS}. In recent years a highly interesting program has been pioneered by Alain Connes and Walter van Suijlekom aiming at approximating geometric data by finite dimensional spectral truncations in a systematic fashion, \cite{CoSu:STN,CoSu:TRO}.

Another line of research, which is aligned with the present paper, seeks to construct new compact quantum metric spaces, in a functorial way, out of smaller building blocks. To our knowledge, this idea was applied for the first time in \cite{BMR:DSS} and has emerged again on several occasions, see e.g. \cite{HSWZ:STC,KaKy:SUq2,AuKaKy:QMC}. Common for this approach, is an underlying decomposition of the geometry into a vertical and a horizontal direction inspired by the classical theory of fibre bundles, where the vector fields are decomposed into vertical and horizontal vector fields respecting the underlying fibre bundle structure, \cite{BeGeVe:HKD,AmBa:DON,ScWa:LST,KaSu:RSF}. The vertical direction of the geometry then resembles the geometry of the fibre whereas the horizontal direction resembles the geometry of the base space.

In noncommutative geometry, the vertical direction of the geometry is witnessed by an unbounded Kasparov module which connects the ``total algebra'' with the ``base algebra'', \cite{BaJu:TBK}. The horizontal part of the geometry is then, at least at the outset, nothing but a spectral triple on the base algebra. The correct way of assembling this data is dictated by the unbounded Kasparov product which is a refined version of the Kasparov product, \cite{Kas:OFE}, adapted to the setting of unbounded operators, \cite{Mes:UCN,MeRe:NST,KaLe:SFU}. One of the principal features of unbounded $KK$-theory is the systematic use of Hermitian connections to lift geometric data from the base algebra, yielding a description of the horizontal direction of the geometry on the total algebra. In many instances the unbounded Kasparov product provides a spectral triple on the total algebra where the relevant abstract Dirac operator is expressed as the sum of the vertical Dirac operator and the horizontal Dirac operator, see e.g. \cite{Con:GCM,KaLe:SFU,MeGo:STF,BrMeSu:GSU,KaSu:FDT,ScWa:LST}.  

From a practical point of view it is possible to formulate the following principle regarding the unbounded $KK$-theoretic approach to compact quantum metric spaces:
\medskip

\emph{If the horizontal spectral triple turns the base algebra into a compact quantum metric space and there is sufficient control on the metric properties of the vertical unbounded Kasparov module, then the spectral data provided by the unbounded Kasparov product turns the total algebra into a compact quantum metric space.}
\medskip

The wording ``sufficient control'' is on purpose left a bit imprecise because we are currently unaware of the correct notion of a \emph{bivariant} compact quantum metric space even though the paper \cite{AuKaKy:QMC} presents a recent attempt at defining such a concept. Another caveat is covered by the word ``spectral data'' in relationship with the unbounded Kasparov product. Examples pertaining to quantum groups and crossed products by non-isometric actions show that the recipe provided by the unbounded Kasparov product does not always yield a spectral triple on the total algebra, see e.g. \cite{KRS:RFH,BMR:DSS,KaKy:DCQ}. The core problem is that the horizontal Dirac operator might fail to interact with the total algebra in a way which is compatible with the bounded commutator condition of a spectral triple. Instead, it often happens that the commutator between algebra elements and the horizontal Dirac operator is displaying an exponential growth phenomenon in the vertical direction and some dampening mechanism has to be incorporated in order to tame the exponential growth.

This exponential growth phenomenon in the vertical direction is indeed what we are observing in this paper and we present a systematic way of handling it in the setting of quantum circle bundles. In particular, we describe the noncommutative metric geometry of quantum circle bundles in the presence of twists, hence providing compact quantum metric space structures in this setting. These ideas are then applied in detail to the higher Vaksman-Soibelman quantum spheres.

In the setting of the higher quantum spheres the input spectral data is the unital spectral triple on quantum projective space as constructed by Francesco D'Andrea and Ludwik D\k{a}browski in \cite{DaDa:DOQ}, relying on the important earlier paper \cite{DDL:NGQ} including Gianni Landi as a coauthor. In general, the unital spectral triple on quantum projective space is related to the work of several other noncommutative geometers as witnessed by the following list of references: \cite{KrTu:DDO,Mat:DDO,DBS:DDS}. In the lowest possible dimension, where quantum projective space agrees with the famous Podle\'s sphere, the input spectral data agrees with the unital spectral triple introduced by Ludwik D\k{a}browski and Andrzej Sitarz in \cite{DaSi:DSP}, see also \cite{NeTu:LIQ}.

The investigation of the spectral metric properties of quantum projective space was begun in \cite{AgKa:PSM} treating the Podle\'s sphere and continued in \cite{MiKa:SMQ} taking care of the general case. It is an interesting challenge to investigate the continuity properties of these compact quantum metric space structures on quantum projective space under variation of the deformation parameter $q \in (0,1]$. With respect to Marc Rieffel's quantum Gromov-Hausdorff distance, \cite{Rie:GHQ}, this continuity property has been settled for the Podle\'s sphere in \cite{AKK:PSC} but similar questions remain open for the general case. Even for the Podle\'s sphere it is not known whether the corresponding spectral metric spaces (metric spectral triples) depend continuously on $q \in (0,1]$ with respect to the more refined spectral propinquity as introduced by Fr\'ed\'eric Latr\'emoli\`ere in \cite{Lat:PMS}.     
\medskip

Let us briefly go through the structure of the present paper:

In Section \ref{s:tensor} we review some background material regarding tensor products of Hilbert $C^*$-modules with particular focus on the interior tensor product and the corresponding theory of Hermitian connections. In Section \ref{s:cqms} we go through relevant preliminaries from the theory of compact quantum metric spaces emphasizing a result of Hanfeng Li regarding actions of compact groups together with recent insights relating to finitely generated projective modules over compact quantum metric spaces. In Section \ref{s:lipschitz} we describe the various well-known constructions of a Lipschitz algebra associated with a selfadjoint unbounded operator and, because of the focus of the present paper, we also show how to construct a twisted Lipschitz algebra in the presence of an extra strongly continuous one-parameter group of unitaries. In Section \ref{s:first} we prove our first main theorem regarding quantum metrics on noncommutative circle bundles allowing for the presence of a twisted Leibniz inequality (which in practice stems from the horizontal part of the geometry). In Section \ref{s:vertical} we analyze the vertical part of the geometry in more details showing that the seminorm coming from a circle action (and the usual arc length metric) can be quantized, meaning that it can be described by taking operator norms of first order commutators with a single selfadjoint unbounded operator. In Section \ref{s:horizontal} we present a thorough analysis of lifts of selfadjoint unbounded operators (along Hermitian connections) in the setting of quantum circle bundles. In particular, we provide general conditions ensuring that such lifts interact with the total algebra via twisted bounded commutators. In Section \ref{s:second} our second main theorem is stated and proved, showing that the vertical and the horizontal part of the geometry can be combined into a quantum metric on the total algebra. Notably, this is possible in spite of the unequal footing of the two parts of the geometry as witnessed by the discrepancy between bounded commutators and twisted bounded commutators. In Section \ref{s:unbKK} we relate our constructions to considerations in unbounded $KK$-theory emphasizing the geometric ideas behind our approach to noncommutative metric geometry. In the final Section \ref{s:quasph}, our main theorems are applied to the higher Vaksman-Soibelman spheres. In particular, we endow these quantum homogeneous spaces with compact quantum metric space structures which relate directly to the underlying $q$-geometry.

%The core idea is to assemble the geometry in the fiber direction (which ressembles the geometry of the unit circle) with the geometry of the base space into a geometry on the total space. The geometry of the fiber is referred to as the vertical geometry and the geometry of the base space is referred to as the horizontal geometry. In the present noncommutative setting, the total space is in fact a unital $*$-algebra $\C A$ carrying an action of the unit circle and the base space is the unital $*$-subalgebra consisting of all the invariant elements $\C A_0$. 
%
%In the specific context of the higher Vaksman-Soibelman spheres the total algebra would be the coordinate algebra for the quantum sphere (of some fixed odd dimension) and the base algebra would be the coordinate algebra for the corresponding quantum projective space. 
%
%The topological aspects of our setting are witnessed by a $C^*$-norm on the total algebra $\C A$ 

\subsection{Acknowledgements} 
The author gratefully acknowledge the financial support from the Independent Research Fund Denmark through grant no.~9040-00107B and 1026-00371B. This research is part of the EU Staff Exchange project 101086394 "Operator Algebras That One Can See''.

The author would also like to thank the anonymous referee for pointing out the two important papers \cite{AmBa:DON,DaSi:NCB} which we failed to reference in a first version of this paper. % and 7014-00145B. 

\subsection{Statements and Declarations}
\subsubsection{Conflict of interest} The author declares that there is no conflict of interest.
\subsubsection{Data availability} This work uses no additional data.

\subsection{Standing conventions}
For a Hilbert space $H$ we use the notation $\B L(H)$ for the unital $C^*$-algebra of bounded operators on $H$ and the operator norm is denoted by $\| \cd \|_\infty : \B L(H) \to [0,\infty)$. Similarly, for a right Hilbert $C^*$-module $X$ over a $C^*$-algebra $A$, we let $\B L(X)$ refer to the unital $C^*$-algebra of bounded adjointable operators on $X$ and $\| \cd \|_\infty : \B L(X) \to [0,\infty)$ is the operator norm. %All Hilbert $C^*$-modules appearing in this text are assumed to be \emph{countably generated}

    \section{Preliminaries on tensor products of $C^*$-modules}\label{s:tensor}
An important part of the present text uses the interior tensor product of Hilbert $C^*$-modules and we therefore start out by reviewing the main constructions regarding this tensor product following the standard source \cite{Lan:HCM}. 
    
Let $X$ and $Y$ be right Hilbert $C^*$-modules over $C^*$-algebras $A$ and $B$, respectively. Moreover, let $\rho : A \to \B L(Y)$ be a $*$-homomorphism so that $Y$ becomes a $C^*$-correspondence from $A$ to $B$. We equip the module tensor product $X \ot_A Y$ with the $B$-valued inner product given by the expression
\[
\inn{x_1 \ot y_1, x_2 \ot y_2} := \binn{y_1, \rho( \inn{x_1,x_2}) y_2}
\]
on simple tensors. The corresponding Hilbert $C^*$-module completion of $X \ot_A Y$ is denoted by $X \hot_\rho Y$ and referred to as the \emph{interior tensor product}. We emphasize that the canonical map $X \ot_A Y \to X \hot_\rho Y$ is injective, see \cite[Proposition 4.5]{Lan:HCM}.
%
%In the case where $X$ and $T$ 
For each $x \in X$ we have a bounded adjointable operator $\te_x : Y \to X \hot_\rho Y$ given by the formula
$\te_x(y) = x \ot y$ and a straightforward computation yields the explicit formula $\te_x^*(x_1 \ot y_1) = \rho(\inn{x,x_1})y_1$ for the adjoint. Moreover, if $T : X \to X$ is a bounded adjointable operator, then we have an induced bounded adjointable operator $T \hot 1 : X \hot_\rho Y \to X \hot_\rho Y$ which agrees with $T \ot 1$ on the module tensor product $X \ot_A Y \su X \hot_\rho Y$, see \cite[Lemma 4.2]{Lan:HCM}. In particular, if $C$ is a third $C^*$-algebra which acts on $X$ via a $*$-homomorphism $\psi : C \to \B L(X)$, then there is an induced $*$-homomorphism
\[
\psi \hot 1 : C \to \B L(X \hot_\rho Y) \q (\psi \hot 1)(c) := \psi(c) \hot 1 . 
\]
Let us clarify that if $\rho : A \to \B L(Y)$ is injective, then $\| T \hot 1 \|_\infty = \| T \|_\infty$ for all $T \in \B L(X)$. 

On some occasions, we view the bounded adjointable operators $\B L(Y)$ as a $C^*$-correspondence from $A$ to $\B L(Y)$ where the left action comes from the $*$-homomorphism $\rho : A \to \B L(Y)$. Notice in this respect that an arbitrary $C^*$-algebra $C$ can be viewed as a right Hilbert $C^*$-module over itself with inner product $\inn{c_1,c_2} := c_1^* c_2$. For a vector $y \in Y$ we have the evaluation map $\T{ev}_y : X \hot_\rho \B L(Y) \to X \hot_\rho Y$ satisfying that $\T{ev}_y(x \ot T) := x \ot T(y)$. In fact, applying these evaluation maps for various vectors in $Y$ yields a unitary isomorphism of right Hilbert $C^*$-modules
\[
\T{ev} : ( X \hot_\rho \B L(Y)) \hot_{\T{id}} Y \to X \hot_\rho Y ,
\]
where $\T{id} : \B L(Y) \to \B L(Y)$ denotes the identity map. 

%\[
%\begin{split}
%\inn{ x_1 \ot T_1(y_1), x_2 \ot T_2(y_2) } & = \inn{ T_1(y), \inn{x_1,x_2} T_2(y)} = \inn{y, T_1^*\inn{x_1,x_2} T_2(y) } \\
%& = \inn{y, \inn{x_1 \ot T_1, x_2 \ot T_2}(y)}
%\end{split}
%\]

Let us fix a norm-dense $*$-subalgebra $\C A \su A$ and a norm-dense vector subspace $\C X \su X$ satisfying that $\C X \cd \C A \su \C X$ and $\inn{\C X,\C X} \su \C A$.

\begin{dfn}\label{d:frame}
  We say that $\C X$ has a \emph{finite frame} if there exist finitely many vectors $\ze_1,\ldots,\ze_n \in \C X$ such that $\sum_{j = 1}^n \ze_j \cd \inn{\ze_j,x} = x$ for all $x \in \C X$.
\end{dfn}

  The existence of a finite frame implies that $\C X$ is finitely generated projective as a right module over $\C A$. In fact, $\C X$ can be identified with a direct summand in the free module $\C A^{\op n}$ via the $\C A$-linear map $\Phi : \C X \to \C A^{\op n}$ given by $\Phi(x) := \{ \inn{\ze_j,x} \}_{j = 1}^n$. The corresponding left inverse $\Phi^* : \C A^{\op n} \to \C X$ is given explicitly by the formula $\Phi^*( \{ a_j\}_{j = 1}^n) = \sum_{j = 1}^n \ze_j \cd a_j$.  

Let $X'$ be an extra right Hilbert $C^*$-module over $A$. The following important result is certainly well-known but for lack of a precise reference we present the short proof here.

\begin{lemma}\label{l:extension}
  Suppose that $\C X \su X$ has a finite frame $\{ \ze_j\}_{j = 1}^n$. It then holds that the inclusion maps induce an isomorphism of right $B$-modules $\C X \ot_{\C A} Y \cong X \hot_\rho Y$. Moreover, if $\C R : \C X \to X' \hot_\rho \B L(Y)$ is a $\cc$-linear map satisfying that
\begin{equation}\label{eq:linear}
\C R(x \cd a) = \C R(x) \cd \rho(a) \q \mbox{for all } x \in \C X \, \, \mbox{and} \, \, \, a \in \C A ,
\end{equation}
then there exists a unique $B$-linear bounded operator $c(R) : X \hot_\rho Y \to X' \hot_\rho Y$ such that $c(R)(x \ot y) = \T{ev}_y( \C R(x) )$ for all $x \in \C X$ and $y \in Y$.
%\, \, \mbox{and} \, \, \,
%  \te_{x_1}^* R(x_2) = ( \te_{x_2}^* R(x_1) )^*  for all $x,x_1,x_2 \in \C X$ and $a \in \C A$.
  \end{lemma}
\begin{proof}
The inverse of the $B$-linear map $\C X \ot_{\C A} Y \to X \hot_\rho Y$ is given explicitly in terms of the finite frame by $z \mapsto \sum_{j = 1}^n \ze_j \ot \te_{\ze_j}^*(z)$.

The condition in \eqref{eq:linear} on the $\cc$-linear map $\C R$ ensures that we have a well-defined $B$-linear map $c(R) : \C X \ot_{\C A} Y \to X' \hot_\rho Y$ given by $c(R)(x \ot y) = \T{ev}_y( \C R(x))$ on simple tensors. Using our finite frame again, we obtain the result of the lemma by letting $z \in \C X \ot_{\C A} Y$ and record the estimates
  \[
  \begin{split}
    \| c(R)(z) \| & = \big\| \sum_{j = 1}^n \T{ev}_{ \te_{\ze_j}^*(z)} \C R(\ze_j) \big\|
    \leq \sum_{j = 1}^n \| \te_{\ze_j}^*(z) \| \cd \| \C R(\ze_j) \| \\
    & \leq \| z \| \cd \sum_{j = 1}^n \| \ze_j \| \cd \| \C R(\ze_j) \| . \qedhere
  \end{split}
  \]
\end{proof}
%
%Using our finite frame one more time we get that
%  \begin{equation}\label{eq:clifford}
%  \begin{split}
%    c(\C R)(x \ot y) = \sum_{j = 1}^n c(\C R)(\ze_j \ot \rho(\inn{\ze_j,x}) y ) 
%     = \sum_{j = 1}^n \T{ev}_{ \te_{\ze_j}^*(x \ot y)} \C R(\ze_j)
%  \end{split}
%  \end{equation}
%  for all $x \in \C X$ and $y \in Y$. 
%  
%  \[
%  \inn{ c(R)(x_1 \ot y_1), x_2 \ot y_2 } = \inn{ \T{ev}_{y_1} R(x_1), x_2 \ot y_2} = \inn{\te_{x_2}^* R(x_1)(y_1),y_2}
%  = \inn{y_1, \te_{x_1}^* R(x_2)(y_2)} = \inn{y_1, c(R)(x_2 \ot y_2)}
%  \]

  \subsection{Horizontal lifts of selfadjoint unbounded operators}\label{ss:lifts}
We are now going to review how to transport a selfadjoint unbounded operator from one Hilbert space to another via a Hermitian connection. The material presented here form an important part of the core constructions in unbounded $KK$-theory \cite{Mes:UCN,KaLe:SFU,MeRe:NST}, but the main idea can also be found in the work of Alain Connes in the context of finitely generated projective modules, \cite{Con:GCM}.

Let $X$ be a countably generated right Hilbert $C^*$-module over a unital $C^*$-algebra $A$ and let $H$ be a separable Hilbert space. Let moreover $\rho : A \to \B L(H)$ be an injective unital $*$-homomorphism and let $D : \T{Dom}(D) \to H$ be a selfadjoint unbounded operator. Suppose that $\C A \su A$ is a norm-dense unital $*$-subalgebra such that for every $a \in \C A$ it holds that $\rho(a)( \T{Dom}(D) ) \su \T{Dom}(D)$ and the commutator $[D,\rho(a)] : \T{Dom}(D) \to H$ extends to a bounded operator $\de(a) : H \to H$. The associated $\cc$-linear map $\de : \C A \to \B L(H)$ is a $*$-derivation in so far that the two equalities
\[
\de(a \cd b) = \rho(a) \cd \de(b) + \de(a) \cd \rho(b) \, \, \T{ and } \, \, \, \de(a^*) = - \de(a)^*,
\]
are valid for all $a,b \in \C A$. %It moreover holds that $\de$ is \emph{closable} so that if $\{a_n\}_{n = 1}^\infty$ is a sequence in $\C A$ which converges to $0$ (in $C^*$-norm) and $\{ \de(a_n) \}_{n = 1}^\infty$ converges (in operator norm) to some element $T \in \B L(H)$, then $T = 0$.

Let us specify that our data $(\C A,H,D)$ is a \emph{unital spectral triple} on $A$, if we impose the extra requirement that the resolvent $(i + D)^{-1}$ is a compact operator on $H$. Without this extra compactness conditions we refer to the triple $(\C A,H,D)$ as a \emph{Lipschitz triple} on $A$.

It is sometimes relevant to incorporate a grading operator which is a fixed selfadjoint unitary operator $\ga : H \to H$ which commutes with $\rho(a)$ for all $a \in A$ and anti-commutes with $D : \T{Dom}(D) \to H$. In this situation we say that our Lipschitz triple $(\C A,H,D)$ is \emph{graded}. The grading operator $\ga$ then induces a selfadjoint unitary operator $1 \hot \ga : X \hot_\rho H \to X \hot_\rho H$ which on simple tensors is given by the formula $(1 \hot \ga)(x \ot \xi) = x \ot \ga \xi$.

We emphasize that our conditions imply that the interior tensor product $X \hot_\rho H$ is a separable Hilbert space. In this respect, recall that $X$ is assumed to be countably generated so that there exists a sequence $\{\xi_n\}_{n = 1}^\infty$ in $X$ such that the corresponding $A$-linear span is norm-dense in $X$.

 %The notion of a connection appears in many places and can at least be found in \cite{} even though the extra condition of being Hermitian does not seem to appear in there. The relationship with Hilbert $C^*$-modules was only uncovered at a later stage.

\begin{dfn}\label{d:hermit}
Suppose that $\C X \su X$ is a norm-dense vector subspace satisfying that $\C X \cd \C A \su \C X$ and $\inn{\C X,\C X} \su \C A$. A \emph{Hermitian $\de$-connection} is a $\cc$-linear map $\Na : \C X \to X \hot_\rho \B L(H)$ satisfying that
  \begin{enumerate}
  \item $\Na(x \cd a) = \Na(x) \cd \rho(a) + x \ot \de(a)$ for all $x \in \C X$ and $a \in \C A$;
  \item $\de( \inn{x_1,x_2} ) = \te_{x_1}^* \Na(x_2) - \big( \te_{x_2}^* \Na(x_1) \big)^*$ for all $x_1,x_2 \in \C X$.
  \end{enumerate}
  In the case where the Lipschitz triple $(\C A,H,D)$ is graded with grading operator $\ga : H \to H$, we say that a Hermitian $\de$-connection $\Na : \C X \to X \hot_\rho \B L(H)$ is \emph{odd}, if it holds that
  \[
(1 \hot \ga)\T{ev}_\xi \Na(x) = - \T{ev}_{\ga \xi} \Na(x) \q \mbox{for all } x \in \C X \mbox{ and } \xi \in H .
  \]
\end{dfn}

It is common to put more restrictions on Hermitian $\de$-connections, so let us briefly explain how this works. Consider the $\C A$-$\C A$-bisubmodule $\Om^1_D(\C A)$ of $\B L(H)$ defined by putting
\[
\Om^1_D(\C A) := \T{span}_{\B C} \big\{ \rho(a) \de(b) \mid a \in \C A , b \in \C B \big\} .
\]
It is then often required that a Hermitian $\de$-connection factorizes through the module tensor product $\C X \ot_{\C A} \Om^1_D(\C A)$ (which we may map to $X \hot_\rho \B L(H)$ via the canonical inclusions), see \cite[Chapter VI, Definition 8]{Con:NCG} and \cite{Sui:NCG}. This additional property is indeed satisfied by the concrete Hermitian $\de$-connections we are considering in the present paper.

In the generality of Definition \ref{d:hermit} it can be proved that Hermitian $\de$-connections always exist under very mild conditions on $X$ and the norm-dense unital $*$-subalgebra $\C A \su A$, see \cite[Theorem 3.1 and Theorem 4.1]{Kaa:DAH}. In fact, it suffices to require the existence of a sequence $\{ \xi_n \}_{n = 1}^\infty$ in $X$ satisfying that the $\C A$-submodule $\T{span}_{\B C}\big\{ \xi_n \cd a \mid n \in \B N \, , \, \, a \in \C A \big\} \su X$ is norm-dense and that $\inn{\xi_n,\xi_m} \in \C A$ for all $n,m \in \B N$.
\medskip

Suppose from now on that we have a fixed norm-dense vector subspace $\C X \su X$ satisfying that $\C X \cd \C A \su \C X$ and $\inn{\C X,\C X} \su \C A$. To ease the exposition we moreover suppose that the inclusions induce an injective map from $\C X \ot_{\C A} \T{Dom}(D)$ to the interior tensor product $X \hot_\rho H$.

\begin{dfn}
  Suppose that $\Na : \C X \to X \hot_\rho \B L(H)$ is a Hermitian $\de$-connection. The \emph{horizontal lift} of the selfadjoint unbounded operator $D$ via $\Na$ is the symmetric unbounded operator
  \[
  1 \ot_\Na D : \C X \ot_{\C A} \T{Dom}(D) \to X \hot_\rho H
  \]
  defined by the formula $(1 \ot_\Na D)(x \ot \xi) := \T{ev}_\xi \Na(x) + x \ot D(\xi)$ for all $x \in \C X$ and $\xi \in \T{Dom}(D)$. The closure of the horizontal lift is denoted by $1 \hot_\Na D$. 
\end{dfn}

Let us clarify that if the Lipschitz triple is graded $(\C A,H,D)$ and we have an odd Hermitian $\de$-connection $\Na : \C X \to X \hot_\rho \B L(H)$, then the associated horizontal lift is also odd in the sense that $1 \ot_\Na D$ anti-commutes with $1 \hot \ga$ (the same therefore also holds for the closure $1 \hot_\Na D$).

It is important to realize that the horizontal lift need not be essentially selfadjoint and can in general even fail to have a selfadjoint extension, see \cite[Subsection 5.1]{Kaa:DAH} for a basic example illustrating this phenomenon. For this reason we are in the rest of this section focusing on the simple case where $\C X \su X$ has a finite frame $\{ \ze_j\}_{j = 1}^n$, see Definition \ref{d:frame}. 
%
%meaning that there exist finitely many elements $\ze_1,\ldots,\ze_n \in \C X$ such that
%\[
%\sum_{j = 1}^n \ze_j \cd \inn{\ze_j,x} = x \q \T{for all } x \in \C X .
%\]
%This assumption implies that $\C X$ is finitely generated projective as a right module over $\C A$ and we single out the corresponding projection $P_{\ze} \in M_n(\C A)$ which is determined by the entries $(P_\ze)_{ij} := \inn{\ze_i,\ze_j}$.

The finite frame $\{\ze_j\}_{j = 1}^n$ may be applied to write down the \emph{Grassmann connection} $\Na_{\T{Gr}} : \C X \to X \hot_\rho \B L(H)$ given by the formula
\begin{equation}\label{eq:grass}
\Na_{\T{Gr}}(x) := \sum_{j = 1}^n \ze_j \ot \de( \inn{\ze_j,x}) .
\end{equation}
It can be verified that the Grassmann connection is indeed a Hermitian $\de$-connection and the corresponding horizontal lift is determined by the expression
\begin{equation}\label{eq:horilift}
(1 \ot_{\Na_{\T{Gr}}} D)(x \ot \xi) := \sum_{j = 1}^n \ze_j \ot D\big( \rho(\inn{\ze_j,x})(\xi) \big)
\end{equation}
for all $x \in \C X$ and $\xi \in \T{Dom}(D)$. Notice in this respect that the existence of our finite frame implies that the module tensor product $\C X \ot_{\C A} \T{Dom}(D)$ does indeed inject into the interior tensor product $X \hot_\rho H$. In the case where the Lipschitz triple $(\C A,H,D)$ is graded it automatically holds that the Grassmann connection $\Na_{\T{Gr}}$ is odd.

Let us denote the $\C A$-linear endomorphisms of $\C X$ by $\T{End}_{\C A}(\C X)$ and record that the existence of the finite frame for $\C X$ also implies that $\T{End}_{\C A}(\C X)$ can be viewed as a norm-dense unital $*$-subalgebra of $\B L(X)$. In fact, the adjoint of an element $T \in \T{End}_{\C A}(\C X)$ is given explicitly by $T^*(x) = \sum_{i = 1}^n \ze_i \cd \inn{T \ze_i,x}$ for all $x \in X$. Since the unital $*$-homomorphism $\rho : A \to \B L(H)$ is assumed to be injective we get an injective unital $*$-homomorphism $\B L(X) \to \B L(X \hot_\rho H)$ given by $T \mapsto T \hot 1$.
%
%\[
%\inn{T^*x,y} = \sum_{i = 1}^n \inn{x,T \ze_i} \inn{\ze_i,y} = \sum_{i = 1}^n \inn{x,Ty}
%\]
%The adjoint of $T \in \T{End}_{}$

The result here below is due to Alain Connes, see \cite{Con:GCM}, but a more detailed proof can be found in \cite[Theorem 6.15]{Sui:NCG}. Regarding selfadjointness of the horizontal lift $1 \ot_\Na D$ it is important to clarify that the domain $\C X \ot_{\C A} \T{Dom}(D)$ can be identified with the norm-dense subspace $P \T{Dom}(D)^{\op n} \su P H^{\op n}$, where $P \in M_n(\B L(H) )$ is the orthogonal projection determined by the entries $P_{ij} := \rho( \inn{\ze_i,\ze_j})$.

\begin{thm}\label{t:esself}
  Suppose that $\{ \ze_j\}_{j = 1}^n$ is a finite frame for $\C X \su X$. If $\Na : \C X \to X \hot_\rho \B L(H)$ is a Hermitian $\de$-connection, then the horizontal lift $1 \ot_\Na D : \C X \ot_{\C A} \T{Dom}(D) \to X \hot_\rho H$ is a selfadjoint unbounded operator.
Moreover, if $(\C A,H,D)$ is a unital spectral triple, then $( \T{End}_{\C A}(\C X), X \hot_\rho H, 1 \ot_\Na D)$ is also a unital spectral triple.
\end{thm}

\section{Preliminaries on compact quantum metric spaces}\label{s:cqms}
Our aim is now to review the main definitions relating to compact quantum metric spaces, which is a theory initiated and developed by Marc Rieffel in a series of papers, \cite{Rie:MSA,Rie:MSS,Rie:GHQ}. We are here focusing on the version of the theory pertaining to operator systems in line with the recent developments in \cite{CoSu:TRO,CoSu:STN,Sui:GSS}, but see also \cite{Ker:MQG} for earlier work in this direction.

\begin{dfn}
  A concrete \emph{operator system} is a vector subspace $\C E$ of a unital $C^*$-algebra $A$ satisfying that $1 \in \C E$ and that $x^* \in \C E$ for all $x \in \C E$. An operator system $\C E \su A$ is \emph{complete}, if $\C E$ is closed in the $C^*$-norm on $A$.
\end{dfn}

For an operator system $\C E \su A$ we are interested in the state space $S(\C E)$ which consists of all the unital positive functionals $\mu : \C E \to \B C$. Notice in this respect that an element $\xi \in \C E$ is \emph{positive} if it is positive as an element of the unital $C^*$-algebra $A$. The state space $S(\C E)$ is equipped with the weak$^*$-topology and it follows from the Banach-Alaoglu theorem that $S(\C E)$ is a compact Hausdorff space. We equip $\C E$ with the norm-topology coming from the unique $C^*$-norm on $A$.

\begin{dfn}\label{d:slip}
A \emph{slip-norm} on an operator system $\C E \su A$ is a semi-norm $L : \C E \to [0,\infty)$ satisfying that $L(1) = 0$ and $L(x^*) = L(x)$ for all $x \in \C E$. A slip-norm $L : \C E \to [0,\infty)$ is \emph{lower semicontinuous}, if for every $y \geq 0$ the subset $U_y := \big\{ x \in \C E \mid L(x) > y \big\}$ is open in $\C E$. % (equipped with the subspace topology inherited from $A$). 
\end{dfn}

The theory of compact quantum metric spaces is concerned with metrics on the state space $S(\C E)$ arising from slip-norms in the following sense:

\begin{dfn}
  Let $L : \C E \to [0,\infty)$ be a slip-norm on an operator system $\C E \su A$. The \emph{Monge-Kantorovich (extended) metric} $d_L$ is defined by
    \[
d_L(\mu,\nu) := \sup \big\{ | \mu(x) - \nu(x)| \mid L(x) \leq 1 \big\} 
\]
for all $\mu, \nu \in S(\C E)$.
\end{dfn}

The Monge-Kantorovich extended metric is indeed an extended metric $d_L : S(\C E) \ti S(\C E) \to [0,\infty]$ meaning that it satisfies all the usual properties of a metric except that states are allowed to be infinitely far from one another. In particular, we may define the associated metric topology on $S(\C E)$ with basis consisting of all the metric open balls with finite radius.

\begin{dfn}\label{d:cqms}
We say that a pair $(\C E,L)$, consisting of an operator system $\C E \su A$ and a slip-norm $L : \C E \to [0,\infty)$, is a compact quantum metric space, if the weak$^*$-topology on the state space $S(\C E)$ agrees with the metric topology coming from the Monge-Kantorovich extended metric. \end{dfn}

The following theorem, established by Marc Rieffel as \cite[Theorem 1.8]{Rie:MSA}, provides a characterization of compact quantum metric spaces which allows us to work directly with the pair $(\C E,L)$ instead of the Monge-Kantorovich extended metric on the state space. Let us identify $\B C$ with the closed subspace of $\C E$ spanned by the unit $1 \in \C E$ and consider the quotient vector space $\C E/\B C$ together with the quotient map $[\cd] : \C E \to \C E/\B C$.

\begin{thm}\label{t:charac}
  Let $\C E \su A$ be an operator system equipped with a slip-norm $L : \C E \to [0,\infty)$. It holds that $(\C E,L)$ is a compact quantum metric space if and only if the subset
    \[
\big\{ [x] \in \C E/\B C \mid L(x) \leq 1 \big\} 
    \]
    is totally bounded (in the norm-topology on the quotient space $\C E/ \B C$).
\end{thm}

It is also possible to obtain a characterization of compact quantum metric spaces in terms of finite dimensional approximations and the interested reader can consult \cite[Section 3]{Kaa:ExPr} for more details on this result.

%{\red A positive unital map need not be a contraction. We are sometimes omitting the subscript $n$ when applying a map entrywise.}

\subsection{Li's theorem for circle actions}\label{ss:li}
In this subsection we go through a special version of a theorem due to Hanfeng Li, \cite[Theorem 4.1]{Li:DCM}. Li's theorem treats the general case of compact groups acting on unital $C^*$-algebras but for our purposes it suffices to consider the special setting where the compact group is equal to the unit circle.

Let us fix a unital $C^*$-algebra $A$ and equip $A$ with a strongly continuous action $\si$ of the unit circle $S^1 \su \B C$.

For each $n \in \B Z$, let $A_n \su A$ denote the \emph{spectral subspace} defined by
\[
A_n := \big\{ a \in A \mid \si_\la(a) = \la^n a \, \, \T{for all } \la \in S^1 \big\} .
\]
The corresponding spectral projection $P_n : A \to A$ with image $P_n(A) = A_n$ is given by the Riemann integral
\begin{equation}\label{eq:specproj}
P_n(a) = \frac{1}{2\pi}\int_0^{2 \pi} e^{-itn} \cd \si_{e^{it}}(a) \, dt \q \T{for all } a \in A .
\end{equation}
Notice that the \emph{fixed point algebra} $A_0 \su A$ is a unital $C^*$-subalgebra and that $P_0$ is a faithful conditional expectation.

On the unit circle, we define the continuous function $\ell : S^1 \to [0,\pi]$ by putting $\ell(e^{it}) = |t|$ for all $t \in (-\pi,\pi]$. In this way, the arc length metric is given by the formula $d(z,w) := \ell( z^{-1} w)$ for all $z,w \in S^1$.

    Associated with the above data, we have the unital $*$-subalgebra $\Lip_{S^1}(A) \su A$ such that an element $a \in A$ belongs to $\Lip_{S^1}(A)$ if and only if the subset
    \[
\big\{ (\si_\la(a) - a)/\ell(\la) \mid \la \in S^1 \sem \{1\} \big\} \su A
\]
is bounded in operator norm. It can be verified that $\Lip_{S^1}(A) \su A$ is norm-dense (see e.g. \cite[Proposition 2.5]{Ex:CAP}) and we equip it with the lower semicontinuous slip-norm $L_\ell : \Lip_{S^1}(A) \to [0,\infty)$ defined by
\[
L_\ell(a) := \sup_{\la \in S^1 \sem \{1\}} \| \si_\la(a) - a \|/\ell(\la) .
\]
Remark that $L_\ell( \si_\mu(a)) = L_\ell(a)$ for all $\mu \in S^1$ and $a \in \Lip_{S^1}(A)$. In particular, it holds that $\si_\mu\big( \Lip_{S^1}(A) \big) \su \Lip_{S^1}(A)$. It is also relevant to notice $A_n \su \Lip_{S^1}(A)$ for all $n \in \B Z$ and that
\begin{equation}\label{eq:supspecsub}
L_\ell(a) = \| a \| \cd \sup_{t \in (-\pi,\pi] \sem \{0\}} | e^{itn} - 1|/|t| = \| a \| \cd |n| \q \T{for all } a \in A_n .
\end{equation}

We are now ready to state Li's theorem for circle actions, \cite[Theorem 4.1]{Li:DCM}. Notice in this respect that Marc Rieffel proved a predecessor to Li's theorem focusing on ergodic actions of compact groups, see \cite[Theorem 2.3]{Rie:MSA}. 

\begin{thm}\label{t:licircle}
  Let $\C A \su A$ be a norm-dense unital $*$-subalgebra equipped with a lower semicontinuous slip-norm $L : \C A \to [0,\infty)$. Suppose that the following conditions are satisfied:
    \begin{enumerate}
\item $\C A \su \Lip_{S^1}(A)$ and there exists a constant $C \geq 0$ such that $L_\ell(a) \leq C \cd L(a)$ for all $a \in \C A$;
\item The inclusions $\si_\la(\C A) \su \C A$ and $P_n(\C A) \su \C A$ hold for all $\la \in S^1$ and $n \in \B Z$. Moreover, we require that $L(\si_\la(a)) = L(a)$ for all $\la \in S^1$ and $a \in \C A$; 
\item For every $n \in \B Z \sem \{0\}$, the subset $\big\{ a \in \C A \cap A_n \mid L(a) \leq 1 \big\} \su A$ is totally bounded; % and if $L(a) = 0$ for some $a \in \C A \cap A_n$, then $a = 0$;
\item The pair $(\C A \cap A_0, L)$ is a compact quantum metric space.
    \end{enumerate}
Then it holds that $(\C A,L)$ is a compact quantum metric space.    
\end{thm}

On top of the above theorem, we record the following convenient result, see \cite[Lemma 4.3]{Li:DCM}.

\begin{lemma}\label{l:estimate}
  Let $\C A \su A$ be a norm-dense unital $*$-subalgebra equipped with a lower semicontinuous slip-norm $L : \C A \to [0,\infty)$. If condition $(2)$ from Theorem \ref{t:licircle} is satisfied, then we have the estimate
    \[
L( P_n(a)) \leq L(a) \q \mbox{for all } n \in \B Z \, \, \mbox{ and } \, \, \, a \in \C A.
    \]
\end{lemma}

%{\blu remember to think about whether $A_n \cap \C A \su A_n$ is automatically norm-dense}

\subsection{Finitely generated projective modules}
In order to apply Li's theorem for circle actions, the difficult part is often to verify condition $(3)$ and $(4)$. We are therefore interested in finding manageable criteria which make condition $(3)$ a consequence of condition $(4)$. In many cases of interest, the spectral subspaces $A_n$ (for $n \in \B Z$) are finitely generated projective modules over the fixed point algebra $A_0$ and we therefore study this setting in more details from the point of view of compact quantum metric spaces. The discussion here below is however not limited to the case of circle actions considered in Subsection \ref{ss:li}.

Let us fix a unital $C^*$-algebra $A_0$ equipped with a norm-dense unital $*$-subalgebra $\C A_0 \su A_0$. Let us moreover consider a right Hilbert $C^*$-module $X$ over $A_0$ together with a norm-dense vector subspace $\C X$ satisfying that $\C X \cd \C A_0 \su \C X$ and $\inn{\C X,\C X} \su \C A_0$.

We let $L : \C X \to [0,\infty)$ be a seminorm and $L_0 : \C A_0 \to [0,\infty)$ be a slip-norm. The norm on $X$ coming from the $A_0$-valued inner product is denoted by $\| \cd \| : X \to [0,\infty)$. The following theorem is a variation over \cite[Theorem 2.20]{KaKy:SUq2}. 

\begin{thm}\label{t:finproj}
   Suppose that $(\C A_0,L_0)$ is a compact quantum metric space and that $\C X$ has a finite frame $\{ \ze_j\}_{j = 1}^n$ satisfying the conditions:
    \begin{enumerate}
    %\item There exists a constant $C \geq 0$ such that $L( P_0(a) ) \leq C \cd L(a)$ for all $a \in \C A$;
    %\item $( \C A_0, L_0 )$ is a compact quantum metric space;
    \item There exists a constant $C \geq 0$ such that $\| x \| \leq C \cd L(x)$ for all $x \in \C X$;
    \item There exists a constant $C_0 \geq 0$ such that $L_0( \inn{\ze_j,x} ) \leq C_0 \cd \big(L(x) + \| x \|\big)$ for all $x \in \C X$ and $j \in \{1,\ldots,n\}$. 
    \end{enumerate}
    Then the subset $\big\{ x \in \C X \mid L(x) \leq 1 \big\} \su X$ is totally bounded. 
\end{thm}
\begin{proof}
  We view the direct sum $A_0^{\op n}$ as a right Hilbert $C^*$-module over $A_0$ with inner product $\binn{ \{a_j\}_{j = 1}^n , \{b_j \}_{j = 1}^n} := \sum_{j = 1}^n a_j^* b_j$. The $\C A_0$-linear map $\Phi : \C X \to \C A_0^{\op n}$ given by $\Phi(x) := \{ \inn{\ze_j,x} \}_{j = 1}^n$ then extends to an bounded adjointable isometry $\Phi : X \to A_0^{\op n}$ and the adjoint $\Phi^* : A_0^{\op n} \to X$ restricts to an $\C A_0$-linear map $\Phi^* : \C A_0^{\op n} \to \C X$ given explicitly by the formula $\Phi^*( \{a_j \}_{j = 1}^n) = \sum_{j = 1}^n \ze_j \cd a_j$.

  We equip the direct sum $\C A_0^{\op n}$ with the seminorm $\wit{L}_0( \{a_j\}_{j = 1}^n) := \max\big\{ L_0(a_j) \mid j = 1,\ldots,n\big\}$ and record that condition $(1)$ and $(2)$ imply the estimates
  \begin{equation}\label{eq:tilde}
\wit{L}_0( \Phi(x) ) \leq C_0 \cd \big( L(x) + \| x \| \big) \leq C_0 (1 + C) \cd L(x) \q \T{for all } x \in \C X .
\end{equation}

Since $(\C A_0,L_0)$ is assumed to be a compact quantum metric space we obtain from Theorem \ref{t:charac} that the subset
  \[
D_r := \big\{ a \in \C A_0^{\op n} \mid \wit{L}_0(a) \leq r \, \, \T{and} \, \, \, \| a \| \leq r \big\} \su A_0^{\op n}
  \]
  is totally bounded for all $r \geq 0$ (with respect to the Hilbert $C^*$-module norm $\| \cd \|$ on $A_0^{\op n}$).

  Let us choose $r_0 := \max\{ C_0(1 + C), C \}$ and observe that \eqref{eq:tilde}, condition $(1)$ and the fact that $\Phi : X \to A_0^{\op n}$ is an isometry imply that
  \[
\Phi\big( \big\{ x \in \C X \mid L(x) \leq 1 \big\} \big) \su D_{r_0} .
\]
Since $\Phi^* \Phi(x) = x$ for all $x \in \C X$ we obtain that $\big\{ x \in \C X \mid L(x) \leq 1 \big\} \su \Phi^*( D_{r_0})$. Observing that $\Phi^* : X \to A_0^{\op n}$ is a norm-contraction, we conclude that $\big\{ x \in \C X \mid L(x) \leq 1 \big\} \su X$ is totally bounded. 
\end{proof}

%Let us moreover suppose that $(\C A,H,D)$ is a twisted Lipschitz triple on $A$ with strongly continuous one-parameter group of unitaries $\{V_s\}_{s \in \B R}$ and left action given by an injective $*$-homomorphism $\psi : A \to \B L(H)$. We specify that $\C A \su A$ is a norm-dense unital $*$-subalgebra. The action of $\B R$ on $\B L(H)$ associated with the family $\{ V_s\}_{s \in \B R}$ is denoted by $\be$ and the lower semicontinuous slip-norm is denoted by $L_D^\be : \C A \to [0,\infty)$ so that $L_D^\be(a) = \| \de_\be(a) \|_\infty$ for all $a \in \C A$.

\section{Preliminaries on Lipschitz operators}\label{s:lipschitz}
In this section, we study Lipschitz operators associated to a selfadjoint unbounded operator and we shall see how to equip the corresponding Lipschitz algebra with a slip-norm measuring the size of first-order derivatives. In fact, we are going to transgress a bit beyond the standard material and study a notion of twisted Lipschitz operators where the twist is given by a strongly continuous one-parameter group of unitaries. 
\medskip

Let us fix a selfadjoint unbounded operator $D : \T{Dom}(D) \to H$ acting on a separable Hilbert space $H$. 

\begin{dfn}
A bounded operator $T : H \to H$ is a \emph{Lipschitz operator} (with respect to $D$), if $T\big( \T{Dom}(D) \big) \su \T{Dom}(D)$ and the commutator $[D,T] : \T{Dom}(D) \to H$ extends to a bounded operator on $H$. 
\end{dfn}

The set of Lipschitz operators form a unital $*$-subalgebra of $\B L(H)$ which we denote by $\Lip_D(H)$ and refer to as the \emph{Lipschitz algebra}. For an element $T \in \Lip_D(H)$, we apply the notation $\de(T) \in \B L(H)$ for the closure of the commutator $[D,T] : \T{Dom}(D) \to H$. In this fashion, we obtain a closed $*$-derivation $\de : \Lip_D(H) \to \B L(H)$. In particular, we obtain a slip-norm
\[
L_D : \Lip_D(H) \to [0,\infty) \q L_D(T) := \| \de(T) \|_\infty ,
  \]
  which turns out to be lower semicontinuous by \cite[Proposition 3.7]{Rie:MSS}.

Associated with $D$, we also have a strongly continuous one-parameter group of unitaries $\{U_t\}_{t \in \B R}$ given by $U_t := e^{it D}$ for all $t \in \B R$. The family $\{U_t\}_{t \in \B R}$ then yields an action $\al$ of the real line on $\B L(H)$. For each $t \in \B R$, the corresponding $*$-automorphism $\al_t$ is given by $\al_t(T) = U_t T U_t^*$. Notice that, fixing a $T \in \B L(H)$, the map $t \mapsto \al_t(T)$ is in general only continuous with respect to the $\si$-weak operator topology on $\B L(H)$. 

The following characterization of the elements in the Lipschitz algebra $\Lip_D(H)$ has been detailed out by Eric Christensen in \cite[Theorem 3.8]{Chr:WDO} (even though a few of the equivalent conditions are left out since they are less relevant for the present investigations). For clarity we apply the notation $\| \cd \|_H$ for the norm on $H$ (coming from the inner product).

\begin{thm}\label{t:weak}
  Let $T$ be a bounded operator on $H$. The following conditions are equivalent:
  \begin{enumerate}
  \item $T$ is a Lipschitz operator with respect to $D$;
  \item The subset here below is bounded in operator norm:
    \[
\big\{ (\al_t(T) - T)/t \mid t \in \B R \sem \{0\} \big\} \su \B L(H) .
\]
%\item There exists a bounded operator $R$ such that
%  \[
%\lim_{t \to 0} \Big| \binn{\xi,\big(\frac{\al_t(T) - T}{t} - R\big) \eta} \Big| = 0 \q \mbox{for all } \xi,\eta \in H ;
%\]
\item The subset here below is bounded in modulus:
  \[
\big\{ \inn{D\xi, T \eta} - \inn{\xi, T D \eta}  \mid \xi,\eta \in \T{Dom}(D) \, , \, \, \|\xi\|_H , \| \eta\|_H \leq 1 \big\} \su \B C
\]
 \end{enumerate}
  In this case, we have the identity $\| \de(T) \|_\infty = \sup_{t \in \B R \sem \{0\}} \| \al_t(T) - T\|_\infty/|t|$.
%
% = \sup\big\{ \big| \inn{D\xi, T \eta} - \inn{\xi, T D \eta} \big| \mid \xi,\eta \in \T{Dom}(D) \, , \, \, \|\xi\|_H , \| \eta\|_H \leq 1 \big\} 
\end{thm}

%By Stone's theorem, we know that the scalar multiple $i D : \T{Dom}(D) \to H$ is the infinitesimal generator of the family $\{U_t\}_{t \in \B R}$. 
%This means that the Let us specify what this means:
%\begin{enumerate}
%\item The domain of $D$ is given by
%\[
%\T{Dom}(D) = \Big\{ \xi \in H \mid \lim_{t \to 0} \frac{U_t(\xi) - \xi}{t} \T{ exists } \Big\} \q \T{and};
%\] 
%\item For $\xi \in \T{Dom}(D)$ we have the formula
%  \[
%D(\xi) = -i \cd \lim_{t \to 0} \frac{U_t(\xi) - \xi}{t} .
%  \]
%\end{enumerate}

\subsection{Analytic operators}
In order to describe the twisted Lipschitz operators we need to talk a bit about bounded operators which are analytic with respect to a strongly continuous one-parameter group of unitaries. For a detailed treatment of analytic elements in the general context where a Banach algebra is equipped with a strongly continuous isometric action of $\B R^n$ (for some $n \in \B N$) we refer the reader to \cite[Section 3]{Bos:POK}. Relevant information can also be found in \cite[Appendix A.1]{Tak:OAII}. 

Let us fix a strongly continuous one-parameter group of unitaries $\{ V_s\}_{s \in \B R}$ acting on a separable Hilbert space $H$. The action of the real line on $\B L(H)$ associated with the family $\{ V_s\}_{s \in \B R}$ is denoted by $\be$ so that $\be_s(T) = V_s T V_s^*$ for all $s \in \B R$ and $T \in \B L(H)$. As we remarked in Section \ref{s:lipschitz}, this action is in general only continuous with respect to the $\si$-weak operator topology on $\B L(H)$. 

For each $r \in [0,\infty)$, introduce the closed strips
  \[
  \begin{split}
  I_r := \big\{ z \in \B C \mid \C{I}(z) \in [0,r] \big\} \, \, \T{ and } \, \, \, 
  I_{-r} := \big\{ z \in \B C \mid \C{I}(z) \in [-r,0] \big\},
  \end{split}
\]
where $\C{I}(z)$ denotes the imaginary part of a complex number $z$. %We put $I_{-r} := \ov{I_r} = \big\{ z \in \B C \mid \T{Im}(z) \in [-r,0] \big\}$

\begin{dfn}\label{d:analytic}
  Let $r \in \B R$. We say that a bounded operator $T$ is \emph{analytic of order $r$}, if there exists a norm continuous map
    $f_T : I_r \to \B L(H)$ such that
    \begin{enumerate}
    \item The restriction of $f_T$ to the interior $I_r^\ci \su I_r$ is holomorphic;
    \item For each $s \in \B R$, it holds that $f_T(s) = \be_s(T)$.
    \end{enumerate}
  The subset of analytic operators of order $r$ is denoted by $\Ana_\be^r(H) \su \B L(H)$. %In the special case where $r = 1$ we put $\Ana_\be(H) := \Ana_\be^1(H)$.
\end{dfn}

Let us fix an $r \in \B R$. For $T \in \Ana_\be^r(H)$ it holds that the norm continuous function $f_T : I_r \to \B L(H)$ from Definition \ref{d:analytic} is unique and we define $\be_z(T) := f_T(z)$ for all $z \in I_r$. Using basic properties of operator valued holomorphic maps, it can be verified that $\Ana_\be^r(H) \su \B L(H)$ is a unital $\B C$-subalgebra and the associated map
  \[
\be_z : \Ana_\be^r(H) \to \B L(H)
\]
is a unital algebra homomorphism for all $z \in I_r$. It does not hold that $\be_z$ preserves the adjoint operation instead we have that $T \in \Ana_\be^r(H)$ if and only if $T^* \in \Ana_\be^{-r}(H)$ and in this case $\be_z(T^*) = \be_{\ov{z}}(T)^*$ for all $z \in I_{-r}$. % where $\ov{z} \in I_r$ denotes the complex conjugate of our $z \in I_r$.

An application of Stone's theorem tells us that there exists a selfadjoint unbounded operator $E : \T{Dom}(E) \to H$ such that $V_s = e^{is E}$ for all $s \in \B R$. For each $z \in \B C$, the functional calculus for selfadjoint unbounded operators then yields an invertible normal unbounded operator $V_z := e^{izE}$ with inverse $V_{-z} = e^{-izE}$. In the special case where $z = ir$ for some $r \in \B R$ it holds that $V_{ir} = e^{-rE}$ is an invertible positive unbounded operator. The following lemma can be found as \cite[Lemma 3.2]{PeTa:RNT}. %{\red be careful on the real line!}

\begin{lemma}\label{l:domchar}
  Let $\xi \in H$ and let $z$ be a complex number with imaginary part $r \in \B R$. It holds that $\xi \in \T{Dom}(V_z)$ if and only if there exists a norm continuous map $f_\xi : I_r \to H$ such that
  \begin{enumerate}
  \item The restriction of $f_\xi$ to the interior $I_r^\ci \su I_r$ is holomorphic;
  \item For each $s \in \B R$, it holds that $f_\xi(s) = V_s(\xi)$.
  \end{enumerate}
In this case, we have the identity $f_\xi(z) = V_z(\xi)$.
\end{lemma}

The next proposition is certainly well-known but we have been unable to find a reference for this result. 

\begin{prop}\label{p:anazet}
  Let $r \in \B R$ and let $z \in I_r$. If $T \in \Ana_\be^r(H)$, then it holds that
  \begin{enumerate}
  \item $T$ preserves the domain of $V_z$ and;
  \item We have the identity $V_z(T\xi) = \be_z(T) V_z(\xi)$ for all $\xi \in \T{Dom}(V_z)$.
  \end{enumerate}
  In particular, it holds that $V_z T V_{-z}(\eta) = \be_z(T)(\eta)$ for all $\eta \in \T{Dom}(V_{-z})$.
\end{prop}
\begin{proof}
  Suppose that $T \in \B L(H)$ is analytic of order $r$ and let $\xi \in \T{Dom}(V_z)$ be given. We put $t := \C{I}(z)$ and consider the norm continuous map $f_\xi : I_t \to H$ from Lemma \ref{l:domchar}. Since $T \in \Ana_\be^r(H)$ and $I_t \su I_r$, we know that the norm continuous map $f_{T\xi} : I_t \to H$ given by $f_{T\xi}(w) := \be_w(T) f_\xi(w)$ for all $w \in I_t$ is holomorphic on the interior $I_t^\ci \su I_t$ and agrees with $\be_s(T) V_s(\xi) = V_s(T\xi)$ for all $s \in \B R$. By Lemma \ref{l:domchar} we now get that $T(\xi) \in \T{Dom}(V_z)$ and that
  \[
V_z(T\xi) = f_{T\xi}(z) = \be_z(T) V_z \xi . %\qedhere
\]
The final claim of the proposition follows by noting that $V_{-z}$ is the inverse of $V_z$ implying that the image of $V_{-z}$ agrees with the domain of $V_z$. %the image of $\T{Im}$
\end{proof}

\subsection{Twisted Lipschitz operators and twisted Lipschitz triples}\label{ss:twistedlip}
We are now going to introduce the notion of a twisted Lipschitz operator where the commutator condition appearing in the definition of a Lipschitz operator is twisted by two particular analytic extensions of the automorphism group associated with a strongly continuous one-parameter group of unitaries. This twisting of commutators is in line with ideas appearing in \cite{CoMo:TST} even though we are here insisting on a slightly more symmetric kind of twist in order to have a better relationship with the adjoint operation, see the second identity in \eqref{eq:twistleib} here below. Indeed, the symmetric kind of twist applied here is necessary for obtaining $*$-invariance of the associated semi-norm. This is a key condition for the theory of compact quantum metric spaces, see Definition \ref{d:cqms} and the corresponding notion of a slip-norm from Definition \ref{d:slip}. 

Let us fix a selfadjoint unbounded operator $D : \T{Dom}(D) \to H$ acting on a separable Hilbert space $H$. We are moreover considering a strongly continuous one-parameter group of unitaries $\{ V_s\}_{s \in \B R}$ and denote the corresponding $\si$-weakly continuous one-parameter group of automorphisms of $\B L(H)$ by $\be$. %We are assuming that these ingredients are in so far that both $(i + D)^{-1}$ and $(-i + D)^{-1}$ are analytic of order $1$, see Definition \ref{d:analytic}.  %V_s \in \Lip_D(H)$ and $\de(V_s) = 0$ for all $s \in \B R$. The associated  so that $\be_s(T) = V_s T V_{-s}$ for all $s \in \B R$.

\begin{dfn}\label{d:twistop}
  A bounded operator $T : H \to H$ is a \emph{twisted Lipschitz operator} (with respect to $D$ and $\be$), if the following conditions hold:
  \begin{enumerate}
  \item $T, T^* \in \Ana_\be^1(H)$;
  \item $\be_i(T)$ preserves the domain of $D$ and;
  \item The twisted commutator $D \be_i(T) - \be_{-i}(T) D : \T{Dom}(D) \to H$ extends to a bounded operator on $H$.
  \end{enumerate}
\end{dfn}

The set of twisted Lipschitz operators is a unital $*$-subalgebra of $\B L(H)$ which we denote by $\Lip_D^\be(H)$. For an element $T \in \Lip_D^\be(H)$ we apply the notation $\de_\be(T) : H \to H$ for the bounded extension of the twisted commutator in part $(3)$ of Definition \ref{d:twistop}. The associated map $\de_\be : \Lip_D^\be(H) \to \B L(H)$ is then a \emph{twisted $*$-derivation} in so far that
\begin{equation}\label{eq:twistleib}
\de_\be(T S) = \de_\be(T) \be_i(S) + \be_{-i}(T) \de_\be(S) \, \, \T{ and } \, \, \, \de_\be(T^*) = - \de_\be(T)^* 
\end{equation}
for all $T,S \in \Lip_D^\be(H)$. 
\medskip

In the next definition we incorporate an extra unital $C^*$-algebra $A$ equipped with an injective unital $*$-homomorphism $\rho : A \to \B L(H)$ and a fixed norm-dense unital $*$-subalgebra $\C A \su A$. %Recall that the resolvent $(i + D)^{-1}$ is assumed to belong to the intersection $\Ana_\be^1(H) \cap \Ana_\be^{-1}(H)$.

\begin{dfn}\label{d:twistlip}
  We say that $(\C A,H,D)$ is a \emph{twisted Lipschitz triple} on $A$, if the injective unital $*$-homomorphism $\rho : A \to \B L(H)$ satisfies that $\rho(a) \in \Lip_D^\be(H)$ for all $a \in \C A$. 
\end{dfn}

For the rest of this subsection we assume that $(\C A,H,D)$ is a twisted Lipschitz triple. Our twisted Lipschitz triple is called \emph{graded}, if the separable Hilbert space $H$ is also equipped with a selfadjoint unitary operator $\ga$ such that
  \[
\ga \rho(a) = \rho(a) \ga \, \, , \, \, \, \ga V_s = V_s \ga \, \, \mbox{ and } \, \, \, \ga D = - D \ga
\]
for all $a \in A$ and $s \in \B R$. The selfadjoint unitary operator $\ga$ is then called the \emph{grading operator}.

  In the special situation, where the unitary operator $V_s$ agrees with the identity operator for all $s \in \B R$ we simply say that $(\C A,H,D)$ is a \emph{Lipschitz triple} (in line with the preliminary discussions in Subsection \ref{ss:lifts}). It is important to remark that, if we furthermore know that the resolvent $(i + D)^{-1}$ is compact, then the Lipschitz triple $(\C A,H,D)$ is a \emph{unital spectral triple}.

  We continue in the general setting where $(\C A,H,D)$ is a twisted Lipschitz triple. It is common to suppress the injective unital $*$-homomorphism $\rho$ and identify $A$ with a unital $C^*$-subalgebra of $\B L(H)$. In this fashion, the twisted $*$-derivation $\de_\be : \Lip_D^\be(H) \to \B L(H)$ restricts to a twisted $*$-derivation $\de_\be : \C A \to \B L(H)$. In particular, we may define the slip-norm
 \[
L_D^\be : \C A \to [0,\infty) \q L_D^\be(a) := \| \de_\be(a) \|_\infty , 
  \]
  which satisfies the \emph{twisted Leibniz inequality}:
  \[
L_D^\be(a b) \leq L_D^\be(a) \cd \| \be_i(b) \| + \| \be_{-i}(a) \| \cd L_D^\be(b) \q \T{for all } a,b \in \C A .
\]
Remark here that $\| \be_{-i}(a) \| = \| \be_i(a^*)^* \|$. It therefore makes sense to introduce the following notion, which is partially in line with terminology introduced by Fr\'ed\'eric Latr\'emoli\`ere, see e.g. \cite[Definition 1.8]{Lat:PMS}. The use of the terminology \emph{spectral metric space} goes back to the paper \cite{BMR:DSS} following a suggestion by Marc Rieffel.

  \begin{dfn}\label{d:spemet}
   The twisted Lipschitz triple $(\C A,H,D)$ is called \emph{metric}, if the pair $(\C A,L_D^\be)$ is a compact quantum metric space. In the special case where $(\C A,H,D)$ is a unital spectral triple we say that $(\C A,H,D)$ is a \emph{spectral metric space}, if the underlying Lipschitz triple is metric. 
  \end{dfn}

%  Let us spend some extra time analyzing the slip-norm $L_D^\be$.

 \begin{lemma}\label{l:sliplower}
 If the resolvent $(i + D)^{-1}$ belongs to the intersection $\Ana_\be^1(H) \cap \Ana_\be^{-1}(H)$, then the slip-norm $L_D^\be : \C A \to [0,\infty)$ is lower semicontinuous. 
  \end{lemma}
  \begin{proof}
    To ease the notation, we put $V_{-i} := \Ga$ so that $V_i = \Ga^{-1}$. Observe that the subspace $\T{Dom}(\Ga) \cap \T{Dom}(\Ga^{-1}) \su H$ is norm-dense and define the norm-dense subspace $\C H := (i + D)^{-1}\big( \T{Dom}(\Ga) \cap \T{Dom}(\Ga^{-1}) \big) \su H$. Since the resolvent $(i + D)^{-1} : H \to H$ is analytic of order $1$ and analytic of order $-1$, Proposition \ref{p:anazet} entails that $\C H \su \T{Dom}(\Ga^{-1} D) \cap \T{Dom}(\Ga)$. Indeed, letting $\xi \in \T{Dom}(\Ga) \cap \T{Dom}(\Ga^{-1})$ we get that $(i + D)^{-1}(\xi) \in \T{Dom}(\Ga)$ from Proposition \ref{p:anazet} so we just need to argue that $(i + D)^{-1}(\xi)$ belongs to $\T{Dom}(\Ga^{-1} D)$. Clearly, $(i + D)^{-1}(\xi)$ belongs to $\T{Dom}(D)$ and we also have that $D(i + D)^{-1}(\xi) = \xi - i(i + D)^{-1}(\xi)$ which belongs to $\T{Dom}(\Ga^{-1})$ by one more application of Proposition \ref{p:anazet}. This shows that $(i + D)^{-1}(\xi) \in \T{Dom}(\Ga^{-1}D)$. Let now $a \in \C A$ and apply Proposition \ref{p:anazet} to obtain the identities 
    \[
    %\begin{split}
    \binn{\xi, \de_\be(a) \eta} = \binn{D \xi, \be_i(a) \eta} - \binn{\be_i(a^*)\xi, D \eta} 
    = \binn{\Ga^{-1} D \xi,  a \Ga \eta} - \binn{a^* \Ga \xi, \Ga^{-1} D \eta}
    %\end{split}
    \]
    for all $\xi, \eta \in \C H$. Since $\C H \su H$ is norm-dense, our computation shows that 
    \[
    L_D^\be(a) = \sup_{\xi, \eta \in \C H \, , \, \, \| \xi \|_H , \| \eta \|_H \leq 1}
    \big| \binn{\Ga^{-1} D \xi,  a \Ga \eta} - \binn{a^* \Ga \xi, \Ga^{-1} D \eta} \big| 
    \]
    and hence that $L_D^\be : \C A \to [0,\infty)$ can be computed as a supremum of non-negative continuous functions on $\C A$. We conclude that $L_D^\be$ is lower semicontinuous. 
%
      % Let $\xi \in \T{Dom}(\Ga)$ and $t \in \B R$. Then $f_{V_t \xi} = V_t f_\xi : I_{-1} \to \B L(H)$ is norm-continuous and analytic on $I_{-1}^{\ci}$ and $f_{V_t \xi}(s) = V_t V_s \xi = V_s V_t \xi$ for all $s \in \B R$. This shows that $V_t \xi \in \T{Dom}(\Ga)$ and $\Ga V_t \xi = f_{V_t \xi}(-i) = V_t \Ga \xi$. Then $V_t (i + \Ga)^{-1} = (i + \Ga)^{-1}V_t$ for all $t \in \B R$ showing that $(i + \Ga)^{-1}$ is analytic of order $r$ for all $r \in \B R$. In particular, we obtain that the norm-dense subspace $(i + \Ga)^{-1}(\T{Dom}(\Ga^{-1})) \su H$ is contained in $\T{Dom}(\Ga) \cap \T{Dom}(\Ga^{-1})$.
  \end{proof}
 
  It is always possible to enlarge the unital $*$-algebra $\C A$ in a substantial way, replacing it with the intersection $\Lip_D^\be(A) := A \cap \Lip_D^\be(H)$. We refer to the norm-dense unital $*$-subalgebra $\Lip_D^\be(A) \su A$ as the \emph{twisted Lipschitz algebra} of our twisted Lipschitz triple $(\C A,H,D)$. It clearly holds that $(\Lip_D^\be(A),H,D)$ is again a twisted Lipschitz triple and we therefore have the slip-norm
\[
L_D^\be : \Lip_D^\be(A) \to [0,\infty) \q L_D^\be(a) := \| \de_\be(a) \|_\infty .
\]
This slip-norm is sometimes called the \emph{maximal} slip-norm of the twisted Lipschitz triple $(\C A,H,D)$.

In the special case where the unitary operator $V_s$ agrees with the identity operator for all $s \in \B R$ we apply the notation $\Lip_D(A) := A \cap \Lip_D(H)$ and refer to this norm-dense unital $*$-subalgebra of $A$ as the \emph{Lipschitz algebra}.

%      \begin{lemma}\label{l:deranti}
%If the twisted Lipschitz triple $(\C A,H,D)$ is graded with grading operator $\ga : H \to H$, then it holds that $\ga \de_\be(a) = - \de_\be(a) \ga$ for all $a \in \C A$.
%    \end{lemma}
%    \begin{proof}
%Let $a \in \C A$. Since $\ga \rho(a) = \rho(a) \ga$ and $\ga V_s = V_s \ga$ for all $s \in \B R$ we get that $\ga \be_i(a) = \be_i(a) \ga$ and $\ga \be_{-i}(a) = \be_{-i}(a) \ga$. Since $\ga$ moreover anti-commutes with $D$ we obtain the result of the lemma by verifying that the relevant identity holds on all vectors $\xi$ in the domain of $D$.
%    \end{proof}

\section{First main theorem}\label{s:first}
The purpose of this section is to state and prove our first main theorem which combines the result of Li's theorem for circle actions, Theorem \ref{t:licircle}, with our result on finitely generated projective modules, Theorem \ref{t:finproj}. We are to some extent using the setting of twisted Lipschitz triples presented in Subsection \ref{ss:twistedlip} but we do not need the full setup described there.  
\medskip

Let $A$ be a unital $C^*$-algebra which is equipped with a strongly continuous action $\si$ of the unit circle $S^1 \su \B C$. Let moreover $\C A \su A$ be a norm-dense unital $*$-subalgebra and let $\be_i : \C A \to B$ be a unital algebra homomorphism with values in a unital $C^*$-algebra $B$.

We do not assume that $\be_i$ is compatible with the involutions so that $\be_{-i} : \C A \to B$ defined by $\be_{-i}(a) := \be_i(a^*)^*$ is a unital algebra homomorphism which could be different from $\be_i$.

In many cases, the unital $C^*$-algebra $B$ will just agree with $A$ and $\be_i$ (and $\be_{-i}$) are obtained via analytic extension of the circle action $\si$ on $A$. This does for example happen for the quantum spheres which we are going to analyze in details later on in Section \ref{s:quasph}. However, in other cases it is necessary to work with a different choice of $B$. This situation occurs for crossed products of the form $C_r^*( \B Z, C(M))$ coming from a non-isometric diffeomorphism $\psi$ of a compact Riemannian manifold $M$, see \cite[Section 4]{KaKy:DCQ} for more details on this. But instead of diving too deep into these matters, let us continue with our exposition.

%In general, we refer the reader to Subsection \ref{ss:twistedlip} which discusses a very general setting where $B$ agrees with the bounded operators on a separable Hilbert space. 

Recall the definitions from Subsection \ref{ss:li} of the lower semicontinuous slip-norm $L_\ell : \Lip_{S^1}(A) \to [0,\infty)$ and the spectral projections $P_n : A \to A$ for $n \in \B Z$.

\begin{thm}\label{t:main}
  Let $L$ and $L_\be : \C A \to [0,\infty)$ be slip-norms. Suppose that the following conditions are satisfied:
    \begin{enumerate}
    \item $\si_\la(\C A) \su \C A$ for all $\la \in S^1$ and $\C A \su \Lip_{S^1}(A)$;
    \item $P_0(\C A) \su \C A$ and there exist finitely many elements $\ze_1^R,\ldots,\ze_k^R \in \C A \cap A_1$ and $\ze_1^L,\ldots,\ze_m^L \in \C A \cap A_1$ such that
  \[
  \sum_{j = 1}^k \ze_j^R (\ze_j^R)^* = 1 = \sum_{j = 1}^m (\ze_j^L)^* \ze_j^L ;
  \]
\item There exists a constant $C \geq 0$ such that $L_\ell(a), L_\be(a) \leq C \cd L(a)$ for all $a \in \C A$; 
\item $L_\be : \C A \to [0,\infty)$ satisfies the twisted Leibniz inequality
  \[
L_\be(a \cd b) \leq L_\be(a) \cd \| \be_i(b) \| + \| \be_{-i}(a) \| \cd L_\be(b) \q \mbox{for all } a,b \in \C A
\]
and the restriction $\be_i : \C A \cap A_0 \to B$ extends to a unital $*$-homomorphism $\be_i : A_0 \to B$;
\item $L_\be : \C A \to [0,\infty)$ is lower semicontinuous and $L_\be(\si_\la(a)) = L_\be(a)$ for all $a \in \C A$ and $\la \in S^1$;
\item The pair $(\C A \cap A_0, L_\be)$ is a compact quantum metric space.
    \end{enumerate}
Then it holds that $(\C A,L)$ is a compact quantum metric space.
%\item For each $n \in \B Z$ there exists a constant $K_n \geq 0$ such that
%  \[
%\| \be(a) \|_\infty \leq K_n \cd \| a \| \q \mbox{for all } a \in \C A \cap A_n .
%\]
\end{thm}

In order to apply Theorem \ref{t:finproj}, we consider the spectral subspace $A_n$ (for $n \in \B Z$) as a right Hilbert $C^*$-module over the fixed point algebra $A_0$. The inner product is given by $\inn{a,b} := a^* b$ for all $a,b \in A_n$ and the corresponding Hilbert $C^*$-module norm therefore agrees with the norm inherited from the unital $C^*$-algebra $A$. We put
\[
\C A_n := \C A \cap A_n
\]
and record that $\C A_n \cd \C A_0 \su \C A_n$ and $\inn{\C A_n,\C A_n} \su \C A_0$.

%and 
%For each $N \in \B N_0$ we define the \emph{spectral band}:
%\[
%X_N := \sum_{n = - N}^N A_n := \big\{ a \in A \mid P_n(a) = 0 \T{ for all } n \in \B Z \T{ with } |n| > N \big\} .
%\]
%Clearly, $X_N \su A$ is a complete operator system satisfying that $x \cd a \in X_N$ for all $x \in X_N$ and $a \in A_0$.
%
%We apply the notation $\C A_n := \C A \cap A_n$ for all $n \in \B Z$ and $\C X_N := \C A \cap X_N$.
%
%\begin{lemma}\label{l:betaest}
%  If condition $(7)$ in Theorem \ref{t:main} is satisfied and $P_n(\C A) \su \C A$ for all $n \in \B Z$, then for each $N \in \B N_0$ there exists a constant $C_N \geq 0$ such that
%  \[
%\| \be(x) \| \leq C_N \cd \| x \| \q \mbox{ for all } x \in \C A \cap X_N .
%  \]
%\end{lemma}
%\begin{proof}
%  Let $x \in \C A \cap X_N$ and observe that
%  \[
%  \begin{split}
%  \| \be(x) \| & = \big\| \be( P_0(x) ) + \sum_{n = 1}^N \big( \be(P_n(x)) + \be(P_{-n}(x)) \big) \big\| \\
%  & \leq K_0 \cd \| P_0(x) \| + \sum_{n = 1}^N \big( K_n \| P_n(x)\|  + K_{-n} \| P_{-n}(x)\| \big) \\
%  & \leq K_0 + \sum_{n = 1}^N ( K_n + K_{-n} ) \cd \| x \| . \qedhere
%  \end{split}
%  \]
%\end{proof}

In our proof of Theorem \ref{t:main} we need the following lemmas which we single out as independent results:

\begin{lemma}\label{l:convenient}
  If condition $(2)$ in Theorem \ref{t:main} is satisfied, then for every $n \in \B Z$ the following holds: 
  \begin{itemize}
  \item $P_n(\C A) = \C A_n$ and there exist finitely many elements $\ze_{(n,1)},\ldots,\ze_{(n,m_n)} \in \C A_n$ such that $\sum_{j = 1}^{m_n} \ze_{(n,j)} \ze_{(n,j)}^* = 1$.
  \end{itemize}
  In particular, we get that $\C A_n \su A_n$ is norm-dense and has a finite frame (see Definition \ref{d:frame}).
\end{lemma}
\begin{proof}
For $n = 0$ it suffices to remark that $P_0(\C A) \su \C A$ which implies that $P_0(\C A) = \C A_0$ (notice also that the unit $1$ belongs to $\C A_0$). We may thus focus on the situation where $n \in \B Z \sem \{0\}$. The existence of the elements $\ze_{(n,1)},\ldots, \ze_{(n,m_n)} \in \C A_n$ follows by induction, splitting in cases according to whether $n > 0$ or $n < 0$. Remark in this respect that $\C A_n \cd \C A_1 \su \C A_{n+1}$ and $\C A_n \cd \C A_{-1} \su \C A_{n-1}$. Furthermore, it holds that $x^* \in \C A_{-n}$ whenever $x \in \C A_n$. Consider now an element $a \in \C A$. We then have that
      \[
      P_n(a) = \sum_{j = 1}^{m_n} \ze_{(n,j)} \ze_{(n,j)}^* P_n(a) = \sum_{j = 1}^{m_n} \ze_{(n,j)} P_0( \ze_{(n,j)}^* a) .
      \]
      Since $P_0(\C A) \su \C A$ and $\C A$ is a $*$-algebra we may conclude that $P_n(a) \in \C A$. This shows that $P_n(\C A) \su \C A_n$. The reverse inclusion follows trivially by observing that if $a \in \C A_n$, then $a = P_n(a)$ and therefore, since $\C A_n \su \C A$ we get that $a \in P_n(\C A)$. %and hence that $P_n(\C A) = \C A_n$.
\end{proof}

\begin{lemma}\label{l:twistest}
  If condition $(2)$ and $(4)$ in Theorem \ref{t:main} are satisfied, then for every $n \in \B Z$ and every $y \in \C A_n$ there exists a constant $C_y \geq 0$ such that $L_\be( y^* x) \leq C_y \cd (L_\be(x) + \| x \|)$ for all $x \in \C A_n$.
\end{lemma}
\begin{proof}
  Let $n \in \B Z$ be given. Since $\be_i : \C A_0 \to B$ extends to a unital $*$-homomorphism $\be_i : A_0 \to B$ and $\C A_n \su A_n$ admits a finite frame (by Lemma \ref{l:convenient}), we get that $\be_i : \C A_n \to B$ extends to a bounded operator $\be_i : A_n \to B$ with operator norm say $K_n \geq 0$. Let now $y \in \C A_n$ be given. Putting $C_y := \max\big\{ L_\be(y) \cd K_n, \| \be_i(y) \| \big\}$, the twisted Leibniz inequality from condition $(4)$ now yields the relevant estimate for all $x \in \C A_n$:
  \[
  L_\be(y^* x) \leq L_\be(y^*) \cd \| \be_i(x) \| + \| \be_{-i}(y^*) \| \cd L_\be(x) \leq C_y \cd ( \| x \| + L_\be(x) ) . 
  \]
  Remark here that $L_\be(y^*) = L_\be(y)$ and $\| \be_{-i}(y^*) \| = \| \be_i(y) \|$.
\end{proof}

At this point we are ready to prove the first main result of this paper:

\begin{proof}[Proof of Theorem \ref{t:main}]
  Because of assumption $(1)$ and $(3)$ we may assume without loss of generality that $L = L_\ell + L_\be : \C A \to [0,\infty)$, see Theorem \ref{t:charac}.

  We are going to apply Li's theorem for circle actions, see Theorem \ref{t:licircle}, so we need to verify the conditions in there. First of all, we notice that assumption $(1)$ and $(5)$ imply that $L_\ell + L_\be : \C A \to [0,\infty)$ is indeed lower semicontinuous. Assumption $(1)$ implies that condition $(1)$ in Theorem \ref{t:licircle} is satisfied with $C = 1$ (remember that $L = L_\ell + L_\be$) and condition $(2)$ in Theorem \ref{t:licircle} holds because of assumption $(1)$, $(2)$ and $(5)$ and Lemma \ref{l:convenient}. Condition $(4)$ in Theorem \ref{t:licircle} is a consequence of assumption $(6)$ upon noting that $L_\ell(a) = 0$ for all $a \in \C A_0$.
%
  %Indeed, $L_\be : \C A \to [0,\infty)$ is lower semicontinuous by  and it clearly holds that $L_\ell : \Lip_{S^1}(A) \to [0,\infty)$ is lower semicontinuous (being a supremum of continuous functions) hence $L_\ell$ is also lower semicontinuous on $\C A \su \Lip_{S^1}(A)$.
%      
%       Notice in this respect that $\C A \su \Lip_{S^1}(A)$ by assumption $(1)$. Condition $(2)$ in Theorem \ref{t:licircle}  Remark that the identity $L_\ell(\si_z(a)) = L_\ell(a)$ (for $a \in \C A$) follows immediately from the definition of the slip-norm $L_\ell$.  We are therefore left with the problem of verifying condition $(3)$ in Theorem \ref{t:licircle}.
     
      In order to establish the remaining condition $(3)$ in Theorem \ref{t:licircle} we fix an $n \in \B Z \sem \{0\}$ and focus on applying Theorem \ref{t:finproj} for $X := A_n$ and $\C X := \C A_n$. We already know from Lemma \ref{l:convenient} that $\C X$ has a finite frame and assumption $(6)$ tells us that $(\C A_0,L_\be)$ is a compact quantum metric space. Observe now that condition $(1)$ in Theorem \ref{t:finproj} holds since $|n| \| x\| = L_\ell(x) \leq L(x)$ for all $x \in \C A_n$, see \eqref{eq:supspecsub} and remember that $n \neq 0$. The remaining condition $(2)$ in Theorem \ref{t:finproj} follows immediately from Lemma \ref{l:twistest} (recalling that $L_\be(a) \leq L(a)$ for all $a \in \C A$). This ends the proof of the present theorem.
\end{proof}

\section{The vertical geometric data}\label{s:vertical}
Our aim is now to analyze the setting from Subsection \ref{ss:li} in more details. More precisely, we are going to construct a Lipschitz triple such that the Lipschitz algebra agrees with $\Lip_{S^1}(A)$ and the associated slip-norm agrees with $L_\ell$. This Lipschitz triple captures the vertical geometry of our data.

Let $A$ be a unital $C^*$-algebra equipped with a strongly continuous action $\si$ of the unit circle $S^1 \su \B C$.

As in Section \ref{s:first}, for each $n \in \B Z$, we view the spectral subspace $A_n$ as a right Hilbert $C^*$-module over the fixed point algebra $A_0$. The relevant inner product is given by the formula $\inn{a,b} = a^* b$. We may also consider $A$ as a right module over $A_0$ and define the $A_0$-valued inner product given by
$\inn{a,b} := P_0(a^* b)$, where we recall from Subsection \ref{ss:li} that $P_0 : A \to A$ is a faithful conditional expectation. The completion of $A$ with respect to this inner product is a right Hilbert $C^*$-module over $A_0$ which is denoted by $X$. We apply the notation $\La : A \to X$ for the corresponding inclusion of $A$ into $X$. It then holds that $X$ becomes a $C^*$-correspondence from $A$ to $A_0$ with left action given by the injective unital $*$-homomorphism 
\[
\psi : A \to \B L(X) \q \psi(a)( \La(b)) := \La(a \cd b) .
\]
For each $n \in \B Z$ we identify $A_n$ with a Hilbert $C^*$-submodule of $X$ (via the inclusion $A_n \su A$). In this respect we note that the corresponding injective $A_0$-linear map $A_n \to X$ is an adjointable isometry with adjoint induced by the spectral projection $P_n : A \to A$.
%{\blu Notice there that $\B L(X)$ denotes the unital $C^*$-algebra of bounded adjointable operators on the right Hilbert $C^*$-module $X$.}

Let us also consider the infinite direct sum $\widehat{\bop}_{n \in \B Z} A_n$ of right Hilbert $C^*$-modules over $A_0$. This Hilbert $C^*$-module is obtained as the completion of the algebraic direct sum $\bop_{n \in \B Z} A_n$ with respect to the inner product $\binn{ \{a_n\}_{n \in \B Z}, \{b_n\}_{n \in \B Z}} := \sum_{n \in \B Z} a_n^* b_n$.

The structure of $X$ as a right Hilbert $C^*$-module over $A_0$ is illustrated in the next lemma:

\begin{lemma}\label{l:strucX}
  The $A_0$-linear map $\bop_{n \in \B Z} A_n \to X$ given by the formula $\{a_n\}_{n \in \B Z} \mapsto \sum_{n \in \B Z} \La(a_n)$ extends to a unitary isomorphism of right Hilbert $C^*$-modules $W : \widehat{\bop}_{n \in \B Z} A_n \to X$.
\end{lemma}

We let $\C O_{S^1}(A) \su A$ denote the smallest unital $*$-subalgebra which contains the spectral subspace $A_n$ for all $n \in \B Z$. It then holds that $\C O_{S^1}(A) \su A$ is norm-dense, see e.g. \cite[Proposition 2.5]{Ex:CAP}. Similarly, we define the norm-dense $A_0$-submodule $\C O_{S^1}(X) \su X$ as the image of $\C O_{S^1}(A)$ under the inclusion $\La : A \to X$. Comparing with Lemma \ref{l:strucX} it then holds that $\C O_{S^1}(X) = W\big( \bop_{n \in \B Z} A_n \big)$. 

Let us introduce the symmetric unbounded operator $\C N : \C O_{S^1}(X) \to X$ by putting
\[
\C N( \La(a)) = n \cd \La(a) \q \T{for all } n \in \B Z \T{ and } a \in A_n . 
\]
The closure of $\C N$ is denoted by $N : \T{Dom}(N) \to X$ and it is not hard to check that both of the unbounded operators $i \pm N : \T{Dom}(N) \to X$ are surjective (using that the unbounded operators $i \pm \C N$ have norm-dense images). In other words, we get that the closed symmetric unbounded operator $N : \T{Dom}(N) \to X$ is selfadjoint and regular, see \cite[Proposition 4.1]{KaLe:LGR} and \cite[Lemma 9.7 and Lemma 9.8]{Lan:HCM}.

The next theorem is due to Carey, Neshveyev, Nest, and Rennie, \cite[Proposition 2.9]{CNNR:TEK}. Notice however that the result in \cite[Proposition 2.9]{CNNR:TEK} is more general than the statement presented here (using a more general spectral subspace assumption instead of fullness). Recall that a right Hilbert $C^*$-module $Y$ over a $C^*$-algebra $B$ is said to be \emph{full}, if the $*$-subalgebra 
$\inn{Y,Y} := \T{span}_{\B C}\big\{ \inn{\xi,\eta} \mid \xi, \eta \in Y \big\} \su B$ is norm-dense.

We remind the reader that unbounded Kasparov modules, as introduced in \cite{BaJu:TBK} and further developped in numerous places (see e.g. \cite{Kuc:PUM,Bla:KOA,Hil:BIK,DGM:BGU,Kaa:UPK,DuMe:HEU}), form the correct bivariant analogue of spectral triples.

\begin{thm}\label{t:unbkas}
Suppose that $X$ is countably generated as a right Hilbert $C^*$-module over $A_0$. If $A_1$ and $A_{-1}$ are full, then the triple $(\C O_{S^1}(A),X,N)$ is an unbounded Kasparov module.
\end{thm}

%We now analyze the selfadjoint and regular unbounded operator $N$ in more details (without assuming that $A_1$ and $A_{-1}$ are full).
%
%For each $t \in \B R$, introduce the unitary operator $U_t : X \to X$ satisfying that
%\[
%U_t\big( \La(a) \big) = \La\big( \si_{e^{it}}(a) \big) \q \T{for all } a \in A .
%\]
%The family $\{ U_t\}_{t \in \B R}$ is then a strictly continuous one-parameter group of unitaries. %For each $n \in \B Z$, the family $\{U_t\}_{t \in \B R}$ can be applied to describe the spectral projection $\Phi_n : X \to X$ with image $\La(A_n) \su X$. This spectral projection is given explicitly by the formula
%\[
%\Phi_n(\xi) = \frac{1}{2\pi} \int_0^{2\pi} e^{-itn} \cd U_t(\xi) dt \q \T{for all } \xi \in X .
%\]
%It clearly holds that $\Phi_n( \La(a)) = \La(P_n(a))$ for all $a \in A$.
%
%It is relevant to remark that the family $\{U_t\}_{t \in \B R}$ implements the circle action on $A$ in so far that
%\begin{equation}\label{eq:implement}
%\psi\big( \si_{e^{it}}(a) \big) = U_t \psi(a) U_{-t} \q \T{for all } t \in \B R \T{ and } a \in A.
%\end{equation}

Let us from now on suppose that the right Hilbert $C^*$-module $X$ is countably generated and that $\rho : A_0 \to \B L(H_0)$ is an injective unital $*$-homomorphism with values in the bounded operators on a separable Hilbert space $H_0$. We may then consider the interior tensor product $G := X \hot_\rho H_0$ which we equip with the left action of $A$ given by the injective unital $*$-homomorphism $\psi \hot 1 : A \to \B L(G)$. Notice that our assumptions ensure that $G$ is again a separable Hilbert space. %satisfying that $(\psi \hot 1)(a)(x \ot \xi) = \psi(a)(x) \ot \xi$ for all $x \in X$ and $\xi \in H_0$.

The selfadjoint and regular unbounded operator $N : \T{Dom}(N) \to X$ induces a selfadjoint unbounded operator $N \hot 1 : \T{Dom}(N \hot 1) \to G$, which is defined as the closure of the symmetric unbounded operator $\C N \ot 1 : \C O_{S^1}(X) \ot_{A_0} H_0 \to X \hot_\rho H_0$. We record that the triple $( \C O_{S^1}(A),G, N \hot 1)$ is a Lipschitz triple, see the discussion after Definition \ref{d:twistlip}. This Lipschitz triple is however not a unital spectral triple in general. In fact, it holds that the resolvent $(i + N \hot 1)^{-1}$ is compact if and only if $H_0$ is finite dimensional. 

Associated with the Lipschitz triple $(\C O_{S^1}(A),G, N \hot 1)$, we have the Lipschitz algebra $\Lip_{\T{ver}}(A) := \Lip_{N \hot 1}(A) \su A$ and we refer to this norm-dense unital $*$-subalgebra as the \emph{vertical Lipschitz algebra}. The corresponding closed $*$-derivation is denoted by
$\de_{\T{ver}} : \Lip_{\T{ver}}(A) \to \B L(H)$ and the notation $L_{\T{ver}} := L_{N \hot 1} : \Lip_{\T{ver}}(A) \to [0,\infty)$ refers to the lower semicontinuous slip-norm given by
  \[
L_{\T{ver}}(a) := \big\| \de_{\T{ver}}(a) \big\|_\infty := \big\| \T{cl}( [N \hot 1, \psi(a) \hot 1]) \big\|_\infty .
\]
Notice here that the extra operation ``$\T{cl}$'' means that we take the closure of the commutator appearing (which is otherwise an unbounded operator defined on the domain of $N \hot 1$). 

  Let us put $U_t := e^{it(N \hot 1)}$ for all $t \in \B R$ and consider the strongly continuous one-parameter group of unitaries $\{ U_t \}_{t \in \B R}$ acting on $G = X \hot_\rho H_0$. The family $\{ U_t\}_{t \in \B R}$ then implements the circle action $\si$ on $A$ in so far that
  \begin{equation}\label{eq:implement}
e^{it(N \hot 1)} (\psi(a) \hot 1) e^{-it(N \hot 1)} = \psi( \si_{e^{it}}(a)) \hot 1 \q \T{for all } a \in A \, , \, \, t \in \B R .
\end{equation}
%for all $a \in A$ and $t \in \B R$.

The next theorem is now an immediate consequence of Theorem \ref{t:weak}, the injectivity of the unital $*$-homomorphism $\psi \hot 1 : A \to \B L(G)$ and the identity in \eqref{eq:implement}.

  \begin{thm}\label{t:vertical}
We have the identities $\Lip_{S^1}(A) = \Lip_{\T{ver}}(A)$ and $L_\ell = L_{\T{ver}}$. In particular, we get that the vertical Lipschitz algebra $\Lip_{\T{ver}}(A)$ and the slip-norm $L_{\T{ver}}$ are independent of the injective unital $*$-homomorphism $\rho : A_0 \to \B L(H_0)$.
  \end{thm}
%  \begin{proof}
%  
%    As a consequence of Lemma \ref{l:ideunb} it holds that $\Lip_{\T{ver}}(A) = \Lip_{D}(A)$ and that $L_{\T{ver}} = L_D : \Lip_D(A) \to [0,\infty)$. Since $U_t \hot 1 = e^{it D}$ for all $t \in \B R$ we also get from \eqref{eq:implement} that
%      \[
%(\psi \hot 1)( \si_{e^{it}}(a)) = e^{it D} \cd (\psi \hot 1)(a) \cd e^{-it D} .
%      \]
%      The result of Theorem \ref{t:weak} and the injectivity of the  then imply that
%      \[
%a \in \Lip_{S^1}(A) \lrar (\psi \hot 1)(a) \in \Lip_{N \hot 1}(G) \lrar a \in \Lip_{\T{ver}}(A) . 
%\]
%In this case, we furthermore have that
%\[
%L_\ell(a) = \sup_{t \in \B R \sem \{0\}} \frac{\| \si_{e^{it}}(a) - a\|}{|t|} = \| \de( a) \|_\infty =  L_{\T{ver}}(a).
%\]
%This proves the present theorem.
%  \end{proof}

\section{The horizontal geometric data}\label{s:horizontal} 
Let us continue in the setting of Subsection \ref{ss:li} where we have a unital $C^*$-algebra $A$ equipped with a strongly continuous action $\si$ of the unit circle. We are now interested in a more refined algebraic setup which is in line with the theory of quantum principal bundles, see e.g. \cite{Sch:PHS,BrMa:QGG,Haj:SCQ} and \cite[Chapter 5]{BeMa:QRG} for a more recent exposition. This extra algebraic data is ultimately going to help us in describing the horizontal part of the geometry we are investigating.
\medskip

Let us fix a norm-dense unital $*$-subalgebra $\C A \su A$. At the algebraic level we impose two extra conditions on our data:
%
%of the fixed point algebra $A_0$ together with a norm-dense vector subspace $\C A_1 \su A_1$ of the first spectral subspace $A_1$. The notation $\C A \su A$ refers to the smallest unital $*$-subalgebra of $A$ such that $\C A_1 \su \C A$. 
%
%On top of this data, we suppose that $(\C A_0,H_0,D_0)$ is a Lipschitz triple on the fixed point algebra $A_0 \su A$. The corresponding {\blu injective} $*$-homomorphism is denoted by $\rho : A_0 \to \B L(H_0)$ and the associated $*$-derivation is denoted by $\de_0 : \C A_0 \to \B L(H_0)$.
%
%Our first aim is to lift the selfadjoint unbounded operator $D_0 : \T{Dom}(D_0) \to H_0$ to a selfadjoint unbounded operator on the Hilbert space $G := X \hot_\rho H_0$. To this end, we impose some further conditions on our data:

\begin{assum}\label{a:specsub}
  Suppose that the following is satisfied:
  \begin{enumerate}
  \item $\C A$ is generated as a $*$-algebra by the subspace $\C A \cap A_1 \su \C A$; 
  \item There exist finitely many elements $\ze_1^R,\ldots,\ze^R_k \in \C A \cap A_1$ and $\ze_1^L,\ldots,\ze^L_m \in \C A \cap A_1$ such that
    $\sum_{j = 1}^k \ze_j^R (\ze_j^R)^* = 1 = \sum_{j = 1}^m (\ze_j^L)^* \ze_j^L$.
  \end{enumerate}
\end{assum}

A couple of elementary observations are collected in the next lemma (notice that the injectivity of the map in $(3)$ is a consequence of \cite[Proposition 2.5]{Ex:CAP}):

\begin{lemma}\label{l:algebra}
  If the first condition in Assumption \ref{a:specsub} is satisfied, then the following holds:
  \begin{enumerate}
  \item $\si_\la(\C A) = \C A$ for all $\la \in S^1$ and $\C A \su \Lip_{S^1}(A)$;
  \item $P_n(\C A) = \C A \cap A_n$ for all $n \in \B Z$;
  \item The map $a \mapsto \{ P_n(a) \}_{n \in \B Z}$ yields an isomorphism of vector spaces $\C A \cong \bop_{n \in \B Z} (\C A \cap A_n)$. 
    \end{enumerate}
%  \item For each $n,m \in \B Z$, it holds that $a \cd b \in P_{n+m}(\s A)$ for all $a \in P_n(\s A)$ and $b \in P_m(\s A)$. 
%  \item For each $n \in \B Z$, the $\s A_0$-right module $P_n(\s A) = \s A \cap A_n$ is finitely generated projective. In fact, there exist finitely many elements $\ze_1^n,\ldots,\ze^n_{k_n} \in P_n(\s A)$ such that
%    \[
%\sum_{j = 1}^{k_n} \ze_j^n (\ze_j^n)^* = 1 .
%\]
\end{lemma}

Apply the notation $C(S^1)$ for the unital $C^*$-algebra of continuous functions on the circle and let $\C O(S^1) \su C(S^1)$ denote the smallest unital $*$-subalgebra containing the inclusion $z : S^1 \to \B C$. Provided that the first condition in Assumption \ref{a:specsub} is satisfied, we define the \emph{canonical map} $\T{can} : \C A \ot_{\C A \cap A_0} \C A \to \C A \ot \C O(S^1)$ by the formula $\T{can}(x \ot y) := \sum_{n \in \B Z} x \cd P_n(y) \ot z^n$. Remark that the first condition in Assumption \ref{a:specsub} does indeed imply that the sum appearing in the definition of the canonical map only has finitely many non-zero terms. The result formulated in the next proposition is a consequence of \cite[Theorem 4.3]{AKL:PAG}. Notice that, in the statement, we are really referring to a quantum principal $S^1$-bundle for the \emph{universal differential calculus}, see \cite[Proposition 1.6]{Haj:SCQ}. %Alternatively we could say that $\C A$ is an $\C A_0$-Galois extension. 

\begin{prop}\label{p:quaprincipal}
  Suppose that the conditions in Assumption \ref{a:specsub} are satisfied. The $\B Z$-graded algebra $\C A \cong \bop_{n \in \B Z} (\C A \cap A_n)$ forms a quantum principal $S^1$-bundle in the sense that the canonical map
  $\T{can} : \C A \ot_{\C A \cap A_0} \C A \to \C A \ot \C O(S^1)$ is an isomorphism. % defined by $x \ot y \mapsto $ for all $x,y \in \C A$. % and $y \in \C A \cap A_n$ for some $n \in \B Z$.
%As a vector space $\C A$ a
\end{prop}

We are from now on assuming that the two conditions in Assumption \ref{a:specsub} are satisfied. 
\medskip

For each $n \in \B Z$ we view $A_n$ as a right Hilbert $C^*$-module over $A_0$ and put $\C A_n := \C A \cap A_n = P_n(\C A)$. Notice that Lemma \ref{l:convenient} shows that $\C A_n \su A_n$ is norm-dense and has a finite frame. Referring to the constructions from Section \ref{s:vertical} we also introduce the vector subspace $\C X := \La(\C A) \su X$ and record that $\C X \su X$ is a norm-dense right $\C A_0$-submodule satisfying that $\inn{ \C X, \C X} \su \C A_0$. We record that our assumptions entail that $X$ is countably generated as a right Hilbert $C^*$-module over $A_0$.
\medskip

Suppose that we are given a Lipschitz triple $(\C A_0,H_0,D_0)$ on the fixed point algebra $A_0 \su A$. The corresponding injective $*$-homomorphism is denoted by $\rho : A_0 \to \B L(H_0)$ and the associated $*$-derivation is denoted by $\de_0 : \C A_0 \to \B L(H_0)$. We are interested in transferring this geometric data to the norm-dense unital $*$-subalgebra $\C A \su A$ and in this fashion access the horizontal part of our geometry.
\medskip

For each $n \in \B Z$, define the separable Hilbert space $G_n := A_n \hot_\rho H_0$ and identify $G_n$ with a closed subspace of $G := X \hot_\rho H_0$ via the inclusion of $A_n$ into $X$. Applying Lemma \ref{l:strucX}, we obtain a unitary isomorphism identifying $G$ with the Hilbert space direct sum $\widehat{\bop}_{n \in \B Z} G_n$.
%
%Our first aim is to lift the selfadjoint unbounded operator $D_0 : \T{Dom}(D_0) \to H_0$ to a selfadjoint unbounded operator on the Hilbert space $G := X \hot_\rho H_0$. To this end, we impose some further conditions on our data:
%
%{\blu As in Section \ref{s:lifts} we may apply the finite frame} to write down the associated {\blu \emph{}}
%\[
%\Na_{\T{Gr}} : \s X_n \to X_n \hot_\rho \B L(H_0) \q \Na_{\T{Gr}}\big(\La(\xi) \big)
%:= \sum_{j = 1}^{k_n} \La(\ze_j^n) \ot \de_0\big( (\ze_j^n)^* \cd \xi \big) 
%\]
%together with the horizontal lift $1 \ot_{\Na_{\T{Gr}}} D_0 : \s X_n \ot_{\s A_0} \T{Dom}(D_0) \to X_n \hot_\rho H_0$ of the selfadjoint unbounded operator $D_0 : \T{Dom}(D_0) \to H_0$. This horizontal lift is given explicitly by the formula
%\[
%(1 \ot_{\Na_{\T{Gr}}} D_0)(\La(\xi) \ot \eta) = \sum_{j = 1}^{k_n} \La(\ze_j^n) \ot D_0\big( \rho((\ze_j^n)^* \cd \xi)(\eta) \big)
%\]
%and we recall from Proposition \ref{p:esself} that the horizontal lift is an essentially selfadjoint unbounded operator.

Suppose that we have a fixed Hermitian $\de_0$-connection $\Na_n : \C A_n \to A_n \hot_\rho \B L(H_0)$ for every $n \in \B Z$ and consider the associated horizontal lift $1 \ot_{\Na_n} D_0 : \C A_n \ot_{\C A_0} \T{Dom}(D_0) \to A_n \hot_\rho H_0$ given by the expression
\[
(1 \ot_{\Na_n} D_0)(x \ot \xi) := \T{ev}_{\xi}( \Na_n(x) ) + x \ot D_0(\xi)
\]
on simple tensors. Recall from Theorem \ref{t:esself} that the horizontal lift $1 \ot_{\Na_n} D_0$ is a selfadjoint unbounded operator.

We define the Hermitian $\de_0$-connection
\begin{equation}\label{eq:direct}
\Na : \C X \to X \hot_\rho \B L(H_0) \q \Na( \La(a)) = \sum_{n \in \B Z} \Na_n( P_n(a)) ,
\end{equation}
where we are suppressing the inclusion $A_n \hot_\rho \B L(H_0) \to X \hot_\rho \B L(H_0)$ for every $n \in \B Z$. Let us now argue that the associated horizontal lift
\[
1 \ot_\Na D_0 : \C X \ot_{\C A_0} \T{Dom}(D_0) \to X \hot_\rho H_0
\]
is essentially selfadjoint. First of all, upon suppressing the unitary isomorphism $G \cong \widehat{\bop}_{n \in \B Z} G_n$ we get the formula
\[
(1 \ot_\Na D_0)( \{ \xi_n\}_{n \in \B Z}) = \big\{ (1 \ot_{\Na_n} D_0)(\xi_n) \big\}_{n \in \B Z}
\]
for all vectors $\{ \xi_n\}_{n \in \B Z}$ in the algebraic direct sum $\bop_{n \in \B Z} \T{Dom}(1 \ot_{\Na_n} D_0)$, which is now equal to $\T{Dom}(1 \ot_\Na D_0)$. So, we may view $1 \ot_\Na D_0$ as an infinite algebraic direct sum of the selfadjoint unbounded operators $1 \ot_{\Na_n} D_0$ for $n \in \B Z$. An application of \cite[Proposition 3.8]{Sch:USO} now shows that it suffices to verify that $1 \ot_\Na D_0 - i$ and $1 \ot_\Na D_0 + i$ both have dense images. But this is immediate since $1 \ot_{\Na_n} D_0$ is selfadjoint for every $n \in \B Z$ implying that $1 \ot_{\Na_n} D_0 - i$ and $1 \ot_{\Na_n} D_0 + i$ are in fact surjective as operators from $\C A_n \ot_{\C A_0} \T{Dom}(D_0)$ to $A_n \hot_\rho H_0 = G_n$.

Recall from Section \ref{s:tensor} that the interior tensor product $X \hot_\rho H_0$ can be viewed as a left module over $A$ where the left action is provided by the injective unital $*$-homomorphism $\psi \hot 1 : A \to \B L(X \hot_\rho H_0)$. During the proof of the next lemma, we also consider the interior tensor product $X \hot_\rho \B L(H_0)$ as a left module over $A$ (in the the same fashion). Both of the associated injective unital $*$-homomorphisms are being suppressed from the notation.
%
%{\blu (explain left module structure on interior tensor products).}
%
%Let $n,m \in \B Z$ and remark that an element $x \in A_{m-n}$ induces a bounded adjointable operator $\psi(x) : X_n \to X_m$. In particular, we get a bounded adjointable operator $\psi(x) \hot 1 : X_n \hot_\rho \B L(H_0) \to X_m \hot_\rho \B L(H_0)$.

\begin{lemma}\label{l:bounded}
  Suppose that the conditions in Assumption \ref{a:specsub} are satisfied and let $\Na_n : \C A_n \to A_n \hot_\rho \B L(H_0)$ be a Hermitian $\de_0$-connection for every $n \in \B Z$. Consider the Hermitian $\de_0$-connection $\Na : \C X \to X \hot_\rho \B L(H_0)$ from \eqref{eq:direct}. For each $m \in \B Z$ and each $a \in \C A$, there exists a unique bounded operator
  $d_m(a) : G_m \to G$ satisfying that
  \[
\big[ 1 \ot_\Na D_0, a \big](\eta) = d_m(a)(\eta) \q \mbox{for all } \eta \in \T{Dom}(1 \ot_{\Na_m} D_0) .
\]
%for all $\eta \in \T{Dom}(1 \ot_{\Na_m} D_0)$.
%\[
%  d_m(x)(\La(a) \ot \xi) = \T{ev}_\xi\Na\big( \La(x \cd a) \big) - x \cd \T{ev}_\xi\big( \Na_m(\La(a))\big)
%\]
%  for all $a \in \C A_m$ and $\xi \in H_0$;
%  \item $R_m^x(\La(\xi \cd a)) = R_m^x(\La(\xi)) \cd \rho(a)$ for all $\xi \in A_m$ and $a \in A_0$.
\end{lemma}
\begin{proof}
  Consider elements $x \in \C X$ and $\xi \in \T{Dom}(D_0)$. For every $a \in \C A$ it then holds that $a \cd (x \ot \xi) \in \C X \ot_{\C A_0} \T{Dom}(D_0)$ and the commutator is given by
  \[
  \big[ 1 \ot_\Na D_0, a \big](x \ot \xi) = \T{ev}_\xi \Na( a \cd x) - a \cd \T{ev}_\xi \Na(x)
  = \T{ev}_\xi \big( \Na( a \cd x) - a \cd \Na(x) \big) .
  \]
  Let now $m \in \B Z$ be given and recall that $\C A_m \su A_m$ has a finite frame. By Lemma \ref{l:extension} it therefore suffices to show that the commutator 
  \[
\C R_a := [\Na, a] : \C A_m \to X \hot_\rho \B L(H_0)
\]
satisfies that $\C R_a(x \cd b) = \C R_a(x) \cd \rho(b)$ for all $x \in \C A_m$ and $b \in \C A_0$. But this follows from the fact that the Hermitian $\de_0$-connection $\Na$ obeys the Leibniz rule, see Definition \ref{d:hermit} $(1)$. Notice that, in the notation of Lemma \ref{l:extension}, we get that $d_m(a) = c(R_a) : A_m \hot_\rho H_0 \to X \hot_\rho H_0$.
%\[
%\C R_a(x \cd b) = \Na( a \cd x \cd b) - a \cd \Na(x \cd b) = \Na(a \cd x) \cd \rho(b) + a \cd x \ot \de_0(b) - a \cd \Na(x) \cd \rho(b) - a \cd x \ot \de_0(b) = \C R_a(x) \cd \rho(b) . \qedhere
%\]
\end{proof}

\subsection{Commutators with the horizontal lift}\label{ss:commutator}
At this point of our exposition we are going to present a more careful investigation of the result of Lemma \ref{l:bounded}. In the context of this lemma, for each $a \in \C A$ it would of course be desirable to have an upper bound for the operator norm of the operators $d_m(a) : G_m \to G$ (independent of $m \in \B Z$). However, in the applications we are primarily interested in, such a condition is way too restrictive and instead we are allowing these operator norms to satisfy an exponential growth condition. % with respect to the integers $m \in \B Z$.

Let $A$ be a unital $C^*$-algebra equipped with a strongly continuous action $\si$ of the unit circle $S^1 \su \B C$ and let $\C A \su A$ be a fixed norm-dense unital $*$-subalgebra satisfying that $P_0(\C A) \su \C A$. Suppose that $(\C A \cap A_0,H_0,D_0)$ is a Lipschitz triple on the fixed point algebra $A_0$. We let $\rho : A_0 \to \B L(H_0)$ and $\de_0 : \C A \cap A_0 \to \B L(H_0)$ denote the associated injective unital $*$-homomorphism and the $*$-derivation, respectively.

%Suppose that the conditions in  are satisfied with respect to a fixed norm-dense vector subspace $\C A_1 \su A_1$. We let $\C A \su A$ denote the smallest unital $*$-subalgebra containing $\C A_1$.
%
%For each $n \in \B Z$, we let $\Na_n : \C X_n \to X_n \hot_\rho \B L(H_0)$ be a Hermitian $\de_0$-connection. Putting these connections together we obtain the Hermitian $\de_0$-connection $\Na : \C X \to X \hot_\rho \B L(H_0)$ which, for each $n \in \B Z$, agrees with $\Na_n$ on the right $\C A_0$-submodule $\C X_n$.
%
%We assume that $\Na_0( \La(a)) = 1 \ot \de_0$
%
%Recall from Lemma \ref{l:bounded} that for every $a \in \C A$ and $m \in \B Z$, the operator $d_m(a) : G_m \to H$ denotes the bounded extension of the commutator $[1 \ot_\Na D_0, a] : \C X_m \ot_{\C A_0} \T{Dom}(D_0) \to X \hot_\rho H_0$.

In order to treat a number of interesting examples from the point of view of compact quantum metric spaces (and noncommutative geometry in general) we impose the following extra conditions on our data. Notice that these constraints imply that the conditions in Assumption \ref{a:specsub} are satisfied as well. % (see Lemma \ref{l:algebra}). 

\begin{assum}\label{a:exponential}
  Suppose that the norm-dense unital $*$-subalgebra $\C A \su A$ is generated as a $*$-algebra by the subspace $\C A \cap A_1 \su \C A$. Suppose moreover that we are given
  \begin{itemize}
  \item A separable Hilbert space $H$ containing $H_0$ as a closed subspace;
  \item A unital $*$-homomorphism $\phi : A \to \B L(H)$ and a $\cc$-linear map $\de : \C A \to \B L(H)$,
  \end{itemize}
  such that the following holds:
  \begin{enumerate}
  \item $\de(a)(\xi) = \de_0(a)(\xi)$ and $\phi(a)(\xi) = \rho(a)(\xi)$ for all $a \in \C A \cap A_0$ and $\xi \in H_0$;
  \item There exists a constant $\mu > 0$ such that $\de(a b) = \de(a) \phi(b) + \mu^n \phi(a) \de(b)$ whenever $a \in \C A \cap A_n$ for some $n \in \B Z$ and $b \in \C A$;
  \item There exist finitely many elements $\ze_1^R,\ldots,\ze^R_k \in \C A \cap A_1$ and $\ze_1^L,\ldots,\ze^L_m \in \C A \cap A_1$ such that
    \[
    \begin{split}
    & \sum_{j = 1}^k \ze_j^R (\ze_j^R)^* = 1 = \sum_{j = 1}^m (\ze_j^L)^* \ze_j^L \q \mbox{and} \\ 
    & \sum_{j = 1}^k \phi( \ze_j^R) \de((\ze_j^R)^*) = 0 = \sum_{j = 1}^m \phi(\ze_j^L)^* \de( \ze_j^L) .
    \end{split}
    \]
  \end{enumerate}
\end{assum}

Let us present some further discussion of the conditions in Assumption \ref{a:exponential} (which are thus in effect here below).

For every $n \in \B Z$ we put $\C A_n := \C A \cap A_n = P_n(\C A)$ so that $\C A_n \su A_n$ is norm-dense and has a finite frame, see the discussion after Proposition \ref{p:quaprincipal}. Recall also that the separable Hilbert space $G_n = A_n \hot_\rho H_0$ may be viewed as a closed subspace of the separable Hilbert space
\[
G = X \hot_\rho H_0 \cong \widehat{\bop}_{n \in \B Z} G_n .
\]% together with the closed subspaces $G_n = A_n \hot_\rho H_0$.
%
%Let $\io : H_0 \to H$ denote the inclusion of the closed subspace $H_0$ into $H$. The first condition in Assumption \ref{a:exponential} can then be rephrased by saying that for every $a \in \C A_0$ we have the following identities in $\B L(H_0,H)$:
%\[
%\de(a) \io = \io \de_0(a) \, \, \T{ and } \, \, \, \phi(a) \io = \io \rho(a) .
%\]

For every $n \in \B Z$, we may identify $G_n = A_n \hot_\rho H_0$ with a closed subspace $H_n \su H$ via the isometry
\begin{equation}\label{eq:isometry}
S_n : A_n \hot_\rho H_0 \to H \q S_n(x \ot \xi) := \phi(x)(\xi) .
\end{equation}
Let us apply the notation $U_n : G_n \to H_n$ for the associated unitary isomorphism. Defining the separable Hilbert space $K := \widehat{\bop}_{n \in \B Z} H_n$, we thus obtain a unitary isomorphism of Hilbert spaces $U : G \to K$ which, for each $n \in \B Z$, identifies the closed subspaces $G_n \su G$ and $H_n \su K$ via $U_n$.

In a similar fashion, for every $n \in \B Z$, we introduce the $\B L(H_0)$-linear map
\begin{equation}\label{eq:defphi}
\Phi_n : A_n \hot_\rho \B L(H_0) \to \B L(H_0,H_n) \q \Phi_n(R)(\xi) := U_n( \T{ev}_\xi R) 
\end{equation}
and record that $\Phi_n(R)^* \Phi_n(S) = \inn{R,S}$ for all $R,S \in A_n \hot_\rho \B L(H_0)$ so that $\Phi_n$ becomes an isometry. We may in fact collect these isometries (for different values of $n \in \B Z$) into a single $\B L(H_0)$-linear isometry $\Phi : X \hot_\rho \B L(H_0) \to \B L(H_0,K)$ satisfying that
\[
\Phi(R)(\xi) := U( \T{ev}_\xi R) \q \T{for all } R \in X \hot_\rho \B L(H_0) \, \, \T{ and } \, \, \, \xi \in H_0 .
\]

%It then becomes relevant to record the formula 
%\begin{equation}\label{eq:phiyou}
%\Phi_n(R)(\xi) = U_n( \T{ev}_\xi R) \q \T{for all } R \in A_n \hot_\rho \B L(H_0) \, \, \T{ and } \, \, \, \xi \in H_0 .
%\end{equation}

Fixing the constant $\mu > 0$ appearing in the second condition in Assumption \ref{a:exponential}, we also consider the strongly continuous one-parameter group of unitary operators $\{V_s^\mu\}_{s \in \B R}$ on the separable Hilbert space $G = X \hot_\rho H_0$ defined by the formula
\begin{equation}\label{eq:grpuni}
V_s(\mu)(\ze) = \mu^{is \frac{n}{2}} \cd \ze %\q \ze \in A_n \hot_\rho H_0 .
\end{equation}
whenever $\ze$ belongs to the subspace $G_n = A_n \hot_\rho H_0$ for some $n \in \B Z$. Comparing with the strongly continuous one-parameter group of unitaries $\{U_t\}_{t \in \B R} = \{ e^{it(N \hot 1)} \}_{t \in \B R}$ introduced in Section \ref{s:vertical} we have that
$V_s(\mu) = U_{\log(\mu) s/2}$ for all $s \in \B R$.

Let us denote the corresponding $\si$-weakly continuous action of $\B R$ on $\B L(G)$ by $\be := \be(\mu)$ so that $\be_s(T) = V_s(\mu) T V_{-s}(\mu)$ for all $s \in \B R$. Suppressing the injective unital $*$-homomorphism $\psi \hot 1 : A \to \B L(G)$, for every $n \in \B Z$ and $a \in A_n$ it then holds that $a \in \Ana^r_{\be}(G)$ for all $r \in \B R$ and we have the formula 
\begin{equation}\label{eq:beta}
\be_z(a) = \mu^{i z \frac{n}{2} } \cd a \q \T{for all } z \in \B C .
\end{equation}
In particular, we get an induced algebra automorphism $\be_z : \C A \to \C A$ for every $z \in \B C$. The second condition in Assumption \ref{a:exponential} can now be restated as a \emph{twisted Leibniz rule}:
\[
\de(ab) = \de(a) \phi(b) + \phi( \be_{-2i}(a) ) \de(b) \q a,b \in \C A .
\]

\begin{lemma}\label{l:framevanish}
  If the conditions in Assumption \ref{a:exponential} are satisfied, then for every $n \in \B Z$ there exists a finite frame $\{\ze_{(n,j)} \}_{j = 1}^{m_n}$ for $\C A_n \su A_n$ such that
  \begin{equation}\label{eq:framevanish}
  \sum_{j = 1}^{m_n} \phi( \ze_{(n,j)} ) \de(\ze_{(n,j)}^*) = 0 .
  \end{equation}
\end{lemma}
\begin{proof}
  For $n = 0$ our finite frame simply consists of one element namely the unit $1$ and the relevant identity follows by noting that $\de(1) = 0$.

  The proof now runs by induction according to the cases $n > 0$ or $n < 0$. We focus on the case where $n > 0$ (since the remaining case is similar). The induction start follows immediately from Assumption \ref{a:exponential} with $m_1 = k$ and $\ze_{(1,j)} := \ze_j^R$ for all $j \in \{1,\ldots,k\}$. Let thus $n \in \B N$ and suppose that we have our finite frame $\{ \ze_{(n,i)}) \}_{i = 1}^{m_n}$ for $\C A_n \su A_n$ satisfying \eqref{eq:framevanish}. We put $m_{n + 1} := k \cd m_n$ and consider the finite index set $I_{n+1} := \{1,\ldots,m_n\} \ti \{1,\ldots,k\}$ (with $m_{n+1}$ elements). The relevant frame for $\C A_{n+1}$ is then given by $\{ \ze_{(n,i)} \ze_j^R \}_{(i,j) \in I_{n+1}}$. Indeed, suppressing $\phi : A \to \B L(H)$, we get from the twisted Leibniz rule that
  \[
  \begin{split}
  & \sum_{i = 1}^{m_n} \sum_{j = 1}^k \ze_{(n,i)} \ze_j^R \de( (\ze_j^R)^* \ze_{(n,i)}^*) \\
  & \q = \sum_{i = 1}^{m_n} \sum_{j = 1}^k \ze_{(n,i)} \ze_j^R \de( (\ze_j^R)^*) \ze_{(n,i)}^* + \sum_{i = 1}^{m_n} \sum_{j = 1}^k \mu^{-1} \ze_{(n,i)} \ze_j^R (\ze_j^R)^* \de(\ze_{(n,i)}^*) = 0 . \qedhere
  \end{split}
  \]
%  \[
%  \begin{split}
%  & \sum_{j = 1}^m \sum_{i = 1}^{m_n} \de( \ze_{ (n,i)} \ze_{(-1,j)} ) \ze_{(-1,j)}^* \ze_{(n,i)}^* \\
%  & \q = \sum_{j = 1}^m \sum_{i = 1}^{m_n} \big( \de( \ze_{ (n,i)} )  \ze_{(-1,j)} + \mu^n \ze_{(n,i)} \de(\ze_{(-1,j)}) \big) \ze_{(-1,j)}^* \ze_{(n,i)}^*
%  = 0
%  \end{split}
%  \]
\end{proof}

In the next proposition we give an explicit computation of Grassmann connections (see \eqref{eq:grass}) associated with finite frames subject to the vanishing condition in \eqref{eq:framevanish}. 

\begin{prop}\label{p:grasstwist}
  Suppose that the conditions in Assumption \ref{a:exponential} are satisfied. Let $n \in \B Z$ and let $\{\ze_j\}_{j = 1}^{m_n}$ be a finite frame for $\C A_n \su A_n$ such that $\sum_{j = 1}^{m_n} \phi(\ze_j)\de(\ze_j^*) = 0$. Then the corresponding Grassmann connection $\Na_{\T{Gr}} : \C A_n \to A_n \hot_\rho \B L(H_0)$ satisfies that
  \[
\Phi_n\big( \Na_{\T{Gr}}(x ) \big)\xi = \mu^{-n} \de(x) \xi  \q \mbox{for all } x \in \C A_n \mbox{ and } \xi \in H_0 .
\]
%In particular, we get that $\Na_\ze$ is independent of the choice of frame in this case.
\end{prop}
\begin{proof}
  Let $x \in \C A_n$ and let $\xi \in H_0$. The result of the proposition follows from the computation 
  \[
  \begin{split}
  \Phi_n\big( \Na_{\T{Gr}}( x)\big) \xi & = \sum_{j = 1}^{m_n} \Phi_n\big( \ze_j \ot \de_0( \ze_j^* \cd x) \big) \xi 
 = \sum_{j = 1}^{m_n} \phi(\ze_j) \de( \ze_j^* \cd x) \xi \\
  & = \sum_{j = 1}^{m_n} \phi(\ze_j) \big( \de( \ze_j^*) \cd \phi(x) + \phi( \be_{-2i}(\ze_j^*) ) \de(x) \big) \xi
    = \mu^{-n} \de(x) \xi . \qedhere
  \end{split}
  \]
\end{proof}

%Assume that there exist a constant $\mu > 0$ and a length function $\ell : \B Z \to \B R$ such that for every $a \in \C A$ the set 
%  \[
%\big\{ \mu^{\ell(m)} \cd d_m(a) \mid m \in \B Z \big\} 
%\]
%is bounded in operator norm.
%Suppose that the conditions in Assumption \ref{a:specsub} are satisfied and let $\Na_n : \s X_n \to X_n \hot_\rho \B L(H_0)$ be a Hermitian $\de_0$-connection for every $n \in \B Z$. 

For each $n \in \B Z$, let us choose a finite frame $\{ \ze_{(n,j)} \}_{j = 1}^{m_n}$ for $\C A_n \su A_n$ subject to the vanishing condition in \eqref{eq:framevanish}. Since the result of Proposition \ref{p:grasstwist} in particular shows that the corresponding Grassmann connection does not depend on the specific frame, we denote this Grassmann connection by $\Na_n : \C A_n \to A_n \hot_\rho \B L(H_0)$. Accordingly, upon suppressing the inclusion $\La : A \to X$, we obtain the Hermitian $\de_0$-connection $\Na : \C A \to X \hot_\rho \B L(H_0)$ satisfying that $\Na(a) = \Na_n(a)$ as soon as $a \in \C A_n$ for some $n \in \B Z$, see \eqref{eq:direct}. % It follows by Proposition \ref{p:grasstwist} that $\Phi \Na(a) = \mu^{-n} \de(a) $

For $a \in \C A$, we are now going to analyze the bounded operators $d_m(a) : G_m \to G$ appearing in Lemma \ref{l:bounded} in our present more specific setting. Notice that $d_m(a)$ factorizes through $G_{n + m} \su G$ in the case where $a \in \C A_n$ for some $n \in \B Z$.

%For each $a \in \C A$ and $m \in \B Z$ we recall that $d_m(a) : G_m \to G$ is notation for the bounded extension of the commutator $[ 1 \ot_\Na D_0, a ] : \C X_m \ot_{\C A_0} \T{Dom}(D_0) \to G$. Recall in this respect that the restriction of the Hermitian $\de$-connection $\Na : \C X \to X \hot_\rho \B L(H_0)$ agrees with $\Na_n$ as soon as the input belongs to $\C X_n$ for some $n \in \B Z$.  

\begin{lemma}\label{l:deedel}
  Suppose that the conditions in Assumption \ref{a:exponential} are satisfied. Let $n,m \in \B Z$ and $a \in \C A_n$ be given. For each $\eta \in G_m$, we have the identity
  \[
U_{n + m} d_m(a) (\eta) = \mu^{-n - m} \cd \de( a ) U_m(\eta) .
\]
In particular, it holds that $\de(a) U_m(\eta) \in H_{n + m}$.
\end{lemma}
\begin{proof}
  Let $x \in \C A_m$ and let $\xi \in H_0$. It follows from Proposition \ref{p:grasstwist} and the twisted Leibniz rule that
\[
\begin{split}
\Phi_{n + m}\big( \Na_{n + m} ( a x) - a \Na_m(x) \big)(\xi)
& = \mu^{-n-m} \de(a \cd x) \xi - \mu^{-m} \phi(a) \de(x) \xi \\
& = \mu^{-n-m} \de(a) \phi(x) \xi .
%& = \mu^{-n-k} \de(a) \phi(x) \io + \mu^{-n-k} \be^{-2}(a) \de(x) \io - \mu^{-n} \phi(a) \de(x) \io \\
 % = \mu^{-m} \de( \be^2(a)) \phi(x) \io .
\end{split}
\]
Consulting the proof of Lemma \ref{l:bounded}, we then get that 
\[
\begin{split}
U_{n + m} d_m(a)(x \ot \xi) & = U_{n + m} \T{ev}_\xi( [\Na,a](x) ) =  \Phi_{n + m}( [\Na,a](x) )(\xi) \\
& = \mu^{-n-m} \de(a) \phi(x)(\xi) = \mu^{-n-m} \de(a) U_m(x \ot \xi) . \qedhere
\end{split}
\]
\end{proof}

\subsection{Twisted Lipschitz triples and the modular lift}\label{ss:modular}
We continue our investigations in the setting explained in Subsection \ref{ss:commutator} and our goal is to obtain a description of the horizontal geometry of our data. In order to achieve this, for each $a \in \C A$, we need to tame the exponential growth of the operators $d_m(a) : G_m \to G$ (for $m \in \B Z$) expressed in Lemma \ref{l:deedel}. The core idea is to replace the horizontal lift $1 \hot_\Na D_0$ with the \emph{modular lift} obtained by appropriately rescaling the horizontal lift on each spectral subspace (in a way which depend on the parameter $\mu > 0$). The price to pay for applying this rescaling technique is that we obtain a twisted Lipschitz triple on $A$ instead of an ordinary Lipschitz triple. The ideas presented here are thus in line with the perturbed spectral triples appearing in \cite{CoMo:TST} and further developped in e.g. \cite{Mo:LIF,LaMa:TRS,GMR:UTS,PoWa:NGC,MaYu:RTS,CoTr:GBT,FhKh:SCN,Kaa:UKP}. Contrary to the situation described in \cite{CoMo:TST,PoWa:NGC} (for example) we are not perturbating our data by a positive invertible bounded operator but rather with a positive invertible \emph{unbounded} operator, see also \cite{Kaa:UKP,KaKy:DCQ}. The choice of this positive invertible unbounded operator is in fact dictated by the twisted Leibniz rule described in condition $(2)$ of Assumption \ref{a:exponential}.

As usual, we consider a unital $C^*$-algebra $A$ equipped with a strongly continuous action $\si$ of the unit circle. We let $\C A \su A$ be a norm-dense unital $*$-subalgebra such that $P_0(\C A) \su \C A$ and consider a Lipschitz triple $(\C A \cap A_0,H_0,D_0)$ on the fixed point algebra $A_0 \su A$. Our data is subject to the conditions stated in Assumption \ref{a:exponential}. We put $\C A_n := \C A \cap A_n = P_n(\C A)$ for all $n \in \B Z$.

Let us fix all of the extra data appearing in Assumption \ref{a:exponential}. In particular, we have the strongly continuous one-parameter group of unitaries $\{ V_s(\mu) \}_{s \in \B R}$ acting on the separable Hilbert space $G = X \hot_\rho H_0$, see \eqref{eq:grpuni}. Let us apply the notation $\Ga := V_{-i}(\mu)$ and record that this positive invertible unbounded operator satisfies that
\[
\Ga(\eta) = \mu^{\frac{n}{2}} \cd \eta
\]
whenever $\eta$ belongs to $G_n = A_n \hot_\rho H_0$ for some $n \in \B Z$. As described in Proposition \ref{p:grasstwist} (and the succeeding discussion), for every $n \in \B Z$ we have the Grassmann connection $\Na_n : \C A_n \to A_n \hot_\rho \B L(H_0)$ satisfying that $\Phi_n( \Na_n( x))(\xi) = \mu^{-n} \de(x) \xi$ for all $x \in \C A_n$ and $\xi \in H_0$. These different Grassmann connections can be put together into a single Hermitian $\de_0$-connection
\[
\Na : \C A \to X \hot_\rho \B L(H_0)
\]
satisfying that $\Na(x) = \Na_n(x)$ as soon as $x$ belongs to $\C A_n$ for some $n \in \B Z$. We refer to the discussion near \eqref{eq:direct} for more information on the associated horizontal lift $1 \ot_\Na D : \C A \ot_{\C A_0} \T{Dom}(D_0) \to X \hot_\rho H_0$.

\begin{dfn}\label{d:modular}
  The \emph{modular lift} $D_\Ga : \T{Dom}(D_\Ga) \to X \hot_\rho H_0$ is the selfadjoint unbounded operator defined as the closure of the essentially selfadjoint unbounded operator
  \[
\Ga^2(1 \ot_{\Na} D_0) : \C A \ot_{\C A_0} \T{Dom}(D_0) \to X \hot_\rho H_0 .
  \]
\end{dfn}

Let us spend a bit of extra time on understanding the modular lift $D_\Ga$. Recall from the discussion around \eqref{eq:isometry} that we have a unitary isomorphism of Hilbert spaces $U : G \to K$, where $K = \widehat{\bop}_{n \in \B Z} H_n$. This unitary isomorphism identifies $G_n = A_n \hot_\rho H_0$ with $H_n \su H$, for every $n \in \B Z$, via the unitary isomorphism $U_n : G_n \to H_n$. For each $n \in \B Z$ we put 
\[
\C H_n := U_n( \C A_n \ot_{\C A_0} \T{Dom}(D_0)) \su H_n
\]
and record that the algebraic direct sum $\C K := \bop_{n \in \B Z} \C H_n$ agrees with the norm-dense subspace $U( \C A \ot_{\C A_0} \T{Dom}(D_0)) \su K$.

\begin{lemma}\label{l:modular}
The selfadjoint unitary operator $U D_\Ga U^*$ is given explicitly by the formula  
  \[
(U D_\Ga U^*)( \phi(a) \xi) = \de(a) \xi + \mu^n \phi(a) D_0 \xi
  \]
 whenever $\xi \in \T{Dom}(D_0)$ and $a \in \C A_n$ for some $n \in \B Z$. 
\end{lemma}
\begin{proof}
  Let $a \in \C A_n$ and $\xi \in \T{Dom}(D_0)$ be given. Using Proposition \ref{p:grasstwist} (and the definition of $\Phi_n$ from \eqref{eq:defphi}) we get that
  \[
  \begin{split}
    (U D_\Ga U^*)(\phi(a) \xi)
    & = \mu^n U_n (1 \ot_{\Na_n} D_0)(a \ot \xi) = \mu^n U_n ( \T{ev}_\xi \Na_n(a) + a \ot D_0(\xi) ) \\
    & = \mu^n \Phi_n ( \Na_n(a))(\xi) + \mu^n \phi(a) D_0(\xi) = \de(a) \xi + \mu^n \phi(a) D_0 \xi . \qedhere
    \end{split}
  \]
\end{proof}

We may now apply the modular lift to capture the horizontal part of our geometric data in a single twisted Lipschitz triple. As usual, we suppress the injective unital $*$-homomorphism $\psi \hot 1 : A \to \B L(G)$.

In the situation where the Lipschitz triple $(\C A_0,H_0,D_0)$ is graded we denote the grading operator by $\ga_0 : H_0 \to H_0$. As explained in the beginning of Subsection \ref{ss:lifts}, this yields the selfadjoint unitary operator $1 \hot \ga_0$ on the interior tensor product $G = X \hot_\rho H_0$. 

\begin{thm}\label{t:twistlip}
  If the conditions in Assumption \ref{a:exponential} are satisfied, then it holds that $(\C A,G,D_\Ga)$ is a twisted Lipschitz triple on $A$ (with respect to $\{ V_s(\mu) \}_{s \in \B R}$). Moreover, letting $\de_\Ga : \C A \to \B L(G)$ denote the associated twisted $*$-derivation, for each $n,m \in \B Z$ we have the explicit formula
  \begin{equation}\label{eq:twistdelta}
U_{n + m} \de_\Ga(a)(\eta) = \mu^{-n/2} \de(a) U_m(\eta) \q \mbox{for all } a \in \C A_n \, \, \mbox{ and } \, \, \, \eta \in G_m \su G. 
  \end{equation}
  If the Lipschitz triple $(\C A_0,H_0,D_0)$ is graded, then the twisted Lipschitz triple $(\C A,G,D_\Ga)$ is also graded with grading operator $1 \hot \ga_0$. 
  %, where we specify that $H = X \hot_\rho H_0$ is equipped with the strongly continuous one-parameter group of unitaries {\blu $\{V_t(\mu)\}_{t \in \B R}$} from \eqref{eq:ellact}.
\end{thm}
\begin{proof}
Let $a \in \C A$. In order to get a twisted Lipschitz triple we only need to show that $a$ is a twisted Lipschitz operator with respect to $D_\Ga$ and $\be = \be(\mu)$. To this end, it suffices to take care of the case where $a \in \C A_n$ for some $n \in \B Z$.

We already observed that $a$ (and hence also $a^*$) is analytic of order $r$ for all $r \in \B R$ and that we have the formula $\be_z(a) = \mu^{i z \frac{n}{2}} \cd a$ for all $z \in \B C$. Let now $m \in \B Z$ and $\eta \in \C A_m \ot_{\C A_0} \T{Dom}(D_0)$ be given. Using Proposition \ref{p:anazet} we compute as follows (see also Lemma \ref{l:bounded}):
  \begin{equation}\label{eq:modtwist}
  \begin{split}
  D_\Ga \be_i(a)(\eta) & = \Ga (1 \ot_{\Na} D_0) \Ga \be_i(a) (\eta) = \Ga (1 \ot_\Na D_0) a \Ga(\eta) \\
  & =  \Ga d_m(a) \Ga(\eta) + \Ga a (1 \ot_\Na D_0) \Ga (\eta) \\
  & = \mu^{m + \frac{n}{2}} d_m(a)(\eta) + \be_{-i}(a) D_\Ga(\eta) .
  \end{split}
  \end{equation}

 For each $m \in \B Z$, we know from Lemma \ref{l:deedel} that
\begin{equation}\label{eq:deedel}
\mu^{m + \frac{n}{2}} U_{n + m} d_m(a)(\eta) = \mu^{-n/2} \de(a) U_m(\eta) \q \T{for all } \eta \in G_m
\end{equation}
In particular, we get the following operator norm estimate
\[
\mu^{m + \frac{n}{2}} \| d_m(a) \|_\infty \leq \mu^{-n/2} \| \de(a) \|_\infty ,
\]
where the right hand side is independent of $m \in \B Z$. Since we also know that $d_m(a) : G_m \to G$ factorizes through $G_{n + m}$ we get a well-defined bounded operator $\de_\Ga(a) : G \to G$ satisfying that $\de_\Ga(a)(\eta) = \mu^{m + \frac{n}{2}} d_m(a)(\eta)$ whenever $\eta \in G_m$ for some $m \in \B Z$. Indeed, the operator norm of $\de_\Ga(a)$ can be estimated as follows:
\[
\| \de_\Ga(a) \|_\infty \leq \mu^{-n/2} \| \de(a) \|_\infty .
\]

Since we know from \eqref{eq:modtwist} that $\de_\Ga(a)(\eta) = ( D_\Ga \be_i(a) - \be_{-i}(a) D_\Ga)(\eta)$ for all $\eta$ belonging to the core $\C A \ot_{\C A_0} \T{Dom}(D_0)$ for the modular lift $D_\Ga$, we have now established that $(\C A, G, D_\Ga)$ is a twisted Lipschitz triple.

The identity in \eqref{eq:twistdelta} is an immediate consequence of \eqref{eq:deedel} and the fact that the restriction of $\de_\Ga(a)$ to $G_m$ agrees with $\mu^{m + \frac{n}{2}} d_m(a)$ for all $m \in \B Z$.  

%{\blu talk about cores in relation to Definition \ref{d:twistop}}.

Suppose now that $\ga_0 : H_0 \to H_0$ is a grading operator for the Lipschitz triple $(\C A_0,H_0,D_0)$. To establish that $1 \hot \ga_0 : G \to G$ is a grading operator for the twisted Lipschitz triple $(\C A,G,D_\Ga)$, the only non-trivial thing to prove is that $1 \hot \ga_0$ anti-commutes with the modular lift $D_\Ga$. To this end, we recall that our Hermitian $\de_0$-connection $\Na : \C A \to X \hot_\rho \B L(H_0)$ is formed out of Grassmann connections and it therefore follows from the discussions in Subsection \ref{ss:lifts} that the horizontal lift $1 \ot_\Na D_0$ anti-commutes with $1 \hot \ga_0$. Since the positive invertible unbounded operator $\Ga$ clearly commutes with $1 \hot \ga_0$ we conclude that the modular lift $D_\Ga$ indeed anti-commutes with $1 \hot \ga_0$.
\end{proof}

As described after Lemma \ref{l:sliplower}, the twisted Lipschitz triple $(\C A,G, D_\Ga)$ gives rise to the twisted Lipschitz algebra $\Lip_{\T{hor}}(A) := \Lip_{D_\Ga}^\be(A) \su A$ and we refer to this norm-dense unital $*$-subalgebra as the \emph{horizontal Lipschitz algebra}. Likewise, we apply the notation $L_{\T{hor}} : \Lip_{\T{hor}}(A) \to [0,\infty)$ for the slip-norm defined by 
\[
L_{\T{hor}}(a) := L_{D_\Ga}^\be(a) := \| \de_\Ga(a) \|_\infty .
\]
It is not hard to see that the modular lift $D_\Ga$ commutes with the unitary operator $V_s(\mu)$ for all $s \in \B R$. It therefore follows from Lemma \ref{l:sliplower} that the slip-norm $L_{\T{hor}}$ is lower semicontinuous.

In the next lemma, we are relating the slip-norm $L_{D_0}  : \Lip_{D_0}(A_0) \to [0,\infty)$ coming from the Lipschitz triple $(\C A_0, H_0,D_0)$ to the horizontal slip-norm $L_{\T{hor}} : \Lip_{\T{hor}}(A) \to [0,\infty)$. 

  \begin{lemma}\label{l:inequality}
If $a \in \Lip_{\T{hor}}(A) \cap A_0$, then it holds that $a \in \Lip_{D_0}(A_0)$ and we have the inequality $L_{D_0}(a) \leq L_{\T{hor}}(a)$.
  \end{lemma}
  \begin{proof}
    Let $a \in \Lip_{\T{hor}}(A) \cap A_0$ and $\xi,\eta \in \T{Dom}(D_0)$ be given. Suppressing the injective unital $*$-homomorphism $\psi \hot 1 : A \to \B L(G)$ we first observe that $\be_i(a) = a = \be_{-i}(a)$. If we moreover suppress the inclusion of $H_0$ into $K$ we may apply Lemma \ref{l:modular} to compute as follows:
    \[
    \begin{split}
      & \inn{D_0 \eta, \rho(a) \xi} - \inn{ \eta, \rho(a) D_0 \xi}
    = \inn{U D_\Ga U^* \eta, \rho(a) \xi} - \inn{\eta, \rho(a) U D_\Ga U^* \xi} \\
    & \q = \inn{ D_\Ga U^* \eta, \be_i(a) U^* \xi} - \inn{\eta, U \be_{-i}(a) D_\Ga U^* \xi} 
    = \inn{\eta, U \de_\Ga(a) U^* \xi} .
    \end{split}
    \]
An application of Theorem \ref{t:weak} now proves the lemma.
 \end{proof}

  In the final lemma of this section we study the invariance properties of the horizontal slip-norm $L_{\T{hor}} : \Lip_{\T{hor}}(A) \to [0,\infty)$ with respect to the circle action $\si : S^1 \ti A \to A$.
  
  \begin{lemma}\label{l:invariance}
If $a \in \Lip_{\T{hor}}(A)$, then $\si_\la(a) \in \Lip_{\T{hor}}(A)$ and $L_{\T{hor}}\big( \si_\la(a)\big) = L_{\T{hor}}(a)$ for all $\la \in S^1$.
  \end{lemma}
  \begin{proof}
    Let $t \in \B R$. As in Section \ref{s:vertical} we put $U_t := e^{it(N \hot 1)} : G \to G$ and record that $U_t(\ze) = e^{itn} \cd \ze$ whenever $\ze \in G_n$ for some $n \in \B Z$. It then holds that $U_t$ commutes with the modular lift $D_\Ga$ and with $V_s(\mu)$ for all $s \in \B R$. Since $U_t$ moreover implements the $*$-automorphism $\si_{e^{it}} : A \to A$ via conjugation, see \eqref{eq:implement}, we obtain the result of the lemma. Notice in this respect that 
\[
\de_\Ga( \si_{e^{it}}(a)) = U_t \de_\Ga(a) U_t^* \q \T{for all } a \in \Lip_{\T{hor}}(A) . \qedhere
\]
\end{proof}

  \section{Second main theorem}\label{s:second}
We are now ready to prove our second main theorem establishing that the vertical and the horizontal geometric input from Section \ref{s:vertical} and Section \ref{s:horizontal} can be combined into a compact quantum metric space structure on $A$. This theorem works under the extra condition that the input for the horizontal part of the geometry turns the fixed point algebra $A_0$ into a compact quantum metric space. 

Let us consider a strongly continuous action $\si$ of the unit circle on a unital $C^*$-algebra $A$ together with a norm-dense unital $*$-subalgebra $\C A \su A$. We suppose that $P_0(\C A) \su \C A$ and that $(\C A \cap A_0, H_0, D_0)$ is a Lipschitz triple on the fixed point algebra $A_0$. As usual, denote the corresponding injective unital $*$-homomorphism and the $*$-derivation by $\rho : A_0 \to \B L(H_0)$ and $\de_0 : \C A \cap A_0 \to \B L(H_0)$, respectively.

Suppose that the conditions in Assumption \ref{a:exponential} are satisfied and fix all of the data appearing in there. Defining the separable Hilbert space $G := X \hot_\rho H_0$, we recall from Theorem \ref{t:twistlip} that the modular lift $D_\Ga : \T{Dom}(D_\Ga) \to G$ fits in a twisted Lipschitz triple $(\C A, G, D_\Ga)$, where the twist comes from the strongly continuous one-parameter group of unitaries $\{ V_s(\mu) \}_{s \in \B R}$ defined in \eqref{eq:grpuni}.

Recall from Section \ref{s:vertical} and Subsection \ref{ss:modular} that the notations
\[
\Lip_{\T{ver}}(A) = \Lip_{S^1}(A) = \Lip_{N \hot 1}(A) \, \, \T{ and } \, \, \, \Lip_{\T{hor}}(A) = \Lip_{D_\Ga}^\be(A) .
\]
refer to the vertical Lipschitz algebra and the horizontal Lipschitz algebras, respectively. See also Theorem \ref{t:vertical}. 

\begin{dfn}
We define the \emph{total Lipschitz algebra} $\Lip_{\T{tot}}(A) := \Lip_{\T{ver}}(A) \cap \Lip_{\T{hor}}(A)$ as the intersection of the vertical and the horizontal Lipschitz algebras. 
\end{dfn}

The total Lipschitz algebra $\Lip_{\T{tot}}(A)$ is a norm-dense unital $*$-subalgebra of $A$ and it comes equipped with the $*$-derivation
\[
\de_{\T{ver}} : \Lip_{\T{tot}}(A) \to \B L(G) \q \de_{\T{ver}}(a) := \T{cl}\big( [ N \hot 1, a] \big)
\]
and the twisted $*$-derivation
\[
\de_{\T{hor}} := \de_\Ga : \Lip_{\T{tot}}(A) \to \B L(G) \q \de_{\T{hor}}(a) := \T{cl}\big( D_\Ga \be_i(a) - \be_{-i}(a) D_\Ga \big) .
\]
The associated lower semicontinuous slip-norms on $\Lip_{\T{tot}}(A)$ are denoted by $L_{\T{ver}}$ and $L_{\T{hor}}$, respectively.

Let us explain how to assemble our data into an interesting slip-norm on the total Lipschitz algebra $\Lip_{\T{tot}}(A)$. There are two situations to take care of depending on whether the Lipschitz triple $(\C A \cap A_0,H_0,D_0)$ is graded or ungraded.

In the graded case, we denote the grading operator by $\ga_0 : H_0 \to H_0$ and define the slip-norm
\begin{equation}\label{eq:totgra}
L_{\T{tot}} : \Lip_{\T{tot}}(A) \to [0,\infty) \q L_{\T{tot}}(a) := \big\| (1 \hot \ga_0)\de_{\T{ver}}(a) + \de_{\T{hor}}(a) \big\|_\infty .
\end{equation}
In the ungraded case, the slip-norm $L_{\T{tot}} : \Lip_{\T{tot}}(A) \to [0,\infty)$ is defined by the formula
\begin{equation}\label{eq:totungra}
L_{\T{tot}}(a) := \max\big\{ \|\de_{\T{ver}}(a) + i \de_{\T{hor}}(a) \|_\infty, \| \de_{\T{ver}}(a) - i \de_{\T{hor}}(a) \|_\infty \big\} .
\end{equation}

As we shall see in Section \ref{s:unbKK}, the construction of the slip-norm $L_{\T{tot}}$ is heavily inspired by considerations coming from unbounded $KK$-theory.

%Moreover, the three results here below apply to $L_{\T{tot}}$ 
%ith a slip-norm 
%
%The corresponding injective unital $*$-homomorphism and $*$-derivation are denoted by $\rho : A_0 \to \B L(H_0)$ and $\de_0 : \C A_0 \to \B L(H_0)$, respectively.

\begin{thm}\label{t:mainII}
  Suppose that the conditions in Assumption \ref{a:exponential} are satisfied. Let $\C L \su \Lip_{\T{tot}}(A)$ be a unital $*$-subalgebra with $\si_\la(\C L) \su \C L$ for all $\la \in S^1$ and satisfying that
  \[
\C A \su \C L \, \, \mbox{ and } \, \, \, P_0(\C L) \su \C L .
\]
If the Lipschitz triple $(\C L \cap A_0, H_0, D_0)$ is metric, then it holds that $(\C L,L_{\T{tot}})$ is a compact quantum metric space.
\end{thm}

Before embarking on the proof, notice that the conditions of the theorem together with Lemma \ref{l:inequality} imply that the norm-dense unital $*$-subalgebra $P_0(\C L) = \C L \cap A_0 \su A_0$ is contained in $\Lip_{D_0}(A_0)$. The triple $(\C L \cap A_0, H_0,D_0)$ is thus indeed a Lipschitz triple.

%To see this, notice that $P_0(\C L) \su \C L \su \Lip_{\T{hor}}(A)$ and hence that $P_0(\C L) \su \Lip_{\T{hor}}(A) \cap A_0$, where the latter unital $*$-subalgebra of $A_0$ is contained in $\Lip_{D_0}(A_0)$ by . 

\begin{proof}
  We are going to apply Theorem \ref{t:main} for the two slip-norms $L_{\T{tot}}$ and $L_{\T{hor}} : \C L \to [0,\infty)$ together with the algebra homomorphism $\be_i := \be_i(\mu) : \C L \to \B L(G)$ coming from the family $\{ V_s(\mu) \}_{s \in \B R}$. Let us therefore go through the six conditions appearing in the statement of Theorem \ref{t:main}.

  We know by assumption on $\C L$ that condition $(1)$ and the first half of condition $(2)$ are satisfied. Moreover, Assumption \ref{a:exponential} $(3)$ and the fact that $\C A \su \C L$ ensure that the second half of condition $(2)$ holds as well. 

   Our next step is to show that condition $(3)$ is satisfied for the constant $C = 1$. To this end, we first observe that the slip-norms $L_\ell$ and $L_{\T{ver}}$ agree on $\Lip_{\T{tot}}(A) \su \Lip_{\T{ver}}(A)$ by Theorem \ref{t:vertical}. This identity of slip-norms immediately entails our claim in the ungraded case. In the graded case we also need to observe the identities 
      \[
      \begin{split}
      L_{\T{tot}}(a) & = \big\| (1 \hot \ga_0) \de_{\T{ver}}(a) + \de_{\T{hor}}(a) \big\|_\infty
      = \big\| \de_{\T{ver}}(a) (1 \hot \ga_0) + (1 \hot \ga_0) \de_{\T{hor}}(a) (1 \hot \ga_0) \big\|_\infty \\
      & = \big\| (1 \hot \ga_0) \de_{\T{ver}}(a) - \de_{\T{hor}}(a) \big\|_\infty,
      \end{split}
      \]
      which are valid for all $a \in \Lip_{\T{tot}}(A)$. Remark here that the selfadjoint unitary operator $1 \hot \ga_0$ commutes with $\de_{\T{ver}}(a)$ and anti-commutes with $\de_{\T{hor}}(a)$, see Theorem \ref{t:twistlip}. %since $1 \hot \ga_0$ commutes with both $N \hot 1$ and $\psi(a) \hot 1$.

      The fact that $L_{\T{hor}} : \C L \to [0,\infty)$ satisfies the twisted Leibniz inequality in condition $(4)$ follows since $(\C L, G,D_\Ga)$ is a twisted Lipschitz triple with respect to the family $\{ V_s(\mu)\}_{s \in \B R}$, see Theorem \ref{t:twistlip} and the general discussion in Subsection \ref{ss:twistedlip}. Since $\be_i(a) = a$ for all $a \in A_0$ (see \eqref{eq:beta}) we also clearly have that $\be_i : \C L \cap A_0 \to \B L(G)$ extends to a unital $*$-homomorphism (namely $\psi \hot 1 : A_0 \to \B L(G)$). 

Regarding condition $(5)$, we already know from Lemma \ref{l:sliplower} that $L_{\T{hor}}$ is lower semicontinuous (recalling that the modular lift $D_\Ga$ commutes with $V_s(\mu)$ for all $s \in \B R$) and the remaining invariance property was already proved in Lemma \ref{l:invariance}.

To verify condition $(6)$, recall that $(\C L \cap A_0, L_{D_0})$ is a compact quantum metric space by assumption and the inequality in Lemma \ref{l:inequality} therefore entails that $(\C L \cap A_0, L_{\T{hor}})$ is a compact quantum metric space as well, see Theorem \ref{t:charac}.
\end{proof}

Let us single out two important corollaries of our second main theorem. The result of the first corollary follows from Theorem \ref{t:mainII} in combination with Lemma \ref{l:algebra}. Recall also that the inclusion $\C A \su \Lip_{\T{tot}}(A)$ is a consequence of Theorem \ref{t:vertical} and Theorem \ref{t:twistlip}. The proof of the second corollary requires a bit more work.

\begin{cor}\label{c:coordcqms}
Suppose that the conditions in Assumption \ref{a:exponential} are satisfied. If the Lipschitz triple $(\C A \cap A_0, H_0, D_0)$ is metric, then it holds that $(\C A,L_{\T{tot}})$ is a compact quantum metric space.
\end{cor}

\begin{cor}\label{c:lipcqms}
Suppose that the conditions in Assumption \ref{a:exponential} are satisfied. If the Lipschitz triple $( \Lip_{D_0}(A_0), H_0, D_0)$ is metric, then it holds that $( \Lip_{\T{tot}}(A),L_{\T{tot}})$ is a compact quantum metric space.
\end{cor}
\begin{proof}
  We are going to apply Theorem \ref{t:mainII}. We already know from Lemma \ref{l:invariance} that $\si_\la\big( \Lip_{\T{hor}}(A) \big) \su \Lip_{\T{hor}}(A)$ and it follows from Theorem \ref{t:vertical} that $\si_\la\big( \Lip_{\T{ver}}(A) \big) \su \Lip_{\T{ver}}(A)$ for all $\la \in S^1$, see also Subsection \ref{ss:li}. Moreover, we already observed that $\C A \su \Lip_{\T{tot}}(A)$. The fact that the Lipschitz triple $( \Lip_{\T{tot}}(A) \cap A_0, H_0, D_0)$ is metric follows from our assumptions in combination with Lemma \ref{l:inequality} and Theorem \ref{t:charac}. We are therefore left with proving that $P_0\big( \Lip_{\T{tot}}(A) \big) \su \Lip_{\T{tot}}(A)$. To this end, remark that $P_0(A) = A_0 \su \Lip_{\T{ver}}(A)$ so it suffices to show that $P_0\big( \Lip_{\T{hor}}(A) \big) \su \Lip_{\T{hor}}(A)$.

Let $a \in \Lip_{\T{hor}}(A)$ be given. For each $n \in \B Z$, let $Q_n : G \to G$ denote the orthogonal projection with image $G_n \su G$. Notice then that we have the identities
\[
Q_n \be_i(a) Q_n =  P_0(a) Q_n = Q_n a Q_n = Q_n \be_{-i}(a) Q_n
\]
and that the modular lift $D_\Ga$ commutes with $Q_n$. It also clearly holds that $P_0(a)$ belongs to $\Ana_\be^r(G)$ for all $r \in \B R$ and satisfies that $\be_z( P_0(a)) = P_0(a)$ for all $z \in \B C$. 
%
%This orthogonal projection is given explicitly by the formula
 % \[
%Q_n(\eta) = \frac{1}{2\pi} \int_0^{2\pi} e^{-itn} (U_t \hot 1)(\eta) \, dt .
%\]
%
% $Q_n\big( \T{Dom}(D_\Ga) \big) \su \T{Dom}(D_\Ga)$ and the commutator $[D_\Ga, Q_n]$ is equal to zero on $\T{Dom}(D_\Ga)$.

Let now $n \in \B Z$ and suppose that $\eta$ belongs to $\C A_n \ot_{\C A_0} \T{Dom}(D_0)$. Our observations yield that
$\be_i( P_0(a) ) \eta = P_0(a) \eta = Q_n \be_i(a) Q_n \eta$ belongs to $\T{Dom}(D_\Ga)$ and that we may compute as follows:
\[
\begin{split}
D_\Ga  \be_i( P_0(a) )  \eta 
& = Q_n D_\Ga \be_i(a) Q_n \eta = Q_n \de_\Ga(a) Q_n \eta
+ Q_n \be_{-i}(a) Q_n D_\Ga \eta \\
& = Q_n \de_\Ga(a) Q_n \eta + \be_{-i}(P_0(a)) D_\Ga \eta .
\end{split}
\]

This entails that $P_0(a) \in \Lip_{\T{hor}}(A)$ and that we have the formula
\[
\de_{\T{hor}}( P_0(a)) Q_n = Q_n \de_{\T{hor}}(a) Q_n \q \T{for all } n \in \B Z . \qedhere
\]
%Choose an $N \in \B N$ such that $\eta = \sum_{n = 1}^N ( Q_n \eta + Q_{-n} \eta) + Q_0(\eta)$. 
\end{proof}

\section{Relationship with unbounded $KK$-theory}\label{s:unbKK} 
In this section we explain how our constructions relate to unbounded $KK$-theory and the unbounded Kasparov product as developed in \cite{Mes:UCN,KaLe:SFU,MeRe:NST}, relying on earlier work in \cite{BaJu:TBK,Con:GCM,Kuc:PUM}. For more details on $KK$-theory and the internal Kasparov product we refer the reader to the seminal papers \cite{Kas:OFE,CoSk:LIF} as well as the book by Bruce Blackadar, \cite{Bla:KOA}. 

As usual we consider a unital $C^*$-algebra $A$ equipped with a strongly continuous action $\si$ of the unit circle and we let $\C A \su A$ be a norm-dense unital $*$-subalgebra satisfying that $P_0(\C A) \su \C A$. Let us moreover assume that $(\C A \cap A_0,H_0,D_0)$ is a unital spectral triple on the fixed point algebra $A_0$. Our data is required to satisfy the conditions appearing in Assumption \ref{a:exponential}. Recall that $\C A_n := \C A \cap A_n$ for all $n \in \B Z$. 

An application of Theorem \ref{t:unbkas} shows that the triple $(\C A,X,N)$ is an unbounded Kasparov module from $A$ to $A_0$ and we therefore have a class in odd $KK$-theory $[X,F_N] \in KK_1(A,A_0)$ obtained by applying the Baaj-Julg bounded transform, see \cite[Proposition 17.11.3]{Bla:KOA}. Likewise, our unital spectral triple $(\C A_0,H_0,D_0)$ yields a class in $K$-homology $[H_0,F_{D_0}] \in K^p(A_0)$, where $p = 0$ in the graded case and $p = 1$ in the ungraded case. Assuming that $A_0$ is separable (entailing that $A$ is also separable) we therefore have the class $[X,F_N] \hot_{A_0} [H_0,F_{D_0}]$ in the $K$-homology of $A$, obtained by taking the internal Kasparov product of our two classes. The parity $q \in \{0,1\}$ of this class is equal to $p + 1$ modulo $2$. 

One of the tasks of unbounded $KK$-theory is to represent the class $[X,F_N] \hot_{A_0} [H_0,F_{D_0}] \in K^q(A)$ by an explicit unital spectral triple. This can be done in the special case where the constant $\mu > 0$ from Assumption \ref{a:exponential} is exactly equal to one. We describe the main construction without this assumption and indicate exactly where the condition $\mu = 1$ is crucial.

Let us emphasize that for $\mu \neq 1$ our data does not yield a unital spectral triple over $\C A$ (nor a twisted unital spectral triple for that matter). In fact, it appears that there is no single selfadjoint unbounded operator which adequately describes the spectral geometry of $\C A$ in this situation. Instead, we have to work with a pair of selfadjoint unbounded operators where one of them comes from the circle action and the other one comes from the unital spectral triple over the corresponding fixed point algebra. It is therefore also unclear how to obtain a class in the $K$-homology of $A$ by taking a bounded transform since we do not know which selfadjoint unbounded operator to start out with. 
\medskip

Recall from the discussion near \eqref{eq:isometry} that the separable Hilbert space $G := X \hot_\rho H_0$ is unitarily isomorphic to the Hilbert space direct sum $K = \widehat{\bop}_{n \in \B Z} H_n$, where the component $H_n$ is defined as the closure of $\phi(A_n) H_0$ inside $H$. In the following, we are suppressing this unitary isomorphism $U$ but notice that $U( x \ot \xi) = \phi(x) \xi$ whenever $x \in A_n$ for some $n \in \B Z$ and $\xi \in H_0$. 

We also remind the reader (see the discussion before Lemma \ref{l:modular}) that the modular lift $D_\Ga$ has the algebraic direct sum $\C K = \op_{n \in \B Z} \C H_n$ as a core, where the component $\C H_n$ is defined by
\[
\C H_n := \phi(\C A_n) \T{Dom}(D_0) = U_n( \C A_n \ot_{\C A_0} \T{Dom}(D_0)) .
\]
The modular lift is determined by the explicit formula
\[
D_\Ga( \phi(x) \xi) = \de(x) \xi + \mu^n \phi(x) D_0 \xi
\]
whenever $x \in \C A_n$ for some $n \in \B Z$ and $\xi \in \T{Dom}(D_0)$, see Lemma \ref{l:modular}. We apply the notation $\C D_\Ga$ for the restriction of the modular lift $D_\Ga$ to the core $\C K$.
%
%We are moreover interested in the positive invertible unbounded operator $\Ga : \T{Dom}(\Ga) \to K$ satisfying that $\Ga(\eta) = \mu^{\frac{n}{2}} \eta$ whenever $\eta \in H_n$ for some $n \in \B Z$. The modular lift $D_\Ga$ then has the algebraic direct sum $\op_{n \in \B Z} \C H_n$ as a core and satisfies that
%\[
%D_\Ga( \phi(x) \xi) = \Ga^2(1 \ot_\Na D_0)(\phi(x) \xi) = \de(x) \xi + \mu^n \phi(x) D_0 \xi
%\]
%as soon as $x \in \C A_n$ for some $n \in \B Z$ and $\xi \in \T{Dom}(D_0)$.
%
In this picture, the selfadjoint unbounded operator $N$ agrees with the closure of the essentially selfadjoint unbounded operator $\C N : \op_{n \in \B Z} H_n \to K$ satisfying that $\C N(\eta) = n \cd \eta$ whenever $\eta \in H_n$ for some $n \in \B Z$. Clearly, $\C K \su \T{Dom}(N)$ is a core for $N$.
\medskip

There are now two cases to consider depending on whether the unital spectral triple $(\C A_0,H_0,D_0)$ is graded or ungraded.

In the graded case (where $p = 0$), we consider the Hilbert space $K_0 := K = \widehat{\bop}_{n \in \B Z} H_n$ and lift the grading operator $\ga_0 : H_0 \to H_0$ to the selfadjoint unitary operator $\ga := 1 \hot \ga_0 : K \to K$ satisfying that $\ga( \phi(x) \xi) = \phi(x) \ga_0 \xi$ whenever $x \in A_n$ for some $n \in \B Z$ and $\xi \in H_0$. The \emph{unbounded product operator} $N \ti_{\Na} D_\Ga$ is defined as the closure of the symmetric unbounded operator
$\C N \ga + \C D_\Ga : \C K \to K$. We specify that the unbounded product operator is given by the formula
\[
(N \ti_{\Na} D_\Ga)(\phi(x) \xi) = n \phi(x) \ga_0 \xi + \de(x) \xi + \mu^n \phi(x) D_0 \xi
\]
whenever $x \in \C A_n$ for some $n \in \B Z$ and $\xi \in \T{Dom}(D_0)$.

In the ungraded case (where $p = 1$), we consider the Hilbert space $K_1 := K \op K$ which is graded by the selfadjoint unitary operator $\ga := \ma{cc}{1 & 0 \\ 0 & - 1}$. The \emph{unbounded product operator} $N \ti_{\Na} D_\Ga$ is defined as the closure of the symmetric unbounded operator
\[
\ma{cc}{ 0 & \C N +  i \C D_\Ga \\ \C N - i \C D_\Ga & 0}
: \C K \op \C K \to K \op K .
\]

In the next proposition we investigate the properties of the unbounded product operators. It is this respect important to clarify that the part of the proposition concerning selfadjointness is in fact a consequence of much more general results regarding weakly anti-commuting operators, see e.g. \cite[Theorem 2.6]{LeMe:SSH}. We provide an elementary proof based on Theorem \ref{t:esself}.

\begin{prop}\label{p:almostspec}
  Suppose that the conditions in Assumption \ref{a:exponential} are satisfied and that $(\C A_0,H_0,D_0)$ is a unital spectral triple on the fixed point algebra $A_0$. It then holds that the unbounded product operator $N \ti_{\Na} D_\Ga$ is selfadjoint and the resolvent $(i + N \ti_\Na D_\Ga)^{-1}$ is a compact operator. 
\end{prop}
\begin{proof}
  Let us focus on the graded case (since similar methods apply in the ungraded case).

  For $n \in \B Z$ we apply the notation $(D_\Ga)_n$ for the restriction of the modular lift $D_\Ga$ to the subspace $\C H_n \su \T{Dom}(D_\Ga)$. It then holds that $(D_\Ga)_n = \mu^n (1 \ot_{\Na_n} D_0)$ where $\Na_n : \C A_n \to A_n \hot_\rho \B L(H_0)$ is a Grassmann connection satisfying that $\Phi_n \Na_n(x)(\xi) = \mu^{-n} \de(x) \xi$ for all $x \in \C A_n$, see Proposition \ref{p:grasstwist} and Definition \ref{d:modular}. It therefore follows from Theorem \ref{t:esself} that $(D_\Ga)_n : \C H_n \to H_n$ is selfadjoint and has compact resolvent. We also observe that the selfadjoint unitary operator $\ga = 1 \hot \ga_0$ induces a selfadjoint unitary operator $\ga_n : H_n \to H_n$. The restriction of the unbounded product operator to $\C H_n$ then agrees with $n \ga_n + (D_\Ga)_n : \C H_n \to H_n$ and we may thus conclude that this restriction is selfadjoint and has compact resolvent.

  Since an (infinite) algebraic direct sum of selfadjoint unbounded operators is essentially selfadjoint we get that the unbounded product operator is selfadjoint.

  In order to show that the resolvent of the unbounded product operator is compact it now suffices to establish that
  \begin{equation}\label{eq:limit}
\lim_{n \to \pm \infty} \big\| (i + n\ga_n + (D_\Ga)_n)^{-1} \big\|_\infty = 0 .
\end{equation}
Let us again fix an $n \in \B Z$. Using that $(D_\Ga)_n$ anti-commutes with the selfadjoint unitary operator $\ga_n : H_n \to H_n$ we obtain that $\big( n \ga_n + (D_\Ga)_n \big)^2 = n^2 + (D_\Ga)_n^2$. This entails the estimate 
\[
\big\| (i + n\ga + (D_\Ga)_n)^{-1} \big\|_\infty^2 = \big\| (1 + n^2 + (D_\Ga)_n^2)^{-1} \big\|_\infty \leq (1 + n^2)^{-1} ,
\]
which in turn implies the identity in \eqref{eq:limit}. 
\end{proof}

The above result indicates that the unbounded product operator could be part of a unital spectral triple with coordinate algebra $\C A$. The main problem is however that the commutator condition fails (in general) for $\mu \neq 1$. This happens because the vertical leg, represented by the selfadjoint unbounded operator $N$, has bounded commutators with elements in $\C A$ whereas the horizontal leg, represented by the modular lift $D_\Ga$, only has twisted bounded commutators.

For the rest of this section we are therefore focusing on the special case where $\mu = 1$. In this situation, we apply the notation $N \ti_\Na D_0$ for the unbounded product operator. Our next theorem is in fact a consequence of the work of Bram Mesland and Adam Rennie in \cite[Theorem 4.4]{MeRe:NST}. However, in order to keep the current text independent of the operator space techniques applied in \cite{MeRe:NST} we present a short argument referring back to results of Dan Kucerovsky, \cite{Kuc:PUM}.

\begin{thm}\label{t:unbddkasp}
  Suppose that the conditions in Assumption \ref{a:exponential} are satisfied for the constant $\mu = 1$ and that $(\C A_0,H_0,D_0)$ is a unital spectral triple on the fixed point algebra $A_0$. Letting $p \in \{0,1\}$ reflect the grading, it holds that $(\C A,K_p, N \ti_\Na D_0)$ is a unital spectral triple and for $A_0$ separable we have the identity $[K, F_{N \ti_\Na D_0}] = [X,F_N] \hot_{A_0} [H_0, F_{D_0}]$ inside the $K$-homology group $K^q(A)$ (where $q = 1$ for $p = 0$ and $q = 0$ for $p = 1$). %If $p = 1$, then it holds that $(\C A,K \op K, N \ti_\Na D_0)$ is a graded unital spectral triple and for $A_0$ separable we have the identity $[K \op K, F_{N \ti_\Na D_0}] = [X,F_N] \hot_{A_0} [H_0, F_{D_0}]$ inside the $K$-homology group $K^0(A)$.
\end{thm}
\begin{proof}
  Let us focus on the graded case (where $p = 0$). The ungraded case follows by similar arguments. We already know from Proposition \ref{p:almostspec} that the unbounded product operator $N \ti_\Na D_0$ is selfadjoint and has compact resolvent. Since $\mu = 1$ we moreover get from Theorem \ref{t:unbkas} and Theorem \ref{t:twistlip} that $\C A$ is contained in the intersection of the two Lipschitz algebras $\Lip_{N \ga}(K) \cap \Lip_{1 \hot_\Na D_0}(K)$. Since $\C A$ preserves the core $\C K$ for the unbounded product operator $N \ti_\Na D_0$ we get that $\C A$ is in fact contained in the Lipschitz algebra $\Lip_{N \ti_\Na D_0}(K)$ so that $(\C A,K, N \ti_\Na D_0)$ is indeed a unital spectral triple.

  To continue, we verify the three conditions appearing in \cite[Theorem 13]{Kuc:PUM} using the interior tensor product description of the Hilbert space $K \cong X \hot_\rho H_0$ when convenient.

  For the first condition, we use the norm-dense subspace $\C X = \La(\C A) \su X$ and for each $x \in \C X$ observe that $(1 \hot_\Na D_0) \te_x(\xi) - \te_x D_0(\xi) =  \T{ev}_\xi( \Na(x))$ for all $\xi \in \T{Dom}(D_0)$. The point is now that the operator from $H_0$ to $X \hot_\rho H_0$ sending $\xi$ to $\T{ev}_\xi( \Na(x))$ is bounded with operator norm dominated by $\| \Na(x) \|$.

  For the second condition, it suffices to show that $\T{Dom}(N \ti_\Na D_0)$ is contained in $\T{Dom}(N \ga)$. Let $\{\eta_n\}_{n \in \B Z}$ be a vector in $K = \widehat{\bop}_{n \in \B Z} H_n$. It can then be verified that the vector $\{\eta_n\}_{n \in \B Z}$ belongs to $\T{Dom}(N \ti_\Na D_0)$ if and only if $\eta_n \in \C H_n$ for all $n \in \B Z$ and the two series here below converges:
  \[
\sum_{n = -\infty}^\infty n^2 \| \eta_n \|^2 \, \, \T{ and } \, \, \, \sum_{n = -\infty}^\infty \| (1 \ot_{\Na} D_0) \eta_n \|^2 .
\]
On the other hand, it holds that $\{\eta_n\}_{n \in \B Z}$ belongs to $\T{Dom}(N \ga)$ if and only if the series $\sum_{n = -\infty}^\infty n^2 \|\eta_n\|^2$ converges. 

For the third condition, we consider an arbitrary vector $\xi$ in $\T{Dom}(N \ti_\Na D_0)$ and observe that
\[
\begin{split}
\binn{ (N \ti_\Na D_0)\xi, N \ga \xi} + \binn{N \ga \xi, (N \ti_\Na D_0)\xi} 
= \binn{N \ga \xi, N \ga \xi} \geq 0 . \qedhere
\end{split}
\]
\end{proof}

The next result follows immediately from the above Theorem \ref{t:unbddkasp} in combination with Corollary \ref{c:coordcqms}. Recall the notion of a spectral metric space from Definition \ref{d:spemet} and the definition of the slip-norm $L_{\T{tot}} : \C A \to [0,\infty)$ from \eqref{eq:totgra} and \eqref{eq:totungra}. 

\begin{thm}
Suppose that the conditions in Assumption \ref{a:exponential} are satisfied for the constant $\mu = 1$ and that $(\C A_0,H_0,D_0)$ is a spectral metric space. It then holds that $(\C A,K_p, N \ti_\Na D_0)$ is a spectral metric space and the two slip-norms $L_{N \ti_\Na D_0} : \C A \to [0,\infty)$ and $L_{\T{tot}} : \C A \to [0,\infty)$ agree. 
\end{thm}

%For each $n \in \B Z$ we fix the Hermitian $\de_0$-connection $\Na_n : \La(\C A_n) \to X_n \hot_\rho \B L(H_0)$ such that
%\[
%\Phi_n \Na_n = \mu^{-n} \de(x) \io : \C X_n \to \B L(H_0,H) ,
%\]
%where we recall that $\io : H_0 \to H$ is the inclusion and $\de : \C A \to \B L(H)$ is the twisted derivation appearing in Assumption \ref{a:exponential}. The isometry $\Phi_n : X_n \hot_\rho \B L(H_0) \to \B L(H_0,H)$ is described explicitly in Lemma \ref{l:isometry}. Let us denote the associated horizontal lift by
%\[
%1 \ot_\Na D_0 : \C X \ot_{\C A_0} \T{Dom}(D_0) \to X \hot_\rho H_0 .
%\]

\section{The quantum spheres and the quantum projective spaces}\label{s:quasph}
In this last section, we shall see how our framework applies to the higher Vaksman-Soibelman quantum spheres introduced in \cite{VaSo:AFQ}. In particular, our constructions provide each of the higher $q$-deformed spheres with the structure of a compact quantum metric space arising from spectral geometric data.

We let $r \in \B N$ be fixed and the deformation parameter $q$ is a fixed element in the open interval $(0,1)$. 

The \emph{$C^*$-algebraic quantum sphere} $C(S_q^{2r + 1})$ is defined as the universal unital $C^*$-algebra with generators $z_1,\ldots,z_{r + 1}$ subject to the relations here below:
\begin{equation}\label{eq:qsphere}
\begin{split}
  z_i z_j  = q z_j z_i  \q i < j \, \, & , \, \, \, z_i^* z_j = q z_j z_i^* \q i \neq j \\
  [z_i^*,z_i] = (1 - q^2) & \sum_{j = 1}^{i-1} z_j z_j^* \q i > 1 \\
  [z_1^*,z_1] = 0 \, \, & , \, \, \, \sum_{j = 1}^{r + 1} z_j z_j^* = 1  .
\end{split}
\end{equation}
The $C^*$-algebraic quantum sphere $C(S_q^{2r+1})$ is equipped with the strongly continuous action $\si$ of the unit circle $S^1 \su \B C$ defined by putting $\si_\la(z_i) := \la z_i$ for all $i \in \{1,\ldots,r+1\}$. % It clearly holds that the coordinate algebra $\C O(S_q^{2N-1})$ is invariant under the above circle action.

For each $n \in \B Z$, apply the notation $A_n \su C(S_q^{2r+1})$ for the corresponding spectral subspace associated with our circle action so that  
\[
A_n = \big\{ a \in C(S_q^{2r+1}) \mid \si_\la(a) = \la^n a \, \, \T{for all } \la \in S^1 \big\}  .
\]
Moreover, we put $A := C(S_q^{2r+1})$.

The \emph{$C^*$-algebraic quantum projective space} $C(\B CP_q^r)$ is defined as the fixed point algebra of the circle action $\si$ meaning that
\[
C(\B CP_q^r) := A_0 = \big\{ a \in C(S_q^{2r+1}) \mid \si_\la(a) = a \, \, \T{for all } \la \in S^1 \big\}  .
\]
It can be verified that $C(\B CP_q^r)$ agrees with the smallest unital $C^*$-subalgebra of $C(S_q^{2r+1})$ containing the elements $z_i z_j^*$ for all $i,j \in \{1,\ldots,r+1\}$.
\medskip

We shall now see how this specific setting fits with the algebraic framework from Section \ref{s:horizontal}. Our norm-dense unital $*$-subalgebra $\C A \su A$ is defined as the smallest unital $*$-subalgebra containing $z_i$ for all $i \in \{1,\ldots,r+1\}$. In other words, it holds that $\C A$ agrees with the \emph{coordinate algebra for the quantum sphere} which is often denoted by $\C O(S_q^{2r+1})$. Since $z_i \in \C A \cap A_1$ for all $i \in \{1,\ldots,r+1\}$ we obtain by definition that the norm-dense unital $*$-subalgebra $\C O(S_q^{2r+1}) \su C(S_q^{2r+1})$ satisfies condition $(1)$ in Assumption \ref{a:specsub}. The fact that our data also satisfies condition $(2)$ in Assumption \ref{a:specsub} follows by noting that the defining relations in \eqref{eq:qsphere} entail that
\begin{equation}\label{eq:framesphere}
\sum_{j = 1}^{r+1} q^{2(r+1-j)} z_j^* z_j = 1 = \sum_{j = 1}^{r+1} z_j z_j^*.
\end{equation}

The \emph{coordinate algebra for quantum projective space} $\C O(\B CP_q^r)$ is the smallest unital $*$-subalgebra of $C(S_q^{2r+1})$ satisfying that $z_i z_j^* \in \C O(\B CP_q^r)$ for all $i,j \in \{1,\ldots,r+1\}$. We record that $\C O(\B CP_q^r)$ agrees with the intersection $\C O(S_q^{2r+1}) \cap C(\B CP_q^r)$ and we are therefore justified in using the notation $\C A_0 = \C O(\B CP_q^r)$.
\medskip

We are also interested in the $q$-deformed version of the special unitary group $C(SU_q(r + 1))$ which we refer to as the \emph{$C^*$-algebraic quantum special unitary group}. This compact quantum group was introduced by Stanis\l aw Woronowicz in \cite{Wor:TGN,Wor:CMP,Wor:TKD} and it contains the $C^*$-algebraic quantum sphere $C(S_q^{2r+1})$ as a unital $C^*$-subalgebra. Using the description of $C(SU_q(r + 1))$ from \cite[Chapter 9.2]{KlSc:QGR} the relevant map is given by $z_i \mapsto u_{r+1,i}$ for all $i \in \{1,\ldots,r+1\}$. The injectivity of this unital $*$-homomorphism is described in \cite{VaSo:AFQ,Sh:QSG,HoSz:QSP} and a detailed argument can also be found in the review \cite[Section 3.3]{Da:QSG} following the scheme of \cite{MiKa:HSV}.

The \emph{coordinate algebra for the quantum special unitary group} $\C O(SU_q(r+1))$ which agrees with the smallest unital $*$-subalgebra of $C(SU_q(r+1))$ such that $u_{ij} \in \C O(SU_q(r+1))$ for all $i,j \in \{1,\ldots,r+1\}$, see \cite[Chapter 9.2]{KlSc:QGR}.

The coordinate algebra $\C O(SU_q(r+1))$ is deeply intertwined with the \emph{quantized enveloping algebra} $\C U_q(\G{su}(r+1))$ which was introduced in \cite{Jim:QYB,Dri:QG}. It is important to realize that the conventions applied in the present text appear to be less standard and we therefore clarify that $\C U_q(\G{su}(r+1))$ is a Hopf $*$-algebra generated by elements $E_i$, $K_i$ and $K_i^{-1}$, for $i \in \{1,\ldots,r\}$, which are subject to the relations:
\[
\begin{split}
& K_i K_j = K_j K_i \, \, , \, \, \, K_i K_i^{-1} = 1 = K_i^{-1} K_i \, \, , \, \, \, K_i = K_i^* \\
& K_i E_j = q^{\de_{ij} - \frac{1}{2} \de_{i,j-1} - \frac{1}{2} \de_{i,j+1}} E_j K_i \\
%& K_i E_j = q^{-1/2} E_j K_i \quad |i -j| = 1 \\
%& K_i E_j = E_j K_i \, \, , \, \, \, [E_i,E_j] = 0 \quad |i -j| > 1 \\
  & [E_i,E_j^*] = \de_{ij} \cd \frac{K_i^2 - K_i^{-2}}{q - q^{-1}} \\
  & [E_i,E_j] = 0 \quad |i-j| > 1 \\
& E_i^2 E_j - (q + q^{-1}) E_i E_j E_i +  E_j E_i^2 = 0 \quad |i - j | = 1 .
\end{split}
\]
We put $F_i := E_i^*$. The coproduct $\De : \C U_q(\G{su}(r+1)) \to \C U_q(\G{su}(r+1)) \ot \C U_q(\G{su}(r+1))$, the antipode $S : \C U_q(\G{su}(r+1)) \to \C U_q(\G{su}(r+1))$ and the counit $\epsilon : \C U_q(\G{su}(r+1)) \to \B C$ are determined by the formulae:
\[
\begin{split}
& \De(K_i) := K_i \ot K_i \, \, , \, \,  \, \De(E_i) := E_i \ot K_i + K_i^{-1} \ot E_i \\
%\, \, ,  \, \, \, \De(F_i) := F_i \ot K_i + K_i^{-1} \ot F_i \\
& \epsilon(K_i) := 1 \, \,  , \, \, \, \epsilon(E_i) := 0 \\ %\, \, , \, \, \, \epsilon(F_i) := 0 \\ 
& S(K_i) := K_i^{-1} \, \, , \, \, \, S(E_i) := -q E_i . %\\ %\, \, , \, \, \, S(F_i) := -q^{-1} F_i
\end{split}
\]
Remark that $S(F_i) = S(E_i^*) = S^{-1}(E_i)^* = - q^{-1} F_i$. For more details we refer the reader to \cite[Chapter 6.1.2]{KlSc:QGR} -- notice here that Klimyk and Schm\"udgen apply the notation $\breve{U}_q(\T{su}_{r+1})$ for the Hopf $*$-algebra we are denoting by $\C U_q(\G{su}(r+1))$.  

\subsection{Spectral metrics on quantum projective space}\label{ss:specmet}
%Throughout this subsection we let $r \in \B N$ be fixed and the deformation parameter $q$ belongs to the open interval $(0,1)$. %{\blu comparing with earlier $N = r + 1$.}
%
We are now going to review the construction, due to Francesco D'Andrea and Ludwik D\k{a}browski, of a unital spectral triple on quantum projective space, \cite{DaDa:DOQ}. The work of D'Andrea and D\k{a}browski builds on earlier work including Gianni Landi as a coauthor, \cite{DDL:NGQ}. In this earlier paper, the authors treat the important case where $r = 2$. In the special case where $r = 1$, the constructions of D'Andrea and D\k{a}browski recover the unital spectral triple over the Podle\'s sphere which was introduced by Andrzej Sitarz and Ludwik D\k{a}browski in \cite{DaSi:DSP}, see also \cite{NeTu:LIQ}. The unital spectral triple under review is related to the work of several other authors, see e.g. \cite{KrTu:DDO,Mat:DDO,DBS:DDS}, but we believe that the precise relationship still has to be clarified.

In the present text, a lot of details are omitted and we refer the reader to \cite{DaDa:DOQ} and \cite{MiKa:SMQ} for more information.
\medskip

The notation $\La(\B C^r) = \op_{k = 0}^r \La^k(\B C^r)$ refers to the exterior algebra which we view as a finite dimensional Hilbert space with orthonormal basis $\{ e_I \}_{I \su \{1,\ldots,r\}}$. The factors in the direct sum decomposition of the exterior algebra are mutually orthogonal and for $k \in \{0,1,\ldots,r\}$, the corresponding exterior power has the orthonormal basis $\{ e_I\}_{I \su \{1,\ldots,r\} \, , \, \, |I| = k}$ indexed by subsets of $\{1,\ldots,r\}$ with exactly $k$ elements. 

For each $j \in \{1,\ldots,r\}$, define the \emph{$q$-deformed exterior multiplication operator}
\[
\ep^q_j : \La(\B C^r) \to \La(\B C^r) \q
\ep^q_j(e_I) = \fork{ccc}{0 & \T{for} & j \in I \\
  (-q)^{- |I \cap \{1,\ldots,j\}| } \cd e_{I \cup \{j\}} & \T{for} & j \notin I} .
\]
Notice that $\ep^q_j$ maps $\La^k(\B C^r)$ to $\La^{k+1}(\B C^r)$ for all $k \in \{0,1,\ldots,r-1\}$. We define the \emph{$q$-deformed interior multiplication operator} as the adjoint $(\ep^q_j)^*$ of the $q$-deformed exterior multiplication operator. The $q$-deformed exterior multiplication operators satisfy the commutation rule
\[
\ep^q_i \ep^q_j = -q \cd \ep_j^q \ep_i^q \q \T{for } 1 \leq i < j \leq r . %\T{{\blu double-checked on May 27}}
\]
%as soon as $1 \leq i < j \leq r$. 

The constructions of D'Andrea and D\k{a}browski depend on a particular left action of the quantized enveloping algebra $\C U_q(\G{su}(r + 1))$ on the coordinate algebra $\C O(SU_q(r + 1))$. For an element $\eta \in \C U_q(\G{su}(r + 1))$ we denote the corresponding endomorphism of the coordinate algebra by
\[
d_\eta : \C O(SU_q(r + 1)) \to \C O(SU_q(r + 1)) .
\]
Using Sweedler notation $\De(\eta) = \eta_{(1)} \ot \eta_{(2)}$ for coproducts, we record the behaviour
\begin{equation}\label{eq:derprod} 
d_\eta(xy) = d_{\eta_{(2)}}(x) d_{\eta_{(1)}}(y) \, \, \T{ and } \, \, \, d_\eta(x^*) = d_{S(\eta^*)}(x)^*
\end{equation}
for all $x,y \in \C O(SU_q(r+1))$. The left action is then determined by the formulae
\begin{equation}\label{eq:dergen}
d_{K_s}(u_{ij}) = q^{\frac{1}{2} (\de_{is} - \de_{i,s+1})} \cd u_{ij} \, \, , \, \, \,  d_{E_s}(u_{s+1,j}) = -q^{-1} \cd u_{sj}
\, \, , \, \, \, d_{F_s}(u_{sj}) = - q \cd u_{s+1,j} ,
\end{equation}
where we are only listing the non-trivial results and the indices are subject to the constraints $s \in \{1,\ldots,r\}$ and $i,j \in \{1,\ldots,r+1\}$.

It is relevant to clarify that, upon putting $\C A_n := \C O(S_q^{2r+1}) \cap A_n$ for all $n \in \B Z$, we get the description
\begin{equation}\label{eq:specsubK}
\C A_n = \big\{ a \in \C O(S_q^{2r + 1}) \mid d_{K_r}(a) = q^{-\frac{n}{2}} a \big\} .
\end{equation}
Furthermore, for $r \neq 1$ and $s \in \{1,\ldots,r-1\}$, we notice the relations
\[
d_{K_s}(x) = x \, \, \T{ and } \, \, \, d_{E_s}(x) = 0 = d_{F_s}(x) \, \, \T{for all } x \in \C O(S_q^{2r + 1}) . 
\]
%In particular, we notice the formulae
%\[
%d_{K_r}(z_j) = q^{-\frac{1}{2}} z_j \, \, \T{ and } \, \, \, d_{K_r}(z_j^*) = q^{\frac{1}{2}} z_j^*
%\]
%for all $j \in \{1,\ldots,r+1\}$.

For each $j \in \{1,\ldots,r\}$, define the invertible element $L_j := K_j K_{j+1} \clc K_r \in \C U_q(\G{su}(r+1))$. Moreover, we introduce the element $M_j \in \C U_q(\G{su}(r + 1))$ recursively by putting $M_r := E_r$ and $M_j := E_j M_{j+1} - q^{-1} M_{j+1} E_j$ for $j < r$. These elements yield the following endomorphisms of the coordinate algebra
\begin{equation}\label{eq:twistjay}
d_j := d_{L_j M_j^*} \, \, \T{ and } \, \, \, d_j^\da := d_{M_jL_j} : \C O(SU_q(r+1)) \to \C O(SU_q(r + 1)) ,
\end{equation}
which can be combined with the $q$-deformed exterior and interior multiplication operator to yield the endomorphisms 
\begin{equation}\label{eq:overpar}
\ov{\pa} := \sum_{j = 1}^r d_j \ot \ep^q_j \, \, \T{ and } \, \, \, \ov{\pa}^\da := \sum_{j = 1}^r d_j^\da \ot (\ep^q_j)^*
\end{equation}
acting on the algebraic tensor product $\C O(SU_q(r + 1)) \ot \La(\B C^r)$. These two maps satisfy the important identities $\ov{\pa}^2 = 0 = (\ov{\pa}^\da)^2$, see \cite[Proposition 5.6]{DaDa:DOQ}.
%
%For a subset $I \su \{1,\ldots,m\}$ and an element $s \in \{1,\ldots,m\}$ we apply the notation $\de_{s,I}$ for the {\blu Kronecker delta} which is $1$ if $s \in I$ and $0$ otherwise.

We are going to upgrade the endomorphisms $\ov{\pa}$ and $\ov{\pa}^\da$ to unbounded operators acting on a Hilbert space. To this end, consider the Haar state $h : C(SU_q(r + 1)) \to \B C$ together with the separable Hilbert space $L^2(SU_q(r + 1))$ obtained by applying the GNS-construction to the Haar state. Since the Haar state is faithful, see \cite{Nag:HMQ}, we know that the associated map $\La : C(SU_q(r + 1)) \to L^2(SU_q(r + 1))$ is injective. We define $H := L^2(SU_q(r + 1)) \ot \La(\B C^r)$ as the Hilbert space tensor product and equip $H$ with the left action of the $C^*$-algebraic quantum sphere given by the injective unital $*$-homomorphism $\phi : C(S_q^{2r+1}) \to \B L(H)$ satisfying that
\[
\phi(x)( \La(y) \ot \xi) := \La(x \cd y) \ot \xi \q \T{for all } y \in C(SU_q(r + 1)) \, \, \T{ and } \, \, \, \xi \in \La(\B C^r).
\]
We moreover equip $H$ with the selfadjoint unitary operator $\ga_0$ which for every $k \in \{0,1,\ldots,r\}$ satisfies that
\[
\ga_0( \La(y) \ot \xi) = (-1)^k \cd \La(y) \ot \xi \q \T{for all } y \in C(SU_q(r + 1)) \, \, \T{ and } \, \, \, \xi \in \La^k(\B C^r) .
\]

For each $M \in \B Z$, one may introduce the \emph{twisted antiholomorphic forms} $\Om_M \su \C O(SU_q(r+1)) \ot \La(\B C^r)$ by applying the representation theory for the quantized enveloping algebra $\C U_q(\G{su}(r + 1))$ (or more precisely one of the quantum Levi subalgebras). Since the particular definition of this subspace is less relevant for the present text we refer the reader to \cite{DaDa:DOQ} or \cite{MiKa:SMQ} for more details. For the time being, it suffices to collect some properties in the next proposition, see \cite[Proposition 5.6]{DaDa:DOQ} and \cite[Proposition 3.4]{MiKa:SMQ}:

\begin{prop}\label{p:respect}
  Let $M \in \B Z$. For every $\xi, \eta \in \Om_M$ it holds that
  \begin{enumerate}
  \item $\ga_0(\xi) \in \Om_M$ and $\phi(x)(\xi) \in \Om_M$ for all $x \in \C O(\B CP_q^r)$;
  \item $\ov{\pa}(\xi)$ and $\ov{\pa}^\da(\eta)$ belong to $\Om_M$ and we have the identity $\inn{\ov{\pa} \xi, \eta} = \inn{\xi, \ov{\pa}^\da \eta}$. 
\end{enumerate} 
\end{prop}

Let us fix the integer $M \in \B Z$ and let $L^2(\Om_M,h) \su H$ denote the closure of the twisted antiholomorphic forms $\Om_M$ viewed as a subspace of the Hilbert space $H = L^2(SU_q(r + 1)) \ot \La(\B C^r)$.

Because of Proposition \ref{p:respect}, we may define the symmetric unbounded operator
\[
\C D_q := \ov{\pa} + \ov{\pa}^{\da} : \Om_M \to L^2(\Om_M,h), 
\]
which is referred to as the \emph{twisted $q$-deformed Dolbeault operator}. The closure of this unbounded operator is denoted by $D_q := \ov{\C D_q}$. Moreover, we obtain a unital $*$-homomorphism $\rho : C(\B CP_q^r) \to \B L\big( L^2(\Om_M,h) \big)$ satisfying that $\rho(x)(\xi) = \phi(x)(\xi)$ for all $x \in C(\B CP_q^r)$ and $\xi \in L^2(\Om_M,h)$. It can be verified that $\rho$ is injective, see \cite{MiKa:SMQ} for a proof of this fact.

The theorem here below form an essential part of \cite[Theorem 6.2]{DaDa:DOQ}, but D'Andrea and D\k{a}browski also establish the equivariance properties and the reality condition. They moreover show that the spectral dimension is equal to zero (meaning that $(1 + D_q^2)^{-t}$ is of trace class for all $t > 0$).  

\begin{thm}\label{t:spectrip}
The triple $(\C O(\B CP_q^r), L^2(\Om_M,h), D_q)$ is a graded unital spectral triple with respect to the injective unital $*$-homomorphism $\rho : C(\B CP_q^r) \to \B L( L^2(\Om_M,h))$ and the grading operator $\ga_0 : L^2(\Om_M,h) \to L^2(\Om_M,h)$.
\end{thm}

As explained in Subsection \ref{ss:twistedlip} (after Definition \ref{l:sliplower}) we may replace the coordinate algebra $\C O(\B CP_q^r)$ with the substantially larger Lipschitz algebra which we denote by $\Lip_{D_q}(\B CP_q^r)$. The following theorem is due to Max Holst Mikkelsen and the author of the present paper, see \cite[Theorem 5.11]{MiKa:SMQ}. The notion of a spectral metric space is recalled in Definition \ref{d:spemet}. 

\begin{thm}\label{t:quaprojcqms}
The graded unital spectral triple $\big( \T{Lip}_{D_q}(\B CP_q^r), L^2(\Om_M,h), D_q \big)$ is a spectral metric space.
\end{thm}

The corollary here below now follows immediately from Theorem \ref{t:quaprojcqms} and Theorem \ref{t:charac}.

\begin{cor}
The graded unital spectral triple $\big(\C O(\B CP_q^r), L^2(\Om_M,h), D_q\big)$ is a spectral metric space. 
\end{cor}

\subsection{Noncommutative metrics on the quantum spheres}
We are now going to show that Assumption \ref{a:exponential} is satisfied in the context of the higher Vaksman-Soibelman quantum spheres. This leads us straight to the last main result of this paper where we show that each of the quantum spheres can be endowed with the structure of a compact quantum metric space arising from noncommutative geometric data. In particular, we generalize one of the main results from \cite{KaKy:SUq2} regarding quantum metrics on quantum $SU(2)$. More precisely, our methods yield a new proof of \cite[Theorem 5.23]{KaKy:SUq2} in the special case where the parameter $t$ (appearing in the statement of \cite[Theorem 5.23]{KaKy:SUq2}) is equal to one.

It is relevant to clarify that our approach to the spectral geometry of the higher quantum spheres is different from the approach in \cite{ChPa:CES} for a multitude of reasons. First of all, the approach in \cite{ChPa:CES} does not have much to do with the pairing between the quantized enveloping algebra and the coordinate algebra for quantum $SU(r + 1)$ in the sense that it seems difficult to describe the selfadjoint unbounded operators analyzed in \cite{ChPa:CES} in terms of this pairing. Another important difference is that in \cite{ChPa:CES} the authors construct equivariant spectral triples over the higher quantum spheres. As explained in the introduction, our faithfulness to the underlying $q$-geometry, as witnessed by the above mentioned pairing between the quantized enveloping algebra and the coordinate algebra for quantum $SU(r + 1)$, prevents us from obtaining unital spectral triples for the higher quantum spheres. Instead, we are forced to work with a different kind of spectral geometric data resulting in two (twisted) $*$-derivations instead of a single one. These two (twisted) $*$-derivations are in some sense incompatible because they are twisted with two different twists (where one of the twists is in fact trivial). This has to do with the way the corresponding selfadjoint unbounded operators interact with the coordinate algebra for the quantum sphere. The non-twisted $*$-derivation is related to the vertical geometric data as witnessed by the circle action and the twisted $*$-derivation is related to the horizontal geometric data stemming from the unital spectral triple over quantum projective space as reviewed in Theorem \ref{t:spectrip}. %Nonetheless our two (twisted) $*$-derivations can be combined into a single seminorm which 

Let us provide the relevant details. Our starting point is the $C^*$-algebraic quantum sphere $A := C(S_q^{2r + 1})$ which is equipped with the strongly continuous circle action $\si$ satisfying that $\si_\la(z_i) = \la z_i$ for all $i \in \{1,\ldots,r+1\}$. We are moreover considering the coordinate algebra for the quantum sphere, $\C A := \C O(S_q^{2r+1})$, which is a norm-dense unital $*$-subalgebra of $C(S_q^{2r+1})$. On top of this data, we fix an $M \in \B Z$ and consider the graded unital spectral triple $\big( \C O(\B CP_q^r), L^2(\Om_M,h), D_q\big)$ described in Subsection \ref{ss:specmet}.

%The relevant norm-dense subspace $\C A_1 \su A_1$ is defined as the intersection $\C A_1 := \C O(S_q^{2r+1}) \cap A_1$ and the unital $*$-subalgebra generated by this subspace agrees with the coordinate algebra for the quantum sphere $\C O(S_q^{2r+1})$.

Let us put $H_0 := L^2(\Om_M,h)$. The injective unital $*$-homomorphism and the $*$-derivation associated with our graded unital spectral triple are denoted by $\rho : C(\B CP_q^r) \to \B L(H_0)$ and $\de_0 : \C O(\B CP_q^r) \to \B L(H_0) $. We view $H_0$ as a closed subspace of the separable Hilbert space $H = L^2(SU_q(r + 1)) \ot \La(\B C^r)$. In Subsection \ref{ss:specmet} we already introduced the injective unital $*$-homomorphism $\phi : C(S_q^{2r+1}) \to \B L(H)$ satisfying that $\phi(x)\xi = \rho(x) \xi$ for all $x \in C(\B CP_q^r)$ and $\xi \in H_0$. Our aim is now to construct the twisted derivation $\de : \C O(S_q^{2r+1}) \to \B L(H)$ which we need to establish the conditions in Assumption \ref{a:exponential}. As we shall see, the constant $\mu > 0$ agrees with the inverse of the deformation parameter so that $\mu = q^{-1}$.

Let us define the algebra automorphism
\[
\be_i := d_{K_r^{-1}} : \C O(S_q^{2r+1}) \to \C O(S_q^{2r+1})
\]
and record that $\be_i(x) = q^{\frac{n}{2}} x$ as soon as $x \in \C A_n := \C O(S_q^{2r+1}) \cap A_n$ for some $n \in \B Z$. 

Recall from \eqref{eq:twistjay} that $d_j$ and $d_j^\da$ are endomorphisms of $\C O(SU_q(r+1))$ defined by putting $d_j := d_{L_j M_j^*}$ and $d_j^\da := d_{M_j L_j}$ for all $j \in \{1,\ldots,r\}$. Regarding the elements $M_j L_j$ and $L_j M_j^*$ in $\C U_q(\G{su}(r+1))$, it follows from \cite[Lemma 3.14]{DaDa:DOQ} that they operate as twisted derivations in the sense of the next lemma: 

\begin{lemma}\label{l:twistsphere}
  Let $j \in \{1,\ldots,r\}$. We have the formulae
  \[
  d_j(xy) = \be_i^{-2}(x) d_j(y) + d_j(x) y \, \, \mbox{ and } \, \, \, 
  d_j^\da(xy) = \be_i^{-2}(x) d_j^\da(y) + d_j^\da(x) y
  \]
  for all $x \in \C O(S_q^{2r+1})$ and $y \in \C O(SU_q(r+1))$.
  %\[
  % d_j^\da(xy) = d_{L_j}( d_{L_j}(x) d_{M_j^*}(y) + d_{M_j^*}(x) d_{L_j^{-1}}(y) )
  % d_j(xy) = d_{M_j}( d_{L_j}(x) d_{L_j}(y) ) 
  %\[
  %d_{M_j}(x y) = d_{L_j}(x) d_{M_j}(y) + d_{M_j}(x) d_{L_j^{-1}}(y) \, \, \mbox{and} \, \, \,
  %d_{M_j^*}(xy) = d_{L_j}(x) d_{M_j^*}(y) + d_{M_j^*}(x) d_{L_j^{-1}}(y)
  %\]
\end{lemma}
\begin{proof}
  Let $x \in \C O(S_q^{2r+1})$ and $y \in \C O(SU_q(r+1))$ be given. In order to compute how $d_j$ and $d_j^\da$ act on the product $x \cd y$, it suffices to have a formula for the coproduct $\De(M_j)$ and this is provided in the proof of \cite[Lemma 3.14]{DaDa:DOQ}. Observe now that the endomorphisms $d_{E_k}$ and $d_{F_k}$ both vanish on $\C O(S_q^{2r+1})$ as soon as $1 \leq k \leq r-1$. We therefore obtain from \cite[Lemma 3.14]{DaDa:DOQ} that
  \[
  \begin{split}
  d_{M_j}(x y) & = d_{L_j}(x) d_{M_j}(y) + d_{M_j}(x) d_{L_j^{-1}}(y) \q \T{and} \\
  d_{M_j^*}(xy) & = d_{L_j}(x) d_{M_j^*}(y) + d_{M_j^*}(x) d_{L_j^{-1}}(y) .
  \end{split}
  \]
The result of the present lemma now follows by observing that $d_{L_j}$ is an automorphism of $\C O(SU_q(r+1))$ (with inverse $d_{L_j^{-1}}$) satisfying that $d_{L_j}(x) = \be_i^{-1}(x)$.
\end{proof}

Let us define the $\cc$-linear map $\de : \C O(S_q^{2r+1}) \to \B L(H)$ by the formula
\begin{equation}\label{eq:twistsphere}
\de(x) := \sum_{j = 1}^r \big( d_j(x) \ot \ep_j^q  + d_j^\da(x) \ot (\ep_i^q)^* \big) 
\end{equation}
where the factors $d_j(x)$ and $d_j^\da(x)$ in $\C O(SU_q(r+1))$ act via left multiplication on $L^2(SU_q(r + 1))$. 

Recall from Section \ref{s:quasph} that the elements $\ze_i^L := q^{r + 1 - i} z_i$ and $\ze_i^R := z_i$ for $i \in \{1,\ldots,r+1\}$ satisfy condition $(2)$ from Assumption \ref{a:specsub}, see \eqref{eq:framesphere}. We also know that the coordinate algebra $\C O(S_q^{2r+1})$ is generated as a $*$-algebra by the subspace $\C A_1 := \C O(S_q^{2r+1}) \cap A_1$.

\begin{prop}\label{p:expsphere}
The Hilbert space $H = L^2(SU_q(r + 1)) \ot \La(\B C^r)$, the unital $*$-homomorphism $\phi : C(S_q^{2r+1}) \to \B L(H)$, and the $\cc$-linear map $\de : \C O(S_q^{2r+1}) \to \B L(H)$ satisfy the conditions of Assumption \ref{a:exponential} with respect to the constant $\mu := q^{-1}$. 
\end{prop}
\begin{proof}
We start out with the first condition in Assumption \ref{a:exponential}. It follows by the definition of $\rho : C(\B CP_q^r) \to \B L(H_0)$ that $\rho(x)(\eta) = \phi(x)(\eta)$ for all $x \in C(\B CP_q^r)$ and $\eta \in H_0 = L^2(\Om_M,h)$, see the discussion after Proposition \ref{p:respect}. Consider now an element $x \in \C O(\B CP_q^r)$. For each $y \in \C O(SU_q(r + 1))$ and $\xi \in \La(\B C^r)$ we obtain from Lemma \ref{l:twistsphere} and the definition of the endomorphism $\ov{\pa} + \ov{\pa}^\da$ from \eqref{eq:overpar} that
\begin{equation}\label{eq:compar}
  \begin{split}
    & \big[ \ov{\pa} + \ov{\pa}^\da, \phi(x) \big] ( \La(y) \ot \xi) \\
    & \q = \sum_{j = 1}^r \big( \La( d_j(x y)  - x d_j(y) ) \ot \ep_j^q(\xi) + \La( d_j^\da(x y) - x d_j^\da(y) ) \ot (\ep_j^q)^*(\xi) \big) \\
    & \q = \de(x) ( \La(y) \ot \xi) .
    \end{split}
    \end{equation}
    Since the abstract Dirac operator $D_q : \T{Dom}(D_q) \to L^2(\Om_M,h)$ satisfies that $D_q(\eta) = (\ov{\pa} + \ov{\pa}^\da)(\eta)$ for all $\eta \in \Om_M$ we obtain from \eqref{eq:compar} that
    \[
\de_0(x)(\eta) = [\ov{\pa} + \ov{\pa}^\da, \rho(x)](\eta) = \de(x)(\eta) .
\]
Since $\Om_M \su L^2(\Om_M,h)$ is norm-dense we conclude that the first condition in Assumption \ref{a:exponential} is satisfied by our data.

  The second condition in Assumption \ref{a:exponential} follows immediately from Lemma \ref{l:twistsphere} by observing that $\be_i^{-2}(x) = d_{K_r^2}(x) = q^{-n} x$ whenever $x \in \C A_n = \C O(S_q^{2r+1}) \cap A_n$ for some $n \in \B Z$.

  The third condition in Assumption \ref{a:exponential} can be verified by noting that $d_{F_i}(z_j) = 0 = d_{E_i}(z_j^*)$ for all $i \in \{1,\ldots,r\}$ and $j \in \{1,\ldots,r + 1\}$, see \eqref{eq:derprod} and \eqref{eq:dergen}. Indeed, we obtain from this that $d_i(z_j) = 0 = d_i^\da(z_j^*)$ and hence, upon applying the twisted Leibniz rule from Lemma \ref{l:twistsphere}, we get that 
  \[
  %\begin{split}
    \sum_{j = 1}^{r + 1} \phi(z_j) \de(z_j^*)
   = \sum_{j = 1}^{r+1} \sum_{i = 1}^r z_j d_i(z_j^*) \ot \ep_i^q 
  = \sum_{j = 1}^{r+1} \sum_{i = 1}^r q \cd d_i(z_j z_j^*) \ot \ep_i^q = 0 .
  % \end{split}
  \]
  A similar computation shows that the sum $\sum_{j = 1}^{r + 1} q^{2(r + 1 - j)} \phi(z_j^*) \de(z_j)$ is equal to zero as well. 
 \end{proof}

Since we now know that Assumption \ref{a:exponential} is satisfied in the context of the higher quantum spheres we finish this paper by restating our main theorems in this particular case.

Before doing so, let us however spend a bit more time computing the ingredients in the graded twisted Lipschitz triple $\big( \C O(S_q^{2r+1}), G, D_\Ga \big)$ which we obtain by applying Theorem \ref{t:twistlip} in the present setting.

The Hilbert space $G = X \hot_\rho L^2(\Om_M,h)$ is identified with the Hilbert space direct sum $K := \widehat{\bop}_{n \in \B Z} H_n$ where each summand $H_n \su H$ identifies with the closure of $\phi(A_n) L^2(\Om_M,h)$ inside $H = L^2(SU_q(r+1)) \ot \La(\B C^r)$. It therefore follows from \cite[Section 5]{DaDa:DOQ} that $H_n = L^2(\Om_{M + n},h)$ and that the norm-dense subspace $\phi(\C A_n) \Om_M$ agrees with $\Om_{M + n} \su L^2(\Om_{M + n},h)$. In this picture, the grading operator $1 \hot \ga_0$ is simply given by the selfadjoint unitary operator $\ga_0$ (see Proposition \ref{p:respect}) on each of the summands $L^2(\Om_{M + n},h) \su K$ for $n \in \B Z$. 

We specify that our one-parameter group of unitaries is determined by the formula $V_s(q^{-1})(\eta) = q^{-is \frac{n}{2}} \cd \eta$ whenever $\eta$ belongs to $H_n$ for some $n \in \B Z$ and $s \in \B R$. For every $z \in \B C$, the corresponding algebra automorphism $\be_z : \C O(S_q^{2r+1}) \to \C O(S_q^{2r+1})$ satisfies that $\be_z(a) = q^{-iz \frac{n}{2}} a$ whenever $a \in \C A_n$, see \eqref{eq:beta}. %Recall the definition from \eqref{eq:overpar} of the endomorphisms $\ov{\pa}$ and $\ov{\pa}^\da$ acting on the algebraic tensor product $\C O(SU_q(r+1)) \ot \La(\B C^r)$.

\begin{lemma}
  The modular lift $D_\Ga : \T{Dom}(D_\Ga) \to K$ has the algebraic direct sum $\bop_{n \in \B Z} \Om_{M + n}$ as a core and for each $n \in \B Z$ and $\xi \in \Om_{M + n}$ we have the explicit expression $D_\Ga(\xi) = (\ov{\pa} + \ov{\pa}^\da)(\xi)$. %Moreover, the associated twisted $*$-derivation $\de_\Ga : \C O(S_q^{2r+1}) \to \B L(K)$ is given by the formula
%  \[
%\de_\Ga(a) = q^{n/2} \sum_{j = 1}^r \big( d_j(a) \ot \ep_j^q + d_j^\da(a) \ot (\ep_j^q)^* \big)
%  \]
\end{lemma}
\begin{proof}
  The fact that $\bop_{n \in \B Z} \Om_{M + n} \su \bop L^2(\Om_{M+n},h)$ is a core for the modular lift follows from Lemma \ref{l:modular} by recalling that $\Om_M \su L^2(\Om_M,h)$ is a core for the abstract Dirac operator $D_q$ (see the discussion after Proposition \ref{p:respect}) and that the algebraic direct sum $\bop_{n \in \B Z} \phi(\C A_n) \T{Dom}(D_q)$ is a core for the modular lift (see Definition \ref{d:modular}). Regarding the explicit description for the modular lift, we let $n \in \B Z$ and $\xi \in \Om_{M + n}$ be given. Assume without loss of generality that $\xi$ has the form $\phi(a) \eta$ for some $a \in \C A_n$ and $\eta \in \Om_M \su \C O(SU_q(r+1)) \ot \La(\B C^r)$. Let us choose $x_1,\ldots,x_m \in \C O(SU_q(r+1))$ and $\om_1,\ldots,\om_m \in \La(\B C^r)$ such that $\eta = \sum_{i = 1}^m x_i \ot \om_i$. Using Lemma \ref{l:modular} and Lemma \ref{l:twistsphere}, the computation here below ends the proof of the lemma: 
  \[
  \begin{split}
  D_\Ga( \phi(a) \eta) & = \de(a) \eta + q^{-n} \phi(a) (\ov{\pa} + \ov{\pa}^\da)(\eta) \\
  & = \sum_{i = 1}^m \sum_{j = 1}^r ( d_j(a) x_i + q^{-n} a d_j(x_i) ) \ot \ep_j^q \om_i \\
  & \q + \sum_{i = 1}^m \sum_{j = 1}^r ( d_j^\da(a) x_i + q^{-n} a d_j^\da(x_i) ) \ot (\ep_j^q)^* \om_i \\
  & = \sum_{i = 1}^m \sum_{j = 1}^r \big( d_j(a x_i) \ot \ep_j^q \om_i + d_j^\da(a x_i) \ot (\ep_j^q)^* \om_i \big) \\
  & = (\ov{\pa} + \ov{\pa}^\da)( \phi(a) \eta ) . \qedhere
  \end{split}
  \]
\end{proof}

The twisted Lipschitz algebra associated with the twisted Lipschitz triple $(\C O(S_q^{2r+1}),G,D_\Ga)$ is denoted by $\Lip_{\T{hor}}(S_q^{2r+1})$ and referred to as the horizontal Lipschitz algebra. As explained in Subsection \ref{ss:twistedlip}, the horizontal Lipschitz algebra is equipped with the twisted $*$-derivation
\[
\de_{\T{hor}} := \de_\Ga : \Lip_{\T{hor}}(S_q^{2r+1}) \to \B L(G) . 
\]
Remark that Theorem \ref{t:twistlip} provides a computation of $\de_{\T{hor}}$ on the coordinate algebra in terms of the twisted derivation $\de : \C O(S_q^{2r+1}) \to \B L(H)$ which in the current setting is spelled out in \eqref{eq:twistsphere}. The horizontal Lipschitz algebra $\Lip_{\T{hor}}(S_q^{2r+1})$ is however substantially larger than the coordinate algebra $\C O(S_q^{2r+1})$. % (it is not defined by taking a closure of the restriction of $\de_{\T{hor}}$ to the coordinate algebra).

Let us also consider the ungraded Lipschitz triple $\big( \C O(S_q^{2r+1}), G, N \hot 1\big)$ which we obtain from the considerations in Section \ref{s:vertical}. The associated Lipschitz algebra is denoted by $\Lip_{\T{ver}}(S_q^{2r+1})$ and we refer to it as the vertical Lipschitz algebra. The vertical Lipschitz algebra comes equipped with the $*$-derivation
\[
\de_{\T{ver}} := \de_{N \hot 1} : \Lip_{\T{ver}}(S_q^{2r+1}) \to \B L(G) ,
\]
which is given by the expression $\de_{N \hot 1}(x) = n \cd x$ whenever $x \in A_n$ for some $n \in \B Z$.

We equip the total Lipschitz algebra $\Lip(S_q^{2r+1}) := \Lip_{\T{hor}}(S_q^{2r+1}) \cap \Lip_{\T{ver}}(S_q^{2r+1})$ with the $\cc$-linear map
\[
\de_{\T{tot}} := (1 \hot \ga_0)\de_{\T{ver}} + \de_{\T{hor}} : \Lip(S_q^{2r+1}) \to \B L(G) .
\]
The corresponding slip-norm on $\Lip(S_q^{2r+1})$ is denoted by $L_{\T{tot}}$ meaning that $L_{\T{tot}}(x) := \| \de_{\T{tot}}(x) \|_\infty$ for all $x \in \Lip(S_q^{2r+1})$.

The following main theorem is now a consequence of Proposition \ref{p:expsphere}, Theorem \ref{t:quaprojcqms} and Corollary \ref{c:lipcqms}.

\begin{thm}
The pair $\big( \Lip(S_q^{2r+1}), L_{\T{tot}} \big)$ is a compact quantum metric space for the $C^*$-algebraic quantum sphere $C(S_q^{2r+1})$. 
\end{thm}

As usual, this implies the corresponding weaker result regarding the coordinate algebra by observing that $\C O(S_q^{2r+1}) \su \Lip(S_q^{2r+1})$ (see Theorem \ref{t:charac}):

\begin{cor}
The pair $\big( \C O(S_q^{2r+1}), L_{\T{tot}} \big)$ is a compact quantum metric space for the $C^*$-algebraic quantum sphere $C(S_q^{2r+1})$. 
\end{cor}

%\bibliography{JK}
\bibliographystyle{amsalpha-lmp}

\newcommand{\etalchar}[1]{$^{#1}$}
\providecommand{\bysame}{\leavevmode\hbox to3em{\hrulefill}\thinspace}
\providecommand{\MR}{\relax\ifhmode\unskip\space\fi MR }
% \MRhref is called by the amsart/book/proc definition of \MR.
\providecommand{\MRhref}[2]{%
  \href{http://www.ams.org/mathscinet-getitem?mr=#1}{#2}
}
\providecommand{\href}[2]{#2}

\end{document}